\documentclass[12pt,reqno]{amsart}

\usepackage{geometry}
\usepackage{hyperref}
\hypersetup{
    colorlinks=true,
    linkcolor=blue,
    pdfpagemode=FullScreen,
}
\usepackage{amssymb,amsfonts}
\usepackage[all,arc]{xy}
\usepackage{enumerate}
\usepackage{mathrsfs,stmaryrd}
\usepackage{relsize} %Resize mathematical content
\usepackage{enumitem}

\usepackage{pgf, tikz}

\newcommand{\function}[5]{\begin{array}{cccc} #1 : & #2 & \longrightarrow & #3 \\ & #4 & \longmapsto & #5 \end{array}}

\newtheorem{thm}{Theorem}[section]
\newtheorem{cor}[thm]{Corollary}
\newtheorem{prop}[thm]{Proposition}
\newtheorem{lem}[thm]{Lemma}

\theoremstyle{definition}
\newtheorem{defn}[thm]{Definition}
\newtheorem{defns}[thm]{Definitions}

\newtheorem{exmps}[thm]{Examples}
\newtheorem{notn}[thm]{Notation}
\newtheorem{notns}[thm]{Notations}

\theoremstyle{definition}
\newtheorem{rem}[thm]{Remark}

\numberwithin{equation}{section}

\newcommand{\norm}[1]{%
\lnorm #1 \rnorm%
}
\newcommand{\triplenorm}[1]{{\left\vert\kern-0.25ex\left\vert\kern-0.25ex\left\vert #1 
    \right\vert\kern-0.25ex\right\vert\kern-0.25ex\right\vert}}

\newcommand{\hsigball}[1]{%
    B_{#1}^{H^{\sigma}}%
    }                   %la boule de centre 0 et de rayon R de H^sigma

\newcommand{\p}{\partial}
\newcommand{\lnorm}{\left\lVert} %--belle norem à gauche--%
\newcommand{\rnorm}{\right\rVert} %--belle norem à droite--%

\newcommand{\Hgmt}{H^{s-\frac{1}{2}-}(\T)}
\newcommand{\hsig}{H^{\sigma}}
\newcommand{\hsigt}{\hsig(\T)}
\newcommand{\hsignorm}[1]{% 
\lnorm #1 \rnorm_{\hsig}% 
}

\newcommand{\linc}{k_1-k_2+...-k_6=0}
\newcommand{\Omgz}{\Omega(\vec{k}) = 0}
\newcommand{\Omgk}{\Omega(\vec{k})}
\newcommand{\Omgeqkp}{\Omega(\vec{k}) = \kappa}
\newcommand{\Omgnz}{\Omega(\vec{k}) \neq 0}
\newcommand{\lincba}{p_1-p_2+...+p_5 =k_1}
\newcommand{\lincbb}{q_1-q_2+...+q_5 =k_2}

\newcommand{\Ltwo}{\L^2}
\newcommand{\Lsix}{\L^6}
\newcommand{\Lp}{\L^p}

\newcommand{\hsp}{\hspace{0.1cm}}

\newcommand{\cjg}[1]{%
  \overline{#1}%
  }                     %conjugué d'un nombre complexe
  
\newcommand{\tld}{\widetilde}
\renewcommand{\hat}{\widehat}

\newcommand{\lra}{\longrightarrow}
\newcommand{\ra}{\rightarrow}

\newcommand{\rla}{\leftrightarrow}

\newcommand{\tendsto}[1]{%
\underset{#1}{\lra}%
}

\renewcommand{\Re}{\textnormal{Re}}
\renewcommand{\Im}{\textnormal{Im}}

\newcommand{\C}{\mathbb{C}}		% complexes
\newcommand{\E}{\mathbb{E}}		% espérance
\renewcommand{\L}{\mathbb{L}}	% espaces L^p

\newcommand{\N}{\mathbb{N}}		% entiers naturels

\renewcommand{\P}{\mathbb{P}}	% proba
\newcommand{\R}{\mathbb{R}}		% réels
\newcommand{\T}{\mathbb{T}}
\newcommand{\Z}{\mathbb{Z}}		% entiers relatifs

\usepackage{dsfont}
\renewcommand{\1}{\mathds{1}}			% indicatrice

\newcommand{\cA}{\mathcal A}		% tribu
\newcommand{\cB}{\mathcal B}		% tribu
\newcommand{\cC}{\mathcal C}	% tribu
\newcommand{\cD}{\mathcal D}
	
\newcommand{\cI}{\mathcal I}		% idéal
\newcommand{\cN}{\mathcal N}	% loi normale
\newcommand{\cQ}{\mathcal Q}	% O de LANDAU
\newcommand{\cR}{\mathcal R}	% relation
\newcommand{\cS}{\mathcal S}
\newcommand{\cT}{\mathcal T}		% Tribu
\newcommand{\frkR}{\mathfrak{R}}
\newcommand{\frkQ}{\mathfrak{Q}}
\newcommand{\frkT}{\mathfrak{T}}

\newcommand{\eps}{\varepsilon}

\title[Transport of Gaussian measures for NLS]{Transport of low regularity gaussian measures for the 1d quintic nonlinear Schrödinger equation}

\author{Alexis Knezevitch}

\begin{document}

\begin{abstract}

We consider the 1d nonlinear Schrödinger equation (NLS) on the torus with initial data distributed according to the Gaussian measure with covariance operator $(1 - \Delta)^{-s}$, where $\Delta$ is the Laplace operator. We prove that the Gaussian measures are quasi-invariant along the flow of (NLS) for the full range $s > \frac{3}{2}$. This improves a previous result obtained by Planchon, Tzvetkov and Visciglia in \cite{planchon2020transport}, where the quasi-invariance is proven for $s=2k$ for all integers $k\geq 1$. In our approach, to prove the quasi-invariance, we directly establish an explicit formula for the Radon-Nikodym derivative $G_s(t,.)$ of the transported measures, which is obtained as the limit of truncated Radon-Nikodym derivatives $G_{s,N}(t,.)$ for transported measures associated with a
truncated system. We also prove that the Radon-Nikodym derivatives belong to $L^p$, $p>1$, with respect to $H^1(\T)$-cutoff Gaussian measures, relying on the introduction of weighted Gaussian measures produced by a normal form reduction, following Sun-Tzvetkov \cite{sun2023quasi}. Additionally, we prove that the truncated densities $G_{s,N}(t,.)$ converges to $G_s(t,.)$ in $L^p$ (with respect to the $H^1(\T)$-cutoff Gaussian measures).

\end{abstract}

\maketitle

\tableofcontents

\section{Introduction}
In this paper, we contribute to the program initiated by Tzvetkov in \cite{tzvetkov2015quasiinvariant} on the transport of Gaussian measures under the flow of Hamiltonian partial differential equations (PDEs). We consider the defocusing quintic nonlinear Schrödinger equation on the torus :

\begin{equation}\label{NLS}
    \begin{cases}
        i\p_t u + \p_x^2 u = |u|^4u, \hspace{.5cm} (t,x) \in \R \times \T \\
        u|_{t=0} = u_0 
    \end{cases}
\end{equation}
This is a Hamiltonian PDE with the associated Hamiltonian :
\begin{equation}\label{hamiltonian}
    H(u) := \frac{1}{2} \int_{\T} |\p_x u|^2 dx + \frac{1}{6} \int_{\T} |u|^6dx
\end{equation}

\subsection{Description of the problem} In the present work, we consider the situation where \eqref{NLS} is globally well posed in a certain Banach space $X$. With such a Banach space $X$, we can invoke, for every time $t \in \R$, the flow of \eqref{NLS} :
\begin{equation*}
    \Phi(t) : \hsp X \lra X
\end{equation*}
which is the continuous map that for any initial data $u_0 \in X$ associates the solution of \eqref{NLS} evaluated at time $t$. \\
Given a Gaussian measure $\mu$ on X (defined on $\cB(X)$, the $\sigma$-algebra of Borel sets of X), we can consider the push-forward measure of $\mu$ under $\Phi(t)$, denoted by $\Phi(t)_\# \mu$, and defined for all $A \in \cB(X)$ as 
\begin{equation*}
    \Phi(t)_\# \mu(A) := \mu(\Phi(t)^{-1}A)
\end{equation*}
We say that $\Phi(t)_\# \mu$ is the \textit{transported measure} of $\mu$ under the flow $\Phi(t)$. This object is of interest because properties on the measure $\Phi(t)_\# \mu$ provide a macroscopic description of the flow. Following the problem raised by Tzvetkov for Hamitlonian PDEs in \cite{tzvetkov2015quasiinvariant}, one wonders if, for every time $t$, the measure $\Phi(t)_\# \mu$ is absolutely continuous with respect to $\mu$. If that is indeed the case, we use the notation $\Phi(t)_\# \mu \ll \mu$. In other words, one wonders if, for any Borel set $A$ in $X$, the following assertion 
\begin{equation}\label{abs continuity}
    \mu(A) = 0 \implies \Phi(t)_\# \mu(A) = 0
\end{equation}
is true. If the answer is positive, we say that the measure $\mu$ is \textit{quasi-invariant} under the flow $\Phi(t)$. In that case, we can invoke the Radon-Nikodym derivative $F_t \in \L^1(d\mu)$ which satisfies:
\begin{equation*}
    \Phi(t)_\# \mu = F_t(u) d\mu
\end{equation*}
and often denoted as $\frac{d \Phi(t)_\# \mu}{d\mu}$. The absolute continuity \eqref{abs continuity} is only a qualitative result because we only obtain the existence of the Radon-Nikodym derivative $F_t$. A more quantitative result would be providing additional information on $F_t$, such as an explicit formula which should be suitably interpreted. \\

The initial data spaces X under consideration will be Sobolev spaces on the torus. We can define Gaussian measures on Sobolev spaces as follows. For any given $s \in \R$, we define the Gaussian measure $\mu_s$ as the law of the random varibale:
\begin{equation}\label{random series S}
        S : \hsp \omega \longmapsto \sum_{n \in \Z} \frac{g_n(\omega)}{\langle n \rangle^s}e^{inx}
\end{equation}
where $\langle n \rangle := (1+n^2)^{\frac{1}{2}}$ and $\{g_n \}_{n \in \Z}$ are independent standard complex-valued Gaussian measures\footnote{in the sense that $g_n = h_n + i l_n$, where $h_n$ and $l_n$ are two independent real Gaussian measures on $\R$ with law $\cN(0,\frac{1}{2})$} on a probability space $(\Omega,\cA,\P)$. More precisely, for $\sigma \in \R$, we have that
\begin{equation*}
    \E \left[ \sum_{n\in\Z} \langle n \rangle^{2\sigma} \left| \frac{g_n}{\langle n \rangle^{s}} \right|^2 \right] < + \infty \iff \sigma < s - \frac{1}{2}
\end{equation*}
so the random series in \eqref{random series S} converges in $\L^2(\Omega, \hsigt) $ if and only if $\sigma < s-\frac{1}{2}$. Thus, $\mu_s = S_\# \P$ is a probability measure on $\cB(\hsigt)$ for all $\sigma < s- \frac{1}{2}$. For more details on Gaussian measures, we refer to \cite{kuo2006gaussian} (see also \cite{bogachev1998gaussian}).
Furthermore, it is well-known that 
\begin{equation*}
    S \in H^{(s-\frac{1}{2})-}(\T) := \bigcap_{\sigma < s-\frac{1}{2} } \hsigt \hspace{.3cm} \text{almost surely},
\end{equation*}
so the transported measure: 
\begin{equation*}
    \Phi(t)_\# \mu_s = (\Phi(t) \circ S)_\# \P
\end{equation*}
makes sense if the flow $\Phi(t)$ is well defined on $H^{(s-\frac{1}{2})-}(\T)$ almost surely. In the situation where $s > \frac{3}{2}$, we have $ H^{(s-\frac{1}{2})-}(\T) \subset H^1(\T)$. And, at the regularity $H^1(\T)$, equation \eqref{NLS} is globally well posed. It follows from the combination of an elementary local wellposedness (thanks to the algebra property of $H^1(\T)$) and the use of the conservation of the Hamiltonian \eqref{hamiltonian} and of the mass:
\begin{align*}
   H(u) &= \frac{1}{2} \int_{\T} |\p_x u|^2 dx + \frac{1}{6} \int_{\T} |u|^6dx, & M(u) &:= \int_\T |u|^2dx
\end{align*}
In conclusion, the transported measure $\Phi(t)_\# \mu_s$ is well defined whenever $s > \frac{3}{2}$, and we can legitimately wonder if it is absolutely continuous with respect to $\mu_s$. For some $s\leq \frac{3}{2}$, it is still possible to construct  the transported measure $\Phi(t)_\# \mu_s$, but our method in this paper only works for $s>\frac{3}{2}$.

\subsection{Formal computation and main results} Formally, we can see the Gaussian measure $\mu_s$ as the measure 
\begin{equation*}
    \frac{1}{Z_s} e^{-\frac{1}{2}\norm{u}_{H^s}^2} du
\end{equation*}
where $du$ is formally the Lebesgue measure (which does not exist on infinite dimensional vector spaces). Let us compute formally $\Phi(t)_\# \mu_s$ in order to predict what the Radon-Nikodym derivative of $\Phi(t)_\# \mu_s$ with respect to $\mu_s$ could be:
\begin{equation*}
    \Phi(t)_\# \mu_s = \Phi(t)_\# \bigl(\frac{1}{Z_s} e^{-\frac{1}{2}\norm{u}_{H^s}^2} du \bigr) = \frac{1}{Z_s} e^{-\frac{1}{2}\norm{\Phi(t)^{-1}u}_{H^s}^2} \Phi(t)_\#du
\end{equation*}
Since \eqref{NLS} is a Hamiltonian PDE, we may formally write that the (non existent) Lebesgue measure is preserved by the flow $\Phi(t)$, that is $\Phi(t)_\#du = du$. Moreover, from the additivity of the flow, we have $\Phi(t)^{-1}=\Phi(-t)$. Hence,
\begin{equation*}
    \Phi(t)_\# \mu_s = \frac{1}{Z_s} e^{-\frac{1}{2}\norm{\Phi(-t)u}_{H^s}^2} du = e^{-\frac{1}{2}(\norm{\Phi(-t)u}_{H^s}^2- \norm{u}_{H^s}^2)} d\mu_s
\end{equation*}
Then, we expect that the actual density of $\Phi(t)_\# \mu_s$ with respect to $\mu_s$ is:
\begin{equation*}
    G_s(t,u):= e^{-\frac{1}{2}(\norm{\Phi(-t)u}_{H^s}^2- \norm{u}_{H^s}^2)}
\end{equation*}
However, since it is known that $\mu_s(H^{s-\frac{1}{2}}(\T))=0$, we have:
\begin{equation*}
    \norm{\Phi(-t)u}_{H^s}^2 = +\infty \hspace{0.2cm} \text{and} \hspace{0.2cm}  \norm{u}_{H^s}^2 = + \infty,\hspace{0.2cm}  \mu_s-\text{almost surely}
\end{equation*}
so this is not even clear that the density $G_s(t,u)$ is well defined on the support of $\mu_s$. But the hope is to observe some cancellation in the difference between $ \norm{\Phi(-t)u}_{H^s}^2$ and $\norm{u}_{H^s}^2$. In order to analyze this difference, we first consider instead an approximated system for $N \in \N$:
\begin{equation*}
    \begin{cases}
        i\p_t u + \p_x^2 u = \Pi_N \left(|\Pi_Nu|^4\Pi_Nu \right), \hspace{0.2cm} (t,x) \in \R \times \T \\
        u|_{t=0} = u_0 
    \end{cases}
\end{equation*}
where $\Pi_N$ is the Dirichlet projector. Denoting by $\Phi_N(t)$ its flow (called the \textit{truncated flow}), we will be able, thanks to a finite-dimensional-type computation in Section~\ref{section Transport of Gaussian measures under the truncated flow}, to prove rigorously that the transported measure $\Phi_N(t)_\#\mu_s$ is indeed:
\begin{equation*}
    \Phi_N(t)_\#\mu_s = e^{-\frac{1}{2}(\norm{\Pi_N \Phi_N(-t)u}_{H^s}^2- \norm{\Pi_N u}_{H^s}^2)} d\mu_s = G_{s,N}(t,u) d\mu_s
\end{equation*}
and the challenge will be to take the limit $N \ra \infty$ into this formula. In order to do so, an integration by parts will give rise to a rewriting of the difference $\norm{\Pi_N \Phi_N(-t)u}_{H^s}^2- \norm{\Pi_N u}_{H^s}^2$ as:
\begin{equation*}
    \begin{split}
   -\frac{1}{2}\bigl( \norm{\Pi_N \Phi_N(-t)u}_{H^s}^2- \norm{\Pi_N u}_{H^s}^2\bigr) &= -\int_0^{-t} \frac{d}{d\tau}\frac{1}{2}\norm{\Pi_N \Phi_N(\tau)u}_{H^s}^2 d\tau \\
    &= R_{s,N}(\Phi_N(-t)u) - R_{s,N}(u)-\int_0^{-t} Q_{s,N}(\Phi_N(\tau)u)d\tau
    \end{split}
\end{equation*}
where $R_{s,N}$ and $Q_{s,N}$ will emerge in Section~\ref{section Poincaré-Dulac normal form reduction and modified energy} from a normal form reduction. Fortunately, we will see, respectively in Section~\ref{section Deterministic properties of the energy correction} and \ref{section the modified energy derivative at 0}, that $R_{s,N}$ and $Q_{s,N}$ are continuous functions on $\hsigt$ (for $\sigma < s-\frac{1}{2}$ close enough to $s-\frac{1}{2}$) that converge pointwisely\footnote{Actually, we will see that the convergence holds uniformly on compact sets of $\hsigt$, which is stronger. See Propositions \ref{prop approx R_s by R_s,N on compact sets} and \ref{prop Qs,N tends to Qs uniformly on compact sets} } to continuous functions respectively denoted  $R_s$ and $Q_s$(the proof of those facts will be postponed to Section~\ref{section Proofs of the deterministic properties}). Hence, the a priori ill-defined quantity $-\frac{1}{2}(\norm{\Phi(-t)u}_{H^s}^2- \norm{u}_{H^s}^2)$ will be seen as:
\begin{equation*}
    \begin{split}
    -\frac{1}{2}\bigl(\norm{\Phi(-t)u}_{H^s}^2- \norm{u}_{H^s}^2 \bigr)&:=  R_{s}(\Phi(-t)u) - R_{s}(u)-\int_0^{-t} Q_{s}(\Phi(\tau)u)d\tau \\
    & = \lim_{N \ra \infty} \big( R_{s,N}(\Phi_N(-t)u) - R_{s,N}(u)-\int_0^{-t} Q_{s,N}(\Phi_N(\tau)u)d\tau \bigr)
    \end{split}
\end{equation*}
And from this, we will be able in Section~\ref{section Transport of Gaussian measures under the flow} to prove that indeed:
\begin{equation*}
    \Phi(t)_\# \mu_s = G_s(t,u) d\mu_s
\end{equation*}

More precisely, assuming that $R_{s,N}$, $R_s$ and $Q_{s,N}$, $Q_s$ have been constructed (see Section~\ref{section Deterministic properties of the energy correction} and \ref{section the modified energy derivative at 0}), we will prove the following result:
\begin{thm}\label{thm Radon-Nikodym derivative for the transported measures}
    Let $s > \frac{3}{2}$ and  $\sigma< s-\frac{1}{2}$ close enough to $s-\frac{1}{2}$. Let $t \in \R$. Then, for every $N \in \N$, the transported measure $\Phi_N(t)_\# \mu_s$ has a density $G_{s,N}(t,.)$ with respect to $\mu_s$ given by:
    \begin{equation*}
        \begin{split}
        G_{s,N}(t,u) &= \textnormal{exp}\big( -\frac{1}{2}(\norm{\Pi_N \Phi_N(-t)u}_{H^s}^2- \norm{\Pi_N u}_{H^s}^2) \big) \\
        &= \textnormal{exp} \left(R_{s,N}(\Phi_N(-t)u) - R_{s,N}(u)-\int_0^{-t} Q_{s,N}(\Phi_N(\tau)u)d\tau \right) 
        \end{split}
    \end{equation*}
    Moreover, the transported measure $\Phi(t)_\# \mu_s$ has a density $G_{s}(t,.)$ with respect to $\mu_s$ given by:
    \begin{equation*}
        G_s(t,u) = \textnormal{exp} \left(R_{s}(\Phi(-t)u) - R_{s}(u)-\int_0^{-t} Q_{s}(\Phi(\tau)u)d\tau \right)
    \end{equation*}
    which is continuous on $\hsigt$. In addition, the densities $G_{s,N}(t,.)$ converge to $G_s(t,.)$ uniformly on compact sets of $\hsigt$.
\end{thm}
As a consequence, 
\begin{cor}\label{cor mu_s quasi-invariant}
    Let $s > \frac{3}{2}$. Then, the Gaussian measure $\mu_s$ is quasi-invariant along the flow of \eqref{NLS}.
\end{cor}
Results of this type were proven recently for many models, see \cite{burq2024almost,debussche2021quasi,forlano_and_soeng2022transport,forlano2022quasi,forlano_and_trenberth2019transport,genovese2023quasi,genovese2022quasi,genovese2023transport,gunaratnam2022quasi,oh_soeng2021quasi,oh2018optimal,oh2019quasi,oh2017quasi,oh2020quasi,planchon2022modified,planchon2020transport,sosoe2020quasi,tzvetkov2015quasiinvariant}.\\
The quasi-invariance of $\mu_s$ along the flow of \eqref{NLS} has already been proven in \cite{planchon2020transport} when $s=2k$, for all integers $k\geq 1$, where the authors relied on modified energy estimates (see Theorem 1.4 in \cite{planchon2020transport}). Here, our approach is different because our aim is to obtain directly the Radon-Nikodym derivative of the transported measures. Such an approach was adopted by Debussche and Tsutsumi in \cite{debussche2021quasi} and later by Genovese-Lucà-Tzvetkov \cite{genovese2023transport} and by Forlano and Seong in \cite{forlano_and_soeng2022transport}. More importantly, we are inspired by the method employed by Sun-Tzvetkov in \cite{sun2023quasi} in the context of the 3d energy critical nonlinear Schrödinger equation. However, in order to reach the full range $s > \frac{3}{2}$ for the quasi-invariance, we will need to employ sharper estimates, notably by incorporating dispersive effects through Strichartz estimates. In addition, the 1d case will allow us to benefit from deterministic properties -- through convergence on compact sets of \textit{truncated densities} -- in order to obtain the explicit formula for the Radon-Nikodym derivative of the transported measure $\Phi(t)_\# \mu_s$. It is however worth noticing that we will not need the \textit{remarkable cancellation} presented in \cite{sun2023quasi}.\\
It would be interesting to prove the quasi-invariance below the threshold $s>\frac{3}{2}$ of this paper. Indeed, the question of quasi-invariance for \eqref{NLS} still arises for smaller $s$ because, thanks to Bourgain in \cite{bourgain2004remark}, we know that \eqref{NLS} is still globally well-posed in $H^{\sigma}(\T)$ for $\sigma>\sigma^*$ with $\sigma^*<\frac{1}{2}$. More precisely, from more recent works we know that \eqref{NLS} is globally well-posed for $\sigma > \frac{2}{5}$, see \cite{Li_Wu_Xu_global} and \cite{bernier2024dynamics}. Besides, we know that the quasi-invariance is true for $s=1$. It is a consequence of an other result obtained by Bourgain in \cite{bourgain1994periodic} which states that the Gibbs measure:
\begin{equation*}
    Gb := \frac{1}{Z}e^{-\frac{1}{6}\norm{u}_{\L^6(\T)}^6}d\mu_1
\end{equation*}
is invariant under the flow of \eqref{NLS} for every time $t\in\R$ (meaning that $\Phi(t)_\# Gb = Gb$). \\
It would be also interesting to see if the recent work of Coe-Tolomeo in \cite{coe2024sharp} may be used to identify a sharp threshold $s_0$ above which the quasi-invariance holds, and under which the transported measure and the initial Gaussian measure are mutually singular for every time. \\

Let us observe that the formula $\Phi(t)_\#\mu_s=G_s(t,.)d\mu_s$ from Theorem~\ref{thm Radon-Nikodym derivative for the transported measures} implies that $G_s(t,.)$ belongs to $\L^1(d\mu_s)$. Thus, it is legitimate to ask if $G_s(t,.)$ belongs to $\L^p(d\mu_s)$, with $p>1$. In the second result of this paper, we provide a partial answer to this question. We prove in Section~\ref{section densities in Lp and convergence} that if we add a $H^1(\T)$-cutoff to the Gaussian measure $\mu_s$, defining the \textit{restricted Gaussian measure}:
\begin{equation}\label{intro restricted GM}
    \mu_{s,R}:= \1_{\{ \cC(u) \leq R \}} \mu_s
\end{equation}
where $\cC(u)$ is the (conserved by the flow) quantity:
\begin{equation}
    \cC(u) := \frac{1}{2} \norm{u}_{\L^2(\T)}^2 + H(u), \hspace{0.2cm} \text{with $H$ the Hamiltonian \eqref{hamiltonian}}
\end{equation}
 then, (with the same $G_{s,N}$ and $G_s$ as before) we have the following result:
\begin{thm}\label{thm density in L^p wrt mu_s,R}
    Let $s>\frac{3}{2}$ and $R>0$. Let $t\in \R$. Then, for every $N\in \N$:
\begin{align*}
     \Phi_N(t)_\#\mu_{s,R} &= G_{s,N}(t,u) d\mu_{s,R} & &\text{and,} & \Phi(t)_\#\mu_{s,R} &= G_{s}(t,u) d\mu_{s,R}
\end{align*}
Moreover, the densities $G_{s,N}(t,.)$, $G_s(t,.)$ belongs to $\L^p(d\mu_{s,R})$; and $G_{s,N}$ converge to $G_s(t,.)$ in $\L^p(d\mu_{s,R})$.
\end{thm}

We stress the fact that $s > \frac{3}{2}$ is the energy threshold where it is still possible to use the cutoff $\cC$. For $s\leq \frac{3}{2}$, $H^1(\T)$ is strictly contained in the full measure space $H^{(s-\frac{1}{2})-}(\T)$, and we would need to consider a renormalized cutoff as in \cite{tzvetkov2013gaussian}. \\

Our approach to prove Theorem~\ref{thm density in L^p wrt mu_s,R} is to work (instead of directly with the Gaussian measures $\mu_{s,R}$) with \textit{weighted Gaussian measures}, that we define in Section~\ref{section weighted Gaussian measures}. Formally, the idea is to replace the restricted Gaussian measure:
\begin{center}
    $\mu_{s,R}="\frac{1}{Z_s}\1_{\{ \cC(u) \leq R \}}e^{-\frac{1}{2}\norm{u}_{H^s}^2}du"$ by $\rho_{s,R}:= "\frac{1}{Z_s}\1_{\{ \cC(u) \leq R \}}e^{-E_s(u)}du" $
\end{center} 
where $E_s(u)$ is a \textit{modified energy} of the form:
\begin{equation*}
    E_s(u) = \frac{1}{2}\norm{u}_{H^s}^2 + R_s(u)
\end{equation*}
and where $R_s(u)$ is a correction term due to the non-linearity in \eqref{NLS}, which will be produced by the normal form reduction from Section~\ref{section Poincaré-Dulac normal form reduction and modified energy}. For the weighted Gaussian measures, we will be able to prove a quantitative inequality (see Proposition~\ref{prop quantitative quasi-invariance for rho_s,R}) that could be transferred afterwards to $\mu_{s,R}$ (see Proposition~\ref{prop quantitative quasi-invariance for mu_s,R}). \\
Besides, contrary to the proof of Theorem~\ref{thm Radon-Nikodym derivative for the transported measures} (which requires only deterministic considerations), the proof of Theorem~\ref{thm density in L^p wrt mu_s,R} will require in addition probabilistic tools, relying on the fact that the initial data are distributed according to Gaussian measures. Mainly, in order to prove $\L^p$-estimates on $R_s$ and $Q_s$ (respectively in Section~\ref{section Estimates for the weight of the weighted Gaussian measures} and \ref{section Estimates for the differential of the modified energy}), we will use the independence between high and low frequency Gaussians, along with a conditional Wiener chaos estimate (see Lemma~\ref{lem Wiener chaos}). This method was adopted before in \cite{sun2023quasi}. However, we point out that we will use the Wiener chaos estimate with respect to three high-frequency Gaussians (that is with $m=3$ in Lemma~\ref{lem Wiener chaos}) whereas in \cite{sun2023quasi} the authors performed the Wiener chaos with respect to two high-frequency Gaussians (that is with $m=2$ in Lemma~\ref{lem Wiener chaos}). This remark is in fact significant because it will imply in our analysis that a "pairing between generations" (see Section 5 of \cite{sun2023quasi}) cannot occur. On the contrary, such a pairing could occur in \cite{sun2023quasi}, and the authors dealt with it by emphasizing a "remarkable cancellation" (see Section 7 of \cite{sun2023quasi}). \\

\paragraph{\textbf{Organization of the paper}}
We organize this paper as follows: \\
In Section \ref{section Poincaré-Dulac normal form reduction and modified energy}, we perform a normal form reduction where energy-type quantities will emerge. In particular, the normal form reduction will produce two crucial quantities: one called the \textit{energy correction}, denoted $R_{s,N}$, and the other called the \textit{derivative of the modified energy at 0}, denoted $Q_{s,N}$. Section \ref{section Deterministic properties of the energy correction} and Section \ref{section the modified energy derivative at 0} are respectively dedicated to deterministic properties of $R_{s,N}$ and $Q_{s,N}$. In Section \ref{section Transport of Gaussian measures under the truncated flow}, we prove that (for every $s>\frac{3}{2}$) $\mu_s$ is quasi-invariant along the truncated flow $\Phi_N(t)$ (for $N \in \N$). More precisely, we prove the formula $\Phi_N(t)_\# \mu_s = G_{s,N}(t,.)d\mu_s$, providing an explicit formula for $G_{s,N}(t,u)$. In Section \ref{section Transport of Gaussian measures under the flow}, we extend this formula to the flow $\Phi(t)$; we prove that $\Phi(t)_\# \mu_s = G_s(t,.) d\mu_s$, where the density $G_s(t,.)$ will be the pointwise limit of the truncated densities $G_{s,N}(t,.)$. In section \ref{section densities in Lp and convergence}, we prove that the densities $G_{s,N}(t,.)$ and $G_s(t,.)$ belong to $L^p$ with respect to the restrictions of $\mu_s$ on bounded sets of $H^1(\T)$; we also prove that, with respect to these measures, $G_{s,N}(t,.)$ converges to $G_s(t,.)$ in $L^p$. To do so, we will rely on the introduction, in Section \ref{section weighted Gaussian measures}, of weighted Gaussian measures. Then, the remaining part of the paper will be dedicated to the proof of the energy estimates we used in the previous sections.
In Section~\ref{section Tools for the energy estimates}, we gather the deterministic and probabilistic tools of this paper. Section \ref{section Proofs of the deterministic properties} is dedicated to the proof of the deterministic properties stated in Sections \ref{section Deterministic properties of the energy correction} and \ref{section the modified energy derivative at 0}. Section \ref{section Estimates for the weight of the weighted Gaussian measures} and Section \ref{section Estimates for the differential of the modified energy} are respectively dedicated to the proof of $L^p$ estimates on $R_{s,N}$ and $Q_{s,N}$. Finally, in Appendix \ref{appendix Construction and properties of the flow and the truncated flow}, we provide a local and global Cauchy theory for \eqref{NLS} and \eqref{truncated equation} on $\hsigt$, for $\sigma \geq 1$, and we then prove an approximation property of $\Phi(t)$ by $\Phi_N(t)$. Besides, we will prove a decomposition for the truncated flow $\Phi_N(t)$. 

\subsection{Further remarks}
\begin{rem}
    In our approach, we view the difference $\norm{\Phi(t)u}_{H^s}^2- \norm{u}_{H^s}^2$ as the limit of continuous functions on $ \hsigt$, with $\sigma < s-\frac{1}{2}$ close enough to $s-\frac{1}{2}$. An alternative to give a meaning to this quantity (for $u \in \hsigt$) could have been to use a \textit{nonlinear smoothing} for the solutions of \eqref{NLS}. In the context of this paper, with $s>\frac{3}{2}$, the result in Theorem 1 from the interesting work \cite{mcconnell_nonlin_smoothing} implies that for  $\sigma < s-\frac{1}{2}$ (close enough to $s-\frac{1}{2}$), we have the following nonlinear smoothing (for $u \in \hsigt$):
    \begin{equation*}
       v(t) := \Phi(t)u - e^{it\p_x^2} e^{-i\frac{3}{2\pi} \int_0^t \norm{\Phi(\tau)u}_{\L^4}^4 d\tau} u \in \cC([-T,T]; H^{\sigma + 1-\eps}(\T))
    \end{equation*}
    for $0<\eps < 1$, and at least for small time $T>0$.
    Hence, we could have tried to define the difference $\norm{\Phi(t)u}_{H^s}^2- \norm{u}_{H^s}^2$ as :
    \begin{equation*}
        \norm{\Phi(t)u}_{H^s}^2- \norm{u}_{H^s}^2 := \norm{v(t)}_{H^s}^2 + 2 \Re \bigl \langle \langle \nabla \rangle^{2s - \sigma} v(t),  \langle \nabla \rangle^\sigma \big(e^{it\p_x^2} e^{-i\frac{3}{2\pi} \int_0^t \norm{\Phi(\tau)u}_{\L^4}^4 d\tau} u\big) \bigr \rangle_{\L^2(\T)}
    \end{equation*}
    since this formula is true when $u$ is smooth (with $\langle.,. \rangle_{\L^2(\T)}$ the $\L^2(\T)$-scalar product). With $\eps$ close to 0, the $1-\eps$ gain of regularity in the smoothing above implies that for $u \in \hsigt$, $\norm{v(t)}_{H^s(\T)}^2<+\infty$ because $\sigma+1-\eps$ is close to $s+\frac{1}{2}>s$. However, this smoothing is not enough to define the term $\langle \nabla \rangle^{2s-\sigma}v(t)$ because $\sigma + 1 - \eps < 2s-\sigma$, no matter how close $\eps$ and $\sigma$ are close to 0 and $s-\frac{1}{2}$ respectively. For  $\langle \nabla \rangle^{2s-\sigma}v(t)$ to be well-defined, we need a gain of $1+ \eps'$ regularity, where $\eps'>0$.
\end{rem}

\begin{rem} In this paper, we deal with the defocusing NLS, but it would have also been possible to consider the focusing NLS. However, in the focusing case, we would have needed an additional cut-off on small $\L^2(\T)$-initial data, ensuring that the flow of \eqref{NLS} is global and a control on the $H^1(\T)$-norm of the solutions.
\end{rem}

\paragraph{\textbf{Acknowledgments}} This work is partially supported by the ANR project Smooth ANR-22-
CE40-0017. The author is grateful to his advisors Chenmin Sun and Nikolay Tzvetkov for suggesting this problem and for their valuable advice. The author would also like to thank Tristan Robert for pointing out the reference \cite{mcconnell_nonlin_smoothing}.

\section{Poincaré-Dulac normal form reduction and modified energy}\label{section Poincaré-Dulac normal form reduction and modified energy}
\subsection{The truncated system}
Let $N\in \N$. We work with the following equation :
\begin{equation}\label{truncated equation}
    \begin{cases}
        i\p_t u + \p_x^2 u = \Pi_N \left(|\Pi_Nu|^4\Pi_Nu \right), \hspace{0.2cm} (t,x) \in \R \times \T \\
        u|_{t=0} = u_0 
    \end{cases}
\end{equation}
called the \textit{truncated equation}, where $\Pi_N$ is the projector on frequencies $\leq N$. More precisely, 
\begin{equation*}
    \Pi_N \left( \sum_{k \in \Z} u_k e^{ikx} \right) := \sum_{|k| \leq N} u_k e^{ikx}
\end{equation*}
We also define $\Pi_N^{\perp} := Id-\Pi_N$ as :
\begin{equation*}
    \Pi_N^{\perp} \left( \sum_{k \in \Z} u_k e^{ikx} \right) := \sum_{|k| > N} u_k e^{ikx}
\end{equation*}
Equation \eqref{truncated equation} is a smoothly approximated system of \eqref{NLS}. We denote by $\Phi_N$ the flow of \eqref{truncated equation}, called the \textit{truncated flow}. Sometimes, we might use the notation $\Phi_{\infty}$ instead of $\Phi$ to refer to the flow of \eqref{NLS}. If we set 
\begin{align*}
    E_N &:= \Pi_N \L^2(\T), & E_N^{\perp} := \Pi_N^{\perp}\L^2(\T) = (Id -\Pi_N) \L^2(\T),
\end{align*}
the truncated flow $\Phi_N(t)$ can be factorized as $(\tld{\Phi}_N(t),e^{it\p_x^2})$ on $E_N \times E_N^{\perp}$ in the sense that 
\begin{equation}\label{factorization of the truncated flow}
    \Phi_N(t)u_0 = \underbrace{\tld{\Phi}_N(t) \Pi_N u_0}_{\in E_N} + \underbrace{e^{it\p_x^2}\Pi_N^{\perp}u_0}_{\in E_N^{\perp}} 
\end{equation}
where $\tld{\Phi}_N$ is the flow of a (finite dimensional) ordinary differential equation (ODE). In Appendix \ref{appendix Construction and properties of the flow and the truncated flow}, we show the local and global well-posedness of the truncated flow as well as its structure. \\

In the same vein, we can decompose the Gaussian measure $\mu_s$ as 
\begin{equation}\label{mu_s = mu_s,N times mu_s,N^perp}
    \mu_s = \mu_{s,N} \otimes \mu_{s,N}^{\perp}
\end{equation}
where $\mu_{s,N}$ and $\mu_{s,N}^{\perp}$ are respectively the law of 
\begin{align*}
    \omega & \longmapsto \sum_{|n|\leq N} \frac{g_n(\omega)}{\langle n\rangle^s}e^{inx} & & \text{and} & \omega  \longmapsto \sum_{|n| > N} \frac{g_n(\omega)}{\langle n\rangle^s}e^{inx}
\end{align*}
Equivalently, $\mu_{s,N}$ is the probability measure on $E_N$ given by:
\begin{equation}\label{mu_s,N}
    \mu_{s,N} = \frac{1}{Z_N} \prod_{|n|\leq N} e^{-\frac{1}{2}\langle n \rangle^2|\hat{u}(n)|^2} d\hat{u}(n) =\frac{1}{Z_N} e^{-\frac{1}{2}\norm{\Pi_N u}_{H^s(\T)}^2} \left( \prod_{|n|\leq N} d\hat{u}(n) \right) 
\end{equation}
where $Z_N>0$ is a normalizing constant, and $d\hat{u}(n)$ is the Lebesgue measure on $\text{Span}(e^{inx})=\C \cdot e^{inx}$.

%%%%%%%%%%%%%%%%%%%%%%%%%%%%%%%%%%%%%%%%%%%%%%%%%%%%%%%%%%%%%%%%%%%%%%%%%%%%%%%%%%%%%%%%%%%%%%%%%%%%%%%%%%%%%%%%%%%%%%%%%%%%%%%%%%%%%%
\subsection{Normal form reduction and modified energy}

We consider a smooth solution $u_N$ of \eqref{truncated equation}. By factoring through the linear flow, we introduce a new unknown 
\begin{equation*}
    v_N := e^{-it\p_x^2}u_N
\end{equation*}
Note that from \eqref{factorization of the truncated flow}, we have $v_N =  \Pi_N v_N + \Pi_N^{\perp}u_0$, so $\p_t v_N = \p_t \Pi_N v_N$. In other words, if we invoke
\begin{equation*}
    w_N := \Pi_N v_N,
\end{equation*}
we have $\p_tv_N = \p_t w_N$. We denote respectively by $u_k(t), v_k(t)$ and $w_k(t)$ the $k$-th Fourier coefficient of $u_N(t),v_N(t)$ and $w_N(t)$, and for better readability, we will simply write $u_k,v_k$ and $w_k$, omitting the variable $t$.\\

Now, we observe that $v_N$ satisfies the equation 
\begin{equation}\label{equation idt v}
    i\p_t v = e^{-it\p_x^2} \Pi_N \left( |e^{it\p_x^2} \Pi_N v|^4e^{it\p_x^2}\Pi_N v\right)
\end{equation}
Since the Fourier transform converts the product into convolution, we deduce from \eqref{equation idt v} that
\begin{equation}\label{equation idt wk}
    \begin{split}
    i \p_t v_k = i \p_t w_k  &=  \1_{|k| \leq N} \sum_{k_1-k_2+k_3-k_4+k_5=k} e^{-it \Omega(\vec{k})} \left( \prod_{j=1}^5 \1_{|k_j|\leq N} \right) v_{k_1}\cjg{v_{k_2}}...v_{k_5} \\
    & = \1_{|k| \leq N} \sum_{k_1-k_2+k_3-k_4+k_5=k} e^{-it \Omega(\vec{k})}  w_{k_1}\cjg{w_{k_2}}...w_{k_5}
    \end{split}
\end{equation}
where $\vec{k}=(k_1,...,k_5,k)$ and,
\begin{equation*}
    \Omega(\vec{k}) := \sum_{j=1}^5 (-1)^{j-1} k_j^2 -k^2
\end{equation*}
is the so-called \textit{resonant function}. 
In the sequel, we use the equivalent Sobolev norm for $s\geq 0$ 
\begin{equation*}
    \triplenorm{f}^2_{H^s(\T)} := \sum_{k \in \Z} (1+|k|^{2s}) |\hat{f}(k)|^2
\end{equation*}
which is more convenient for our purpose. Let us now compute $\frac{1}{2}\frac{d}{dt}\triplenorm{u_N}^2_{H^s(\T)}$:
\begin{equation*}
    \begin{split}
        \frac{1}{2}\frac{d}{dt}\triplenorm{u_N}^2_{H^s(\T)} = \frac{1}{2}\frac{d}{dt}\triplenorm{v_N}^2_{H^s(\T)}= \frac{1}{2}\frac{d}{dt}\triplenorm{w_N}^2_{H^s(\T)}
        & = \Im \left(i\p_t w_N |  w_N\right)_{H^s \times H^s} \\
        &= \Im \sum_{\substack{k_6 \in \Z \\ |k_6| \leq N}} (1+|k_6|^{2s})  (i \p_t w_{k_6}) \cjg{w_{k_6}}
    \end{split}
\end{equation*}
Plugging \eqref{equation idt wk} into this, we get
\begin{equation*}
        \frac{1}{2}\frac{d}{dt}\triplenorm{w_N}^2_{H^s(\T)}
          =\Im \sum_{\linc}(1+|k_6|^{2s})e^{-it\Omega(\vec{k})}w_{k_1}\cjg{w_{k_2}}...w_{k_5}\cjg{w_{k_6}}
\end{equation*}
(where now $\vec{k} = (k_1,...,k_5,k_6)$ and still $\Omega(\vec{k}) = \sum_{j=1}^6 (-1)^{j-1} k_j^2$).\\

The above formula can be symmetrized using the symmetries of the resonant function and of the indices. Firstly, the change of variables $k_4 \rla k_6$ and  $k_2 \rla k_6$ respectively yield,
\begin{equation*}
    \begin{split}
         \frac{1}{2}\frac{d}{dt}\triplenorm{w_N}^2_{H^s(\T)} &= \Im \sum_{\linc}(1+|k_4|^{2s})e^{-it\Omega(\vec{k})}w_{k_1}\cjg{w_{k_2}}...\cjg{w_{k_6}} \\
    & =\Im \sum_{\linc}(1+|k_2|^{2s})e^{-it\Omega(\vec{k})}w_{k_1}\cjg{w_{k_2}}...\cjg{w_{k_6}}
    \end{split}
\end{equation*}
so that, 
\begin{equation}\label{first symmetrization}
    \frac{1}{2}\frac{d}{dt}\triplenorm{w_N}^2_{H^s(\T)} = \Im \sum_{\linc}\frac{3+|k_2|^{2s}+|k_4|^{2s}+|k_6|^{2s}}{3}e^{-it\Omega(\vec{k})}w_{k_1}\cjg{w_{k_2}}...\cjg{w_{k_6}}
\end{equation}
Secondly, the change of variables $(k_1,k_3,k_5) \rla (k_2,k_4,k_6)$ and $(k_1,k_5,k_3) \rla (k_2,k_4,k_6)$ and $(k_5,k_3,k_1) \rla (k_2,k_4,k_6)$ respectively yield,
\begin{equation*}
    \begin{split}
         \frac{1}{2}\frac{d}{dt}\triplenorm{w_N}^2_{H^s(\T)} &= -\Im \sum_{\linc}(1+|k_5|^{2s})e^{-it\Omega(\vec{k})} w_{k_1}\cjg{w_{k_2}}...\cjg{w_{k_6}} \\
    & = -\Im \sum_{\linc}(1+|k_3|^{2s})e^{-it\Omega(\vec{k})}w_{k_1}\cjg{w_{k_2}}...\cjg{w_{k_6}} \\
    & =-\Im \sum_{\linc}(1+|k_1|^{2s})e^{-it\Omega(\vec{k})}w_{k_1}\cjg{w_{k_2}}...\cjg{w_{k_6}}
    \end{split}
\end{equation*}
so that,
\begin{equation}\label{second symmetrization}
    \frac{1}{2}\frac{d}{dt}\triplenorm{w_N}^2_{H^s(\T)} = -\Im \sum_{\linc}\frac{3+|k_1|^{2s}+|k_3|^{2s}+|k_5|^{2s}}{3}e^{-it\Omega(\vec{k})}w_{k_1}\cjg{w_{k_2}}...\cjg{w_{k_6}}
\end{equation}
Finally, from combining \eqref{first symmetrization} and \eqref{second symmetrization}, we obtain the symmetrized formula
\begin{equation}
     \frac{1}{2}\frac{d}{dt}\triplenorm{w_N}^2_{H^s(\T)} = -\frac{1}{6} \Im \sum_{\linc} \psi_{2s}(\vec{k})e^{-it\Omega(\vec{k})}w_{k_1}\cjg{w_{k_2}}...\cjg{w_{k_6}}
\end{equation}
where, 
\begin{equation*}
    \psi_{2s}(\vec{k}) := \sum_{j=1}^{6} (-1)^{j-1}|k_j|^{2s}
\end{equation*}

In order to perform a differentiation by parts, we decompose the set of indices according to whether $\Omega(\vec{k}) = 0$ or not : 
\begin{equation*}
    \begin{split}
        \frac{1}{2}\frac{d}{dt}\triplenorm{w_N}^2_{H^s(\T)} & = -\frac{1}{6} \Im \sum_{\substack{\linc \\ \Omgz}} \psi_{2s}(\vec{k})e^{-it\Omega(\vec{k})}w_{k_1}\cjg{w_{k_2}}...\cjg{w_{k_6}} \\
        & \hspace{0.4cm} -\frac{1}{6} \Im \sum_{\substack{\linc \\ \Omgnz}} \psi_{2s}(\vec{k})e^{-it\Omega(\vec{k})}w_{k_1}\cjg{w_{k_2}}...\cjg{w_{k_6}} \\
        & =  -\frac{1}{6} \Im \sum_{\substack{\linc \\ \Omgz}} \psi_{2s}(\vec{k})e^{-it\Omega(\vec{k})}w_{k_1}\cjg{w_{k_2}}...\cjg{w_{k_6}} \\
        & \hspace{0.4cm} -\frac{1}{6} \Im \sum_{\substack{\linc \\ \Omgnz}} \frac{\psi_{2s}(\vec{k})}{-i \Omega(\vec{k})}\p_t \left(e^{-it\Omega(\vec{k})}w_{k_1}\cjg{w_{k_2}}...\cjg{w_{k_6}} \right) \\
        & \hspace{0.4cm}  +\frac{1}{6} \Im \sum_{\substack{\linc \\ \Omgnz}} \frac{\psi_{2s}(\vec{k})}{-i \Omega(\vec{k})} e^{-it\Omega(\vec{k})} \p_t \left(w_{k_1}\cjg{w_{k_2}}...\cjg{w_{k_6}} \right)
    \end{split}
\end{equation*}
Hence,
\begin{equation}\label{first computation of d_dt |||v|||^2}
    \begin{split}
        &\frac{d}{dt}\Big( \frac{1}{2}\triplenorm{w_N}^2_{H^s(\T)}  +\frac{1}{6} \Im \sum_{\substack{\linc \\ \Omgnz}} \frac{\psi_{2s}(\vec{k})}{-i \Omega(\vec{k})} e^{-it\Omega(\vec{k})}w_{k_1}\cjg{w_{k_2}}...\cjg{w_{k_6}} \Big) \\
        & =  -\frac{1}{6} \Im \sum_{\substack{\linc \\ \Omgz}} \psi_{2s}(\vec{k})e^{-it\Omega(\vec{k})}w_{k_1}\cjg{w_{k_2}}...\cjg{w_{k_6}}  +\frac{1}{6} \Im \sum_{\substack{\linc \\ \Omgnz}} \frac{\psi_{2s}(\vec{k})}{-i \Omega(\vec{k})} e^{-it\Omega(\vec{k})} \p_t \left(w_{k_1}\cjg{w_{k_2}}...\cjg{w_{k_6}} \right)
    \end{split}
\end{equation}
Now, motivated by the above formula, we define the following quantities: 
\begin{defns}[Energy correction and modified energy] \label{def normal form R and E}
    Let $N \in \N$. 
    \begin{enumerate}
        \item For $u \in \cS'(\T)$, we define:
    \begin{equation}
    R_{s,N}(u) := \frac{1}{6} \Im \sum_{\substack{\linc \\ \Omgnz}} \frac{\psi_{2s}(\vec{k})}{-i \Omega(\vec{k})} \left( \prod_{j=1}^6 \mathds{1}_{|k_j|\leq N}\right) u_{k_1}\cjg{u_{k_2}}...\cjg{u_{k_6}}
\end{equation} and,
\begin{equation}\label{modified energy}
   E_{s,N}(u) := \frac{1}{2} \triplenorm{\Pi_N u}_{H^s(\T)}^2 + R_{s,N}(u)
\end{equation}
    \item For $u \in \cC^{\infty}(\T)$, we define:
    \begin{equation}\label{Rs,infty}
    R_{s,\infty}(u) := \frac{1}{6} \Im \sum_{\substack{\linc \\ \Omgnz}} \frac{\psi_{2s}(\vec{k})}{-i \Omega(\vec{k})}u_{k_1}\cjg{u_{k_2}}...\cjg{u_{k_6}}
\end{equation}
and,
\begin{equation}
    E_{s,\infty} := \frac{1}{2}\triplenorm{u}_{H^s(\T)}^2 + R_{s,\infty}(u)
\end{equation}
    \end{enumerate}
For $N \in \N \cup \{ \infty \}$, $R_{s,N}$ and $E_{s,N}$ are respectively called the \textit{energy correction} and the \textit{modified energy}. Moreover, we will use interchangeably the notation $R_s$ and $E_s$ to refer respectively to $R_{s,\infty}$ and $E_{s,\infty}$.
\end{defns}

\begin{rem} In Section \ref{section Deterministic properties of the energy correction} (in Proposition \ref{prop R_s,N continuous}), we will see that for $s>\frac{3}{2}$ and $\sigma< s-\frac{1}{2}$ close enough to $s-\frac{3}{2}$, we can extend $R_s$ to $\hsigt$, because we will be able to prove that the sum in \eqref{Rs,infty} is absolutely convergent for every $u \in \hsigt$.
\end{rem}

Using the modified energy \eqref{modified energy}, we can rewrite \eqref{first computation of d_dt |||v|||^2} as 
\begin{equation*}
    \begin{split}
   \frac{d}{dt} E_{s,N}(\Pi_N u_N(t)
   ) & = -\frac{1}{6} \Im \sum_{\substack{\linc \\ \Omgz}} \psi_{2s}(\vec{k})e^{-it\Omega(\vec{k})}w_{k_1}\cjg{w_{k_2}}...\cjg{w_{k_6}} \\ 
        & +\frac{1}{6} \Im \sum_{\substack{\linc \\ \Omgnz}} \frac{\psi_{2s}(\vec{k})}{-i \Omega(\vec{k})} e^{-it\Omega(\vec{k})} \p_t \left(w_{k_1}\cjg{w_{k_2}}...\cjg{w_{k_6}} \right)
    \end{split}
\end{equation*}
Furthermore, expanding the time derivation $\p_t \left(w_{k_1}\cjg{w_{k_2}}...\cjg{w_{k_6}} \right)$ and performing the change of variables $k_1 \rla k_3$, $k_1 \rla k_5$, $k_2 \rla k_4$ and $k_2 \rla k_6$, we get that 

\begin{equation*}
    \begin{split}
        \frac{d}{dt} E_{s,N}(\Pi_N u_N(t)
   ) = & -\frac{1}{6} \Im \sum_{\substack{\linc \\ \Omgz}} \psi_{2s}(\vec{k})e^{-it\Omega(\vec{k})}w_{k_1}\cjg{w_{k_2}}...\cjg{w_{k_6}} \\ 
        & +\frac{1}{2} \Im \sum_{\substack{\linc \\ \Omgnz}} \frac{\psi_{2s}(\vec{k})}{ \Omega(\vec{k})} e^{-it\Omega(\vec{k})} \big(i \p_tw_{k_1}\big)\cjg{w_{k_2}}...\cjg{w_{k_6}}  \\ 
        & +\frac{1}{2} \Im \sum_{\substack{\linc \\ \Omgnz}} \frac{\psi_{2s}(\vec{k})}{ \Omega(\vec{k})} e^{-it\Omega(\vec{k})} w_{k_1}\left( \cjg{-i\p_t w_{k_2}}\right)...\cjg{w_{k_6}}
    \end{split}
\end{equation*}

At this point, we can make use of the formula \eqref{equation idt wk} so that the above formula can be rewritten as 
\begin{equation}\label{time derivative modified energy with variable w}
    \begin{split}
       \frac{d}{dt} E_{s,N}(\Pi_N u_N(t)) & =-\frac{1}{6} \Im \sum_{\substack{\linc \\ \Omgz}} \psi_{2s}(\vec{k})e^{-it\Omega(\vec{k})}w_{k_1}\cjg{w_{k_2}}...\cjg{w_{k_6}} \\ 
        & +\frac{1}{2} \Im \sum_{\substack{\linc \\ \lincba \\\Omgnz}} \frac{\psi_{2s}(\vec{k})}{ \Omega(\vec{k})}  e^{-it(\Omega(\vec{k}) + \Omega(\vec{p}))} w_{p_1}\cjg{w_{p_2}}...w_{p_5}\cjg{w_{k_2}}... \cjg{w_{k_6}}  \\ 
        & -\frac{1}{2} \Im \sum_{\substack{\linc \\ \lincbb \\\Omgnz}} \frac{\psi_{2s}(\vec{k})}{ \Omega(\vec{k})}  e^{-it(\Omega(\vec{k}) - \Omega(\vec{q}))} w_{k_1}\cjg{w_{q_1}}w_{q_2}...\cjg{w_{q_5}}w_{k_3}... \cjg{w_{k_6}} 
    \end{split}
\end{equation}
where $\vec{p} = (p_1,...,p_5,k_1)$ and $\vec{q} = (q_1,...,q_5,k_2)$. If we come back to the variable $u_N$ on the right hand side of \eqref{time derivative modified energy with variable w}, we obtain 
\begin{equation}\label{time derivative modified energy}
    \begin{split}
       \frac{d}{dt} E_{s,N}(\Pi_N u_N(t)) & =-\frac{1}{6} \Im \sum_{\substack{\linc \\ \Omgz}} \psi_{2s}(\vec{k})u_{k_1}\cjg{u_{k_2}}...\cjg{u_{k_6}} \left( \prod_{j=1}^6 \1_{|k_j| \leq N}\right) \\ 
        & +\frac{1}{2} \Im \sum_{\substack{\linc \\ \lincba \\\Omgnz}} \frac{\psi_{2s}(\vec{k})}{ \Omega(\vec{k})} u_{p_1}\cjg{u_{p_2}}...u_{p_5}\cjg{u_{k_2}}... \cjg{u_{k_6}} \left( \prod_{j=1}^6 \1_{|k_j| \leq N}\right)\left( \prod_{j=1}^5 \1_{|p_j| \leq N}\right) \\ 
        & -\frac{1}{2} \Im \sum_{\substack{\linc \\ \lincbb \\\Omgnz}} \frac{\psi_{2s}(\vec{k})}{ \Omega(\vec{k})} u_{k_1}\cjg{u_{q_1}}u_{q_2}...\cjg{u_{q_5}}u_{k_3}... \cjg{u_{k_6}} \left( \prod_{j=1}^6 \1_{|k_j| \leq N}\right)\left( \prod_{j=1}^5 \1_{|p_j| \leq N}\right)
    \end{split}
\end{equation}
This quantity evaluated at time 0 will play an important role because it will appear in the explicit formula for the Radon-Nikodym derivatives of the transported measures. Motivated by the above formula, we define the following quantities:

\begin{defns}[Modified energy derivative at 0]\label{def normal form Q} Let $N \in \N$. Let $\Phi_N(t)$ the flow of \eqref{truncated equation}. \\
\begin{enumerate}
    \item For $u \in \cS'(\T)$, we define:
    \begin{equation*}\label{def Qs,N in the normal form}
        Q_{s,N}(u) :=  \frac{d}{dt} E_{s,N}(\Pi_N \Phi_N(t)u) |_{t=0} = \text{(RHS)}|_{t=0} \hsp \text{of \eqref{time derivative modified energy}}
    \end{equation*}
    \item For $u \in \cC^{\infty}(\T)$, we define:
    \begin{equation}\label{def Qs,infty normal form}
    \begin{split}
       Q_{s,\infty}(u) & := \text{(RHS)}|_{t=0} \hsp \text{of \eqref{time derivative modified energy} with $N=\infty$} \\
       & = -\frac{1}{6} \Im \sum_{\substack{\linc \\ \Omgz}} \psi_{2s}(\vec{k})u_{k_1}\cjg{u_{k_2}}...\cjg{u_{k_6}} +\frac{1}{2} \Im \sum_{\substack{\linc \\ \lincba \\\Omgnz}} \frac{\psi_{2s}(\vec{k})}{ \Omega(\vec{k})} u_{p_1}\cjg{u_{p_2}}...u_{p_5}\cjg{u_{k_2}}... \cjg{u_{k_6}} \\ 
        & -\frac{1}{2} \Im \sum_{\substack{\linc \\ \lincbb \\\Omgnz}} \frac{\psi_{2s}(\vec{k})}{ \Omega(\vec{k})} u_{k_1}\cjg{u_{q_1}}u_{q_2}...\cjg{u_{q_5}}u_{k_3}... \cjg{u_{k_6}} 
    \end{split}
\end{equation}
\end{enumerate}
   For $N\in \N \cup \{ \infty \}$, $Q_{s,N}$ is called the \textit{derivative of the modified energy at 0}. Moreover, we will use interchangeably the notation $Q_s$ to refer to $Q_{s,\infty}$. 
\end{defns}

\begin{rem}For $N \in \N$, note that from the additivity of the flow we have:
    \begin{equation}\label{Qs,N(phi_N(t)u)}
    \begin{split}
    Q_{s,N}(\Phi_N(\tau)u) = \frac{d}{dt} E_{s,N}\left( \Pi_N \Phi_N(t)\Phi_N(\tau)u \right) |_{t=0} &= \frac{d}{dt} E_{s,N}\left( \Pi_N \Phi_N(t+\tau)u \right) |_{t=0} \\
    &= \frac{d}{dt} E_{s,N}\left( \Pi_N \Phi_N(t)u \right) |_{t=\tau}
    \end{split}
\end{equation}
\end{rem}

\begin{rem} In Section \ref{section the modified energy derivative at 0} (in Proposition \ref{prop continuity of Qs,N and deterministic estimate}), we will see that for $s>\frac{3}{2}$ and $\sigma< s-\frac{1}{2}$ close enough to $s-\frac{1}{2}$, we can extend $Q_s$ to $\hsigt$, because we will be able to prove that the sum in \eqref{def Qs,infty normal form} is absolutely convergent for every $u \in \hsigt$.
\end{rem}

\section{Definition and deterministic properties of the energy correction}\label{section Deterministic properties of the energy correction}
In Section~\ref{section Poincaré-Dulac normal form reduction and modified energy}, an energy correction $R_{s,N}$ (for $N \in \N \cup \{ \infty \}$) has emerged from the normal form reduction (see Definition ~\ref{def normal form R and E}). We dedicate this section to the study of this quantity on $\hsigt$ for $\sigma < s - \frac{1}{2}$ (close enough to $s - \frac{1}{2}$), provided that $s>\frac{3}{2}$. In particular, we show that we are able to extend $R_{s,\infty}$ (which is defined on $\cC^{\infty}(\T)$) to $\hsigt$, because we will prove that the right hand side in \eqref{Rs,infty} is actually an absolutely convergent sum for every  $u \in \hsigt$. This section is composed of two results; firstly, we state that for $\sigma<s-\frac{1}{2}$ close enough to $s-\frac{1}{2}$, the $R_{s,N}$ are continuous functions on $\hsigt$ given by the diagonal of a continuous multi-linear form; secondly, we state that $R_{s,N}$ converges to $R_{s,\infty}$ uniformly on compact sets of $\hsigt$.

\begin{prop}\label{prop R_s,N continuous}
    Let $s>\frac{3}{2}$. For $\sigma < s - \frac{1}{2}$ close enough to $s-\frac{1}{2}$, there exists a constant $C>0$ such that for every $u^{(1)},...,u^{(6)} \in \hsigt$ :
    \begin{equation}\label{energy correction with absolute value}
        \sum_{\substack{\linc \\ \Omgnz}} \big| \frac{\psi_{2s}(\vec{k})}{\Omega(\vec{k})} \big| \big|u^{(1)}_{k_1}\cjg{u^{(2)}_{k_2}}...\cjg{u^{(6)}_{k_6}} \big| \leq C \prod_{j=1}^6 \hsignorm{u^{(j)}}
    \end{equation}
    Hence, the map:
    \begin{equation*}
        \function{\cR}{\hsigt^6}{\C}{(u^{(1)},...,u^{(6)})}{\mathlarger{\sum \limits_{\substack{\linc \\ \Omgnz}} \frac{\psi_{2s}(\vec{k})}{\Omega(\vec{k})}u^{(1)}_{k_1}\cjg{u^{(2)}_{k_2}}...\cjg{u^{(6)}_{k_6}}}}
    \end{equation*}
    is a continuous multi-linear form.
    Then, for $N \in \N \cup \{ \infty \}$\footnote{Using the notation $\Pi_{\infty} = id $}, setting \footnote{By abuse of notation, we still denote by $R_{s,N}$ this new function, even though it has already been defined in Definition~\ref{def normal form R and E}. The $R_{s,N}$ of Definition~\ref{def normal form R and E} and the $R_{s,N}$ of this proposition coincide on $\cC^{\infty}(\T)$.}:
    \begin{equation}\label{R_s,N diag multilinear form}
        \function{R_{s,N}}{\hsigt}{\C}{u}{\frac{1}{6} \Re \hsp \cR(\Pi_N u,...,\Pi_N u)},
    \end{equation}
    we deduce that $R_{s,N}$ is a continuous map on $\hsigt$ that satisfies for all $u,v \in \hsigt$:
    \begin{equation*}
        \left| R_{s,N}(u)- R_{s,N}(v) \right| \leq C \hsignorm{u-v} (\hsignorm{u}^5 + \hsignorm{v}^5)
    \end{equation*}
    uniformly in $N \in \N \cup \{ \infty \}$.
\end{prop}

\begin{notn}
   We also use the notation $R_s$ to refer  to $R_{s,\infty}$.
\end{notn}

We postpone the proof of this proposition for Section~\ref{section Proofs of the deterministic properties}, where a detailed analysis is provided.
Instead, assuming this proposition, we are able to prove now the following approximation result:

\begin{prop}\label{prop approx R_s by R_s,N on compact sets}
    Let $s>\frac{3}{2}$. Let $\sigma < s-\frac{1}{2}$ close enough to $s-\frac{1}{2}$ so that the conclusion of Proposition \ref{prop R_s,N continuous} holds. Then, for every compact set $K \subset \hsigt$,
    \begin{equation*}
        \sup_{u\in K} \left| R_s(u)-R_{s,N}(u) \right| \tendsto{N \ra \infty} 0
    \end{equation*}
    In other words, $R_{s,N}$ converges to $R_s$ uniformly on compact sets of $\hsigt$.
\end{prop}

\begin{proof}[Proof of Proposition~\ref{prop approx R_s by R_s,N on compact sets} assuming Proposition~\ref{prop R_s,N continuous}] Here we assume the statements in Proposition~\ref{prop R_s,N continuous}. Thus, we invoke $\sigma<s-\frac{1}{2}$ close enough to $s-\frac{1}{2}$ along with the constant $C > 0$ from this proposition. Next, we observe that for all $N \in \N$:
    \begin{equation*}
        R_{s}(u) - R_{s,N}(u) = R_s(\Pi_N^{\perp}u)
    \end{equation*}
    Then, it follows from Proposition~\ref{prop R_s,N continuous} that:
    \begin{equation*}
        |R_{s}(u) - R_{s,N}(u)| = |R_s(\Pi_N^{\perp}u)| \leq C \hsignorm{\Pi_N^{\perp}u}^6
    \end{equation*}
    On the other hand, we have that for every $u \in \hsig$, $\norm{\Pi_N^\perp u}_{\hsig} \tendsto{N\ra \infty} 0$ and $\norm{\Pi_N^\perp}_{\hsig \ra \hsig} \leq 1$. So, using the following general abstract lemma in the inequality above finishes the proof:
    \begin{lem}\label{lem cvgce on compact set of bdd lin map}
        Let $E,F$ two Banach spaces. Let $\{ T_N\}_{N \in \N}$ be a sequence of bounded linear maps with the uniform (in $N \in \N$) bound $\norm{T_N}_{E \ra F} \leq M$, for some constant $M>0$. If for every $u \in E$:
        \begin{equation*}
           \norm{T(u)-T_N(u)}_F \tendsto{N \ra \infty} 0
        \end{equation*}
        for some linear map $T : E \ra F$, then for every compact set $K \subset E$, we have:
        \begin{equation*}
            \sup_{u \in K} \norm{T(u)-T_N(u)}_F \tendsto{N \ra \infty} 0
        \end{equation*}
    \end{lem}
\end{proof}

\section{Definition and deterministic properties of the modified energy derivative at 0}\label{section the modified energy derivative at 0}
In Section~\ref{section Poincaré-Dulac normal form reduction and modified energy}, we defined a quantity $Q_{s,N}$ (for $N \in \N \cup \{\infty \}$) in Definition~\ref{def normal form Q}, called the modified energy derivative at 0. We dedicate this section to the study of this quantity on $\hsigt$ for $\sigma < s-\frac{1}{2}$ (close enough to $s-\frac{1}{2}$), provided that $s>\frac{3}{2}$. In particular, we show that we are able to extend $Q_{s,\infty}$ (which is defined on $\cC^{\infty}(\T)$) to $\hsigt$, because we will prove that the right hand side in \eqref{def Qs,infty normal form} is actually the sum of three absolutely convergent sum for every  $u \in \hsigt$. This section is composed of two results; firstly, we state that for $\sigma<s-\frac{1}{2}$ close enough to $s-\frac{1}{2}$, the $Q_{s,N}$ are continuous functions on $\hsigt$ given by the diagonal of a sum of three continuous multi-linear forms; secondly, we state that $Q_{s,N}$ converges to $Q_{s,\infty}$ uniformly on compact sets of $\hsigt$. We point out that the analysis here is similar to the one from Section~\ref{section Deterministic properties of the energy correction}. \\

\begin{prop}\label{prop continuity of Qs,N and deterministic estimate}
    Let $s>\frac{3}{2}$. For $\sigma < s - \frac{1}{2}$ close enough to $s-\frac{1}{2}$, there exists a constant $C>0$ such that for every $u^{(1)},...,u^{(6)},v^{(1)},...,v^{(5)} \in \hsigt$ : 
    \begin{equation}\label{T0 with abs value}
        \sum_{\substack{\linc \\ \Omgz}} \big| \psi_{2s}(\vec{k}) \big| \big| u^{(1)}_{k_1}\cjg{u^{(2)}_{k_2}}...\cjg{u^{(6)}_{k_6}} \big| \leq C \prod_{j=1}^6 \norm{u^{(j)}}_{\hsig}
    \end{equation}
    and,
    \begin{equation}\label{T1 with abs value}
        \sum_{\substack{\linc \\ \lincba \\\Omgnz}} \big|\frac{\psi_{2s}(\vec{k})}{ \Omega(\vec{k})} \big| \big| v^{(1)}_{p_1}\cjg{v^{(2)}_{p_2}}...v^{(5)}_{p_5}\cjg{u^{(2)}_{k_2}}...\cjg{u^{(6)}_{k_6}} \big| \leq C \prod_{j \in \{2,...,6 \}} \norm{u^{(j)}}_{\hsig} \prod_{l=1}^5 \norm{v^{(l)}}_{\hsig}
    \end{equation}
    and, 
    \begin{equation}\label{T2 with abs value}
        \sum_{\substack{\linc \\ \lincbb \\\Omgnz}} \big| \frac{\psi_{2s}(\vec{k})}{ \Omega(\vec{k})} \big| \big| u^{(1)}_{k_1}\cjg{v^{(1)}_{q_1}}v^{(2)}_{q_2}...\cjg{v^{(5)}_{q_5}}u^{(3)}_{k_3}... \cjg{u^{(6)}_{k_6}}\big|\leq C \prod_{j \in \{1,3,...,6 \}} \norm{u^{(j)}}_{\hsig} \prod_{l=1}^5 \norm{v^{(l)}}_{\hsig}
    \end{equation}
    Hence, the maps:
    \begin{equation*}
        \function{\cT_0}{\hsigt^6}{\C}{u^{(1)},...,u^{(6)}}{\mathlarger{\sum_{\substack{\linc \\ \Omgz}} \psi_{2s}(\vec{k})u^{(1)}_{k_1}\cjg{u^{(2)}_{k_2}}...\cjg{u^{(6)}_{k_6}}}},
    \end{equation*}
    \begin{equation*}
        \function{\cT_1}{\hsigt^{10}}{\C}{u^{(2)},...,u^{(6)},v^{(1)},...,v^{(5)}}{\mathlarger{ \sum_{\substack{\linc \\ \lincba \\\Omgnz}} \frac{\psi_{2s}(\vec{k})}{ \Omega(\vec{k})} v^{(1)}_{p_1}\cjg{v^{(2)}_{p_2}}...v^{(5)}_{p_5}\cjg{u^{(2)}_{k_2}}...\cjg{u^{(6)}_{k_6}}}}
    \end{equation*}
    \begin{equation*}
        \function{\cT_2}{\hsigt^{10}}{\C}{u^{(1)},u^{(3)},...,u^{(6)},v^{(1)},...,v^{(5)}}{\mathlarger{\sum_{\substack{\linc \\ \lincbb \\\Omgnz}} \frac{\psi_{2s}(\vec{k})}{ \Omega(\vec{k})} u^{(1)}_{k_1}\cjg{v^{(1)}_{q_1}}v^{(2)}_{q_2}...\cjg{v^{(5)}_{q_5}}u^{(3)}_{k_3}... \cjg{u^{(6)}_{k_6}}}}
    \end{equation*}
    are continuous multi-linear forms. Then, for $j=0,1,2$, setting for $u\in \hsigt$:
    \begin{equation*}
        \cQ_j(u) := \cT_j(u,...,u)
    \end{equation*}
    we deduce that each $Q_j$ is continuous on $\hsigt$, and that for $N \in \N \cup \{ \infty \}$, the map\footnote{By abuse of notation, we still denote by $Q_{s,N}$ this new function, even though it has already been defined in Definition~\ref{def normal form Q}. The $Q_{s,N}$ of Definition~\ref{def normal form Q} and the $Q_{s,N}$ of this proposition coincide on $\cC^{\infty}(\T)$.}:
    \begin{equation}\label{def Qs,N = Q0 + Q1 +Q2}
       Q_{s,N} := \Im(-\frac{1}{6}\cQ_0 \circ \Pi_N  + \frac{1}{2}\cQ_1 \circ \Pi_N - \frac{1}{2}\cQ_2\circ \Pi_N)
    \end{equation}
    is continuous on $\hsigt$ and satisfies for all $u,v \in \hsigt$:
        \begin{equation}\label{Qs,N almost Lipschitz}
    \begin{split}
       \left| Q_{s,N}(u) - Q_{s,N}(v)\right| & \leq C \hsignorm{u-v}(\hsignorm{u}^6+\hsignorm{v}^6+\hsignorm{u}^9+\hsignorm{v}^9) \\
        & \leq C \hsignorm{u-v}(1+\hsignorm{u} + \hsignorm{v})^9
    \end{split}
    \end{equation}
    uniformly in $N\in \N \cup \{\infty \}$.
\end{prop}

 \begin{notn}
     We also use the notation $Q_s$ to refer to $Q_{s,\infty}$. 
 \end{notn}
 
\begin{rem}
 Note that we cannot a priori define $Q_{s}(u)$ on the support of $\mu_s$ as $\frac{d}{dt}E_{s}(\Phi(t)u)|_{t=0}$, because the expression
\begin{equation*}
    E_s(\Phi(t)u) = \frac{1}{2} \triplenorm{\Phi(t)u}_{H^s(\T)}^2 + R_s(\Phi(t)u)
\end{equation*}
is a priori ill defined for initial data $u$ in the support of $\mu_s$ since:
\begin{equation*}
    \triplenorm{\Phi(t)u}_{H^s(\T)}^2 = +\infty, \hspace{0.3cm} \text{$\mu_s$-almost surely.}
\end{equation*}
Indeed, for $\sigma < s- \frac{1}{2}$, it follows from the fact that the flow $\Phi(t)$ is a bijection from $\hsigt$ to itself that:
\begin{equation*}
	\mu_s(\{u \in \hsigt: \hsp \triplenorm{\Phi(t)u}_{H^s(\T)}^2 < +\infty \}) = \mu_s(\{v \in \hsigt: \hsp \triplenorm{v}_{H^s(\T)}^2 < +\infty \})= \mu_s(H^s(\T)) = 0
\end{equation*}
\end{rem}

We postpone the proof of Proposition~\ref{prop continuity of Qs,N and deterministic estimate} for Section~\ref{section Proofs of the deterministic properties} where a detailed analysis is provided. \\
Instead, we prove now the following approximation result:
\begin{prop}\label{prop Qs,N tends to Qs uniformly on compact sets}
    Let $s>\frac{3}{2}$. Let $\sigma < s-\frac{1}{2}$ close enough to $s-\frac{1}{2}$ so that the conclusion of Proposition \ref{prop continuity of Qs,N and deterministic estimate} holds. Then, for every compact set $K \subset \hsigt$,
    \begin{equation*}
        \sup_{u \in K} \left| Q_s(u) - Q_{s,N}(u) \right| \tendsto{N \ra \infty} 0
    \end{equation*}
    In other words, $Q_{s,N}$ converges to $Q_s$ uniformly on compact sets of $\hsigt$.
\end{prop}

\begin{proof}[Proof of Proposition~\ref{prop Qs,N tends to Qs uniformly on compact sets}] Assuming the statements in Proposition~\ref{prop continuity of Qs,N and deterministic estimate}, the proof goes exactly the same as the proof of Proposition~\ref{prop approx R_s by R_s,N on compact sets}.
\end{proof}

\section{Transport of Gaussian measures under the truncated flow}\label{section Transport of Gaussian measures under the truncated flow}
In this section, we prove that the transported measure $\Phi_N(t)_\# \mu_s$ is absolutely continuous with respect to $\mu_s$. To do so, we directly establish an explicit formula for the Radon-Nikodym derivative of $\Phi_N(t)_\# \mu_s$ with respect to $\mu_s$. Our method relies on a change-of-variable formula. Analogous change-of-variables have already been used: see for example \cite{oh2017quasi} Proposition 6.6, or \cite{tzvetkov2015quasiinvariant} Section 4. 
 
 \subsection{A change-of-variable formula}
We dedicate this paragraph to this change-of-variable. It is equivalent to the statement that the truncated flow preserves a certain measure. \\
On the Euclidean space $E_N$, equipped with the orthonormal basis $\{ e^{ikx}\}_{|k|\leq N}$, we consider the Lebesgue measure $\prod_{|n|\leq N} d\hat{u}_k$, where $d\hat{u}_k$ is the Lebesgue measure on $\C \cdot e^{ikx}$.

\begin{prop}\label{prop trc invarianve of leb-low times Gauss-high}
The measure 
    \begin{equation*}
        \prod_{|k| \leq N} d \hat{u}_k \otimes \mu_{s,N}^{\perp}
    \end{equation*}
    is invariant under the truncated flow. In other words, for every $N \in \N$ and $t \in \R$, 
    \begin{equation*}
        (\Phi_N(t))_{\#} \Bigg( \prod_{|k| \leq N} d \hat{u}_k \otimes \mu_{s,N}^{\perp} \Bigg) = \prod_{|k| \leq N} d \hat{u}_k \otimes \mu_{s,N}^{\perp}
    \end{equation*}
\end{prop}
A reformulation of this proposition is:
\begin{cor}[change-of-variable formula]\label{change-of-variable formula}
    Let $A$ a Borel measurable set. Let $f : \hsigt \lra \R_+ $ be a positive measurable function.
    Then, 
    \begin{equation*}
        \int_{\Phi_N(t)A} f(v) \prod_{|k| \leq N} d \hat{v}_k \otimes \mu_{s,N}^{\perp}(v) = \int_A f(\Phi_N(t)u) \prod_{|k| \leq N} d \hat{u}_k \otimes \mu_{s,N}^{\perp}(u)
    \end{equation*}
\end{cor}

\begin{proof}[Proof of Proposition~\ref{prop trc invarianve of leb-low times Gauss-high}] Let $\tld{\Phi}_N(t)$ be the restriction of the truncated flow $\Phi_N(t)$ on the finite-dimensional space $E_N$, which maps $E_N$ to itself (see Proposition~\ref{structure of the truncated flow}). We factorize $\Phi_N(t)$ as $(\tld{\Phi}_N(t), e^{it\p_x^2})$ on $E_N \times E_N^{\perp}$, in the sense that $\Phi_N(t)u_0 = \tld{\Phi}_N(t) \Pi_N u_0 + e^{it\p_x^2} \Pi_N^{\perp} u_0$ (see again Proposition~\ref{structure of the truncated flow}). Now, we provide a proof in two steps.
\bigskip

\underline{Step 1 :} Firstly, we show the formula 
\begin{equation}\label{formula truncated transport}
    (\Phi_N(t))_{\#} \Bigg( \prod_{|k| \leq N} d \hat{u}_k \otimes \mu_{s,N}^{\perp} \Bigg) = \Bigg( \prod_{|k| \leq N} d \hat{u}_k\Bigg) \otimes (e^{it\p_x^2})_{\#} \mu_{s,N}^{\perp}
\end{equation}
To do so, we rely on the fact $\cB(\hsigt) = \cB(E_N) \otimes \cB(E_N^{\perp})$\footnote{Indeed, whenever $X$ and $Y$ are two topological separable spaces, we have $\cB(X \times Y) = \cB(X) \otimes \cB(Y)$, where $\otimes$ is the symbol for tensor-product of sigma-algebra}. Thanks to that point, if we show that for every $A_{low} \in \cB(E_N)$ and $A_{high} \in \cB(E_N^\perp)$, we have: 
\begin{equation}\label{equality on Alow x Ahigh}
        \Big( \prod_{|k| \leq N} d \hat{u}_k \otimes \mu_{s,N}^{\perp} \Big) \big( \Phi_N(t)^{-1}(A_{low} \times A_{high})\big)
         = \Big( \prod_{|k| \leq N} d \hat{u}_k\Big)( A_{low})\mu_{s,N}^{\perp}(e^{-it\p_x^2}A_{high}), 
\end{equation}
then, the property of uniqueness of product measures\footnote{Both $\prod_{|k| \leq N} d \hat{u}_k$ and  $(e^{it\p_x^2})_{\#} \mu_{s,N}^{\perp}$ being $\sigma$-finite measures.} will ensure that the two measures in \eqref{formula truncated transport} coincide on $\cB(E_N) \otimes \cB(E_N^{\perp})$, that is on $\cB(\hsigt)$. Hence, let us prove \eqref{equality on Alow x Ahigh}. Let $A_{low} \in \cB(E_N)$ and $A_{high} \in \cB(E_N^\perp)$. On the one hand, from the factorization of the truncated flow, we have:
\begin{equation*}
    \Phi_N(t)^{-1}(A_{low} \times A_{high}) = \tld{\Phi}_N(t)^{-1}(A_{low}) \times e^{-it\p_x^2}A_{high} \subset E_N \times E_N^\perp
\end{equation*}
so, using the definition of product measures we obtain:
\begin{equation*}
        \Big( \prod_{|k| \leq N} d \hat{u}_k \otimes \mu_{s,N}^{\perp} \Big) \left( \Phi_N(t)^{-1}(A_{low}   \times A_{high})\right) = \Big( \prod_{|k| \leq N} d \hat{u}_k\Big)(\tld{\Phi}_N(t)^{-1}(A_{low}))\mu_{s,N}^{\perp}(e^{-it\p_x^2}A_{high})
\end{equation*}
Thus, to prove \eqref{equality on Alow x Ahigh}, it remains to show that:
\begin{equation}\label{invariance of the Lebesgue measure under the truncated flow}
    \Big( \prod_{|k| \leq N} d \hat{u}_k\Big)\big(\tld{\Phi}_N(t)^{-1}(A_{low})\big) = \Big( \prod_{|k| \leq N} d \hat{u}_k\Big)\big(A_{low}\big)
\end{equation}
To do so, we recall that on the other hand, $\tld{\Phi}_N(t)$ is the flow of the Hamiltonian equation: 
\begin{equation}
    \begin{cases}
        i\p_t u = \frac{\p H_N}{\p \cjg{u}}(u) \\
        u|_{t=0} = u_0 \in E_N
    \end{cases}
    \tag{FNLS}
\end{equation}
on the finite-dimensional space $E_N$ (see Proposition~\ref{structure of the truncated flow}). So the equality \eqref{invariance of the Lebesgue measure under the truncated flow} follows from the application of Liouville's theorem, which states that the flow of finite-dimensional Hamiltonian equation preserves the Lebesgue measure. To sum up, we obtained \eqref{formula truncated transport}. Let us now turn to the second step of the proof:
\bigskip

\underline{Step 2 :} Secondly, we show the following invariance property:
\begin{equation*}
    (e^{it\p_x^2})_{\#} \mu_{s,N}^{\perp} = \mu_{s,N}^{\perp}
\end{equation*}
We recall that the probability measure $\mu_{s,N}^{\perp}$ is the law of the random variable
\begin{equation*}
    \begin{split}
   X : \hsp &  \Omega \lra \hsigt \\
   & \omega \longmapsto \sum_{|n| > N} \frac{g_n(\omega)}{\langle n \rangle^s} e^{inx}
    \end{split}
\end{equation*}
where the sum converges in the space $\L^2(\Omega,\hsigt)$. Since $e^{it\p_x^2}$ is a linear isometry on $\L^2(\T)$, we deduce that 
\begin{equation*}
    e^{it\p_x^2} \circ X (\omega) = \sum_{|n| > N} \frac{e^{-itn^2}g_n(\omega)}{\langle n \rangle^s} e^{inx}
\end{equation*}
where the sum still converges in $\L^2(\Omega,\hsigt)$. Moreover, the Gaussian  measures are invariant under rotations, so the family $\{e^{-itn^2}g_n \}_{n \in \Z}$ is still a family of independent standard complex Gaussian measures. Consequently, we obtain that $X$ and $e^{it\p_x^2} \circ X$ have the same law, and this means that
\begin{equation*}
    (e^{it\p_x^2})_{\#} \mu_{s,N}^{\perp} = \mu_{s,N}^{\perp}
\end{equation*}
which is the invariance property of Step 2.
\bigskip

\underline{Conclusion :} We get the desired result by combining Step 1 and Step 2.
\end{proof}

%%%%%%%%%%%%%%%%%%%%%%%%%%%%%%%%%%%%%%%%%%%%%%%%%%%%%%%%%%%%%%%%%%%%%%%%%%%%%%%%%%%%%%%%%%%%%%%%%%%%%%%%%%%%%%%%%%%%%%%%%%%%%%%%%%%%%%

\subsection{The Radon-Nikodym derivative for the truncated transported Gaussian measures} 
In this paragraph, we fix \underline{$N \in \N$}. We will use Proposition~\ref{prop trc invarianve of leb-low times Gauss-high} in order to obtain the Radon-Nikodym derivative of $\Phi_N(t)_\# \mu_s$ with respect to $\mu_s$.

\begin{prop}\label{prop Radon-Nikodym derivative for the trc trspted GM}
    Let $s > \frac{3}{2}$, $R>0$ and $N \in \N$. For every $t \in \R$, we have
    \begin{equation*}
        \Phi_N(t)_\# \mu_s = \textnormal{exp}\bigl(-\frac{1}{2} (\norm{\Pi_N \Phi_N(-t)u}_{H^s}^2-\norm{\Pi_N u}_{H^s}^2 )   \bigr) d\mu_s
    \end{equation*}
    Moreover, we can rewrite this formula as
    \begin{equation*}
        \Phi_N(t)_\# \mu_s = \textnormal{exp}\left(R_{s,N}(\Phi_N(-t)u) - R_{s,N}(u)-\int_0^{-t} Q_{s,N}(\Phi_N(\tau)u)d\tau \right) d\mu_s
    \end{equation*}
    where $R_{s,N}$ is defined in Definition~\ref{def normal form R and E} (see also \eqref{R_s,N diag multilinear form}) and $Q_{s,N}$ is defined in \eqref{def Qs,N = Q0 + Q1 +Q2}  (see also Definition~\ref{def normal form Q}).
\end{prop}

\begin{rem}
    Such a formula for the density has also been obtained for the different models in \cite{debussche2021quasi}, \cite{forlano_and_soeng2022transport}, \cite{genovese2023transport} and \cite{forlano2022quasi}.
\end{rem}

\begin{proof}[Proof of Proposition~\ref{prop Radon-Nikodym derivative for the trc trspted GM}] Firstly, we decompose $\mu_s$ as:
\begin{equation*}
   d \mu_s = \left( \frac{1}{Z_N}e^{-\frac{1}{2}\norm{\Pi_N u}_{H^s}^2} \prod_{|k| \leq N} d \hat{u_k} \right) \otimes d\mu_{s,N}^{\perp}
\end{equation*}
(see \eqref{mu_s = mu_s,N times mu_s,N^perp} and \eqref{mu_s,N}).
    Then, thanks to a general feature for the transport of density measure, we have :
     \begin{equation*}
        \begin{split}
            \Phi_N(t)_\# d\mu_s &= \Phi_N(t)_\# \left( \frac{1}{Z_N}e^{-\frac{1}{2}\norm{\Pi_N u}_{H^s}^2} \prod_{|k| \leq N} d \hat{u_k} \right) \otimes d\mu_{s,N}^{\perp} \\
            & =  \frac{1}{Z_N}e^{-\frac{1}{2}\norm{\Pi_N(\Phi_N(t)^{-1}u)}^2_{H^s}} \cdot \Phi_N(t)_\# \left( \prod_{|k| \leq N} d \hat{u_k} \right) \otimes d\mu_{s,N}^{\perp} \\
            & =  \frac{1}{Z_N}e^{-\frac{1}{2}\norm{\Pi_N(\Phi_N(-t)(u)}^2_{H^s}} \cdot  \left( \prod_{|k| \leq N} d \hat{u_k} \right) \otimes d\mu_{s,N}^{\perp} \\
            &= e^{-\frac{1}{2}\left(\norm{\Pi_N(\Phi_N(-t)(u)}^2_{H^s} -  \norm{\Pi_N u}^2_{H^s} \right)} d\mu_s
        \end{split}
    \end{equation*}
    where we used in the third line the invariance property from Proposition~\ref{prop trc invarianve of leb-low times Gauss-high}. Hence, the first statement of Proposition~\ref{prop Radon-Nikodym derivative for the trc trspted GM} is proven. To achieve the proof, we rewrite the term inside the exponential thanks to the definition of the modified energy \eqref{modified energy} and the identity \eqref{Qs,N(phi_N(t)u)}:
    \begin{equation*}
        \begin{split}
        -\frac{1}{2}\big(\norm{\Pi_N(\Phi_N(-t)(u)}^2_{H^s} &-  \norm{\Pi_N u}^2_{H^s} \big) = \int_0^{-t} -\frac{1}{2}\frac{d}{d\tau} \norm{\Pi_N(\Phi_N(\tau)(u)}^2_{H^s} d\tau \\
        & = \int_0^{-t} \big( \frac{d}{d\tau}\left( R_{s,N}(\Phi_N(\tau)u) \right) -\frac{d}{d\tau}( E_{s,N}(\Pi_N \Phi_N(\tau)u)) \big) d\tau \\
        & = \int_0^{-t} \big( \frac{d}{d\tau}\left( R_{s,N}(\Phi_N(\tau)u) \right)- Q_{s,N}(\Phi_N(\tau)u) \big) d\tau \\
        & = R_{s,N}(\Phi_N(-t)u) - R_{s,N}(u) - \int_0^{-t} Q_{s,N}(\Phi_N(\tau)u) d\tau
        \end{split} 
    \end{equation*}
    which is the desired rewriting.
\end{proof}

\section{Transport of Gaussian measures under the flow}\label{section Transport of Gaussian measures under the flow}
In Section~\ref{section Transport of Gaussian measures under the truncated flow}, we have seen that for every $t \in \R$, and every $N \in \N$:
\begin{equation}\label{trc trspted GM = density GM}
    \Phi_N(t)_\# \mu_s = G_{s,N}(t,.)\mu_s
\end{equation}
where,
\begin{equation}\label{bis density trc trspted GM}
    G_{s,N}(t,u)= \textnormal{exp} \left(R_{s,N}(\Phi_N(-t)u) -R_{s,N}(u)- \int_0^{-t} Q_{s,N}(\Phi_N(\tau)u) d\tau\right)
\end{equation} 
Our goal in this section is to "take the limit" $N \ra \infty$ in order to extend this formula to $N=\infty$. Thus, we invoke:
    \begin{equation}\label{density transported GM}
    G_s(t,u) := \textnormal{exp} \left(R_s(\Phi(-t)u)-R_s(u) - \int_0^{-t} Q_{s}(\Phi(\tau)u) d\tau\right)
\end{equation}
and we aim to show the following proposition:
\begin{prop}\label{prop Radon-Nikodym derivative for the transported GM}
    Let $s>\frac{3}{2}$ and $R>0$. Let $t \in \R$. Then,
    \begin{equation*}
        \Phi(t)_\# \mu_s = G_s(t,u) d\mu_s
    \end{equation*}
    In particular, $\mu_s$ is quasi-invariant under the flow of \eqref{NLS}.
\end{prop}

\begin{rem}\label{rem continuity of the densities}
    For $\sigma < s-\frac{1}{2}$ close enough to $s-\frac{1}{2}$, we deduce from Proposition~\ref{prop R_s,N continuous}, Proposition~\ref{prop continuity of Qs,N and deterministic estimate} (whose proofs are provided in Section~\ref{section Proofs of the deterministic properties}), and from the continuity properties of the flow, that the map :
    \begin{equation*}
        (t,u) \in \R \times \hsigt \mapsto \textnormal{exp} \left( R_s(\Phi(-t)u)-R_s(u) -\int_0^{-t} Q_{s,N}(\Phi_N(\tau)u) d\tau\right)
    \end{equation*}
    is continuous, for any $N\in \N \cup \{ \infty \}$. \\
    Moreover, on $\hsigt$, the a priori ill-defined object $-\frac{1}{2}(\norm{\Phi(-t)u}_{H^s}^2-\norm{u}_{H^s}^2)$ can be seen as the well-defined object:
    \begin{equation*}
        -\frac{1}{2}(\norm{\Phi(-t)u}_{H^s}^2-\norm{u}_{H^s}^2) := R_s(\Phi(-t)u)-R_s(u) - \int_0^{-t} Q_{s}(\Phi(\tau)u) d\tau
    \end{equation*}
\end{rem}
In this section, we assume the statements from Proposition~\ref{prop R_s,N continuous} and Proposition~\ref{prop continuity of Qs,N and deterministic estimate}. Hence, we work with a $\sigma < s-\frac{1}{2}$ close enough to $s-\frac{1}{2}$ so that the pointwise properties from Proposition~\ref{prop R_s,N continuous} and \ref{prop continuity of Qs,N and deterministic estimate} are satisfied.

\subsection{Approximation properties} Our main ingredients to "take the limit" $N \ra \infty$ in \eqref{trc trspted GM = density GM} is two approximation properties. The first one is the inner regularity satisfied by probability measures on $\hsigt$ :

\begin{prop}[inner regularity]\label{prop inner regularity of rho}Let $\mu$ be a finite measure on $(\hsigt,\cB(\hsigt))$. Then, for any Borel set $A \subset H^\sigma(\T)$, we have 
    \begin{equation*}
    \mu(A) = \sup \{ \mu(K) : \hsp K \subset A, \hsp K \hsp \textnormal{compact set in} \hsp \hsigt \}
\end{equation*}
\end{prop}

\begin{proof}
    This follows from the general fact that finite measures on Polish spaces are regular.
\end{proof}

As a consequence, we have:
\begin{cor}\label{cor mu(K) = nu(K) for all compact sets}
    Let $\mu$ and $\nu$ be two finite measures on $(\hsigt, \cB(\hsigt))$. Assume that for every compact set $K \subset \hsigt$, we have:
    \begin{equation*}
        \mu(K) = \nu(K)
    \end{equation*}
    Then, $\mu = \nu$.
\end{cor}

\begin{proof}
    Let $A \in \cB(\hsigt)$. Let us prove that $\mu(A) = \nu(A)$. Let $\eps > 0$. From the inner regularity of $\mu$ (see Proposition \ref{prop inner regularity of rho}), we invoke a compact set $K$ of $\hsigt$ such that $K \subset A$ and $\mu(A) - \eps \leq \mu(K)$. Thus,
    \begin{equation*}
        \mu(A) - \eps \leq \mu(K) = \nu(K) \leq \nu(A)
    \end{equation*}
    Since $\eps >0$ is arbitrary, we conclude that $\mu(A) \leq \nu(A)$. By interchanging the roles of $\mu$ and $\nu$, we obtain the converse inequality.
\end{proof}

The second approximation property that we will use is the approximation of the expected density $G_s(t,.)$ by the \textit{truncated densities} $G_{s,N}(t,.)$:

\begin{prop}\label{prop GsN cvg uniformly on cpcts to Gs} Let $s>\frac{3}{2}$. Let $K \subset \hsigt$ a compact set. Then, for every $t\in \R$:
\begin{equation*}
    \sup_{u \in K} \left|G_s(t,u) - G_{s,N}(t,u) \right| \tendsto{N\ra \infty} 0 
\end{equation*}
    In other words, $G_{s,N}(t,.)$ converges to $G_s(t,.)$ uniformly on compact sets.
\end{prop}

\begin{proof} Let $t\in \R$. Recall that $G_{s,N}(t,.)$ and $G_s(t,.)$ are respectively defined in \eqref{bis density trc trspted GM} and \eqref{density transported GM}. We prove separately that:
\begin{flushleft}
    \text{$(a):$ $ R_{s,N}(\Phi_N(t)u)-R_{s,N}(u)$ converges to $R_s(\Phi(t)u)-R_s(u)$ uniformly on compact sets as $N \ra \infty$}
\end{flushleft}
    and,
\begin{flushleft}
   \text{ $(b):$ $\int_0^{t} Q_{s,N}(\Phi_N(\tau)u) d\tau$ converges to $ \int_0^{t} Q_{s}(\Phi(\tau)u) d\tau$ uniformly on compact sets as $N \ra \infty$}. 
\end{flushleft}

Indeed, if we do so, we will obtain that: 
\begin{center}
$R_{s,N}(\Phi_N(t)u)-R_{s,N}(u) - \int_0^{t} Q_{s,N}(\Phi(_N\tau)u) d\tau $ converges to $R_s(\Phi(t)u)-R_s(u) - \int_0^{t} Q_{s}(\Phi(\tau)u) d\tau$ uniformly on compact sets as $N \ra \infty$.  
\end{center}

Then, since the exponential is continuous, this will lead to the result. In order not to repeat the same argument, and since $(a)$ is similar to $(b)$ and a little bit easier, we will only prove $(b)$. Then, let $K \subset \hsigt$ be a compact set. Then,
\begin{equation}\label{diffce integrals on compatc set}
    \begin{split}
        &| \int_0^{t} Q_{s}(\Phi(\tau)u) d\tau - \int_0^{t} Q_{s,N}(\Phi_N(\tau)u) d\tau | \\
        & \leq \int_0^{t} |Q_{s}(\Phi(\tau)u) - Q_{s,N}(\Phi(\tau)u) | d\tau +  \int_0^{t} | Q_{s,N}(\Phi(\tau)u) - Q_{s,N}(\Phi_N(\tau)u) |
        d\tau \\
        & \leq |t| \sup_{\tau \in [0,t]} |Q_{s}(\Phi(\tau)u) - Q_{s,N}(\Phi(\tau)u)| + |t| \sup_{\tau \in [0,t]} |Q_{s,N}(\Phi(\tau)u) - Q_{s,N}(\Phi_N(\tau)u)|
    \end{split}
\end{equation}
On the one hand, from the continuity of $\Phi: \R \times \hsig \ra \hsig$, the set $\{\Phi(\tau)u; \hsp \tau \in [0,t], u \in K\}$ is compact. Combining this with the fact that $Q_{s,N} \tendsto{N} Q_s$ uniformly on compact sets (see Proposition~\ref{prop Qs,N tends to Qs uniformly on compact sets}) yields:
\begin{equation*}
   \sup_{u \in K} \sup_{\tau \in [0,t]} |Q_{s}(\Phi(\tau)u) - Q_{s,N}(\Phi(\tau)u)| \tendsto{N \ra \infty} 0
\end{equation*}
On the other hand, using \eqref{Qs,N almost Lipschitz} from Proposition~\ref{prop continuity of Qs,N and deterministic estimate} we can invoke a constant $C_s>0$ independent of $N$ such that:
\begin{equation}\label{Qs,N(phi) - Qs,N(phi_N) on compact sets}
    \begin{split}
   &\sup_{u \in K} \sup_{\tau \in [0,t]} |Q_{s,N}(\Phi(\tau)u) - Q_{s,N}(\Phi_N(\tau)u)| \\
   & \leq C_s \sup_{(\tau,u)\in [0,t]\times K}\hsignorm{\Phi(\tau)u-\Phi_N(\tau)u}\left(1 + \hsignorm{\Phi(\tau)u}+\hsignorm{\Phi_N(\tau)u}\right)^9 
   \end{split}
\end{equation}
Besides, from the Cauchy theory, there exists a constant $C>0$ independent of $N$ such that:
\begin{equation*}
    \sup_{(\tau,u)\in [0,t]\times K} \hsignorm{\Phi(\tau)u} + \hsignorm{\Phi_N(\tau)u} \leq C
\end{equation*}
(see Proposition~\ref{prop exponential bound}). And, from Proposition~\ref{appendix second approximation prop}, we also have:
\begin{equation*}
    \sup_{(\tau,u)\in [0,t]\times K}\hsignorm{\Phi(\tau)u-\Phi_N(\tau)u} \tendsto{N \ra \infty} 0
\end{equation*}
Using these two facts in \eqref{Qs,N(phi) - Qs,N(phi_N) on compact sets} yields:
\begin{equation*}
    \sup_{(\tau,u)\in [0,t]\times K}\ |Q_{s,N}(\Phi(\tau)u) - Q_{s,N}(\Phi_N(\tau)u)| \tendsto{N \ra \infty} 0 
\end{equation*}
Finally, coming back to \eqref{diffce integrals on compatc set}, we conclude that:
\begin{equation*}
    \sup_{u \in K}| \int_0^{t} Q_{s}(\Phi(\tau)u) d\tau - \int_0^{t} Q_{s,N}(\Phi_N(\tau)u) d\tau | \tendsto{N \ra \infty} 0
\end{equation*}
This completes the proof.
\end{proof}

\subsection{The Radon-Nikodym derivative for the transported Gaussian measure}

In this paragraph, we prove Proposition~\ref{prop Radon-Nikodym derivative for the transported GM} based on the combination of \eqref{trc trspted GM = density GM} with the approximation properties above.

\begin{rem}
    We will be able to use Corollary~\ref{cor mu(K) = nu(K) for all compact sets} with the measures $\Phi(t)_\# \mu_s$ and $G_s(t,.) d\mu_s$. Indeed, both are finite measures on $(\hsigt, \cB(\hsigt))$. On the one hand, $\Phi(t)_\# \mu_s$ is a probability measure; on the other hand, Fatou's lemma provide the following a priori bound for $G_s(t,.) d\mu_s(\hsigt)$ :
    \begin{equation*}
        G_s(t,.) d\mu_s(\hsigt) = \int_{\hsig} G_s(t,u)d\mu_s =\int_{\hsig} \lim_N G_{s,N}(t,u)d\mu_s \leq \liminf_N \int_{\hsig}G_{s,N}(t,u)d\mu_s = 1
    \end{equation*}
\end{rem}

\begin{rem}
    The proof we provide below for Proposition \ref{prop Radon-Nikodym derivative for the transported GM} is similar to the one for Theorem 1.4 in \cite{planchon2022modified}. However, it is worth noting that our proof do not require any $\L^p$-integrability for the truncated densities $G_{s,N}(t,.)$.
\end{rem}

\begin{proof}[Proof of Proposition~\ref{prop Radon-Nikodym derivative for the transported GM}] Let $t \in \R$. Relying on Corollary \ref{cor mu(K) = nu(K) for all compact sets}, it suffices to prove that for every compact set $K$ of $\hsigt$, we have:
\begin{equation}\label{it suffices equality on K}
    G_s(t,u) d\mu_s (K) = \Phi(t)_\# \mu_s (K), \hspace{0.2cm} \text{that is}, \hspace{0.2cm} \int_K G_s(t,u) d\mu_s = \mu_s(\Phi(-t) K) 
\end{equation}
Fix $K$ a compact of $\hsigt$.\\
We invoke two real numbers $\sigma_1$ and $\sigma_2$ such that $\sigma < \sigma_1 < \sigma_2 < s -\frac{1}{2}$. Moreover, we invoke for $k \in \N$:
\begin{equation*}
    B^{H^{\sigma_2}}_k := \{ u \in H^{\sigma_2}(\T): \hsp \norm{u}_{H^{\sigma_2}(\T)} \leq k  \}
\end{equation*}
the closed centered ball in $H^{\sigma_2}(\T)$ of radius $k$. Note that $ B^{H^{\sigma_2}}_k$ is compact in $H^{\sigma_1}(\T)$ because $\sigma_1 < \sigma_2$. To establish \eqref{it suffices equality on K}, it suffices to prove that for all $k \in \N$:
\begin{equation}\label{it suffices equality on K cap Ball stronger topology}
    \int_{K \cap B^{H^{\sigma_2}}_k} G_s(t,u) d\mu_s = \mu_s(\Phi(-t) (K \cap B^{H^{\sigma_2}}_k) )
\end{equation}
Indeed, if we do so, we will obtain that:
\begin{equation*}
    \begin{split}
    &\int_{K \cap H^{\sigma_2}(\T)} G_s(t,u) d\mu_s  = \int_{ \bigcup_{k \in \N} (K \cap B^{H^{\sigma_2}}_k )} G_s(t,u) d\mu_s  = \lim_{k \ra \infty} \nnearrow \int_{K \cap B^{H^{\sigma_2}}_k} G_s(t,u) d\mu_s \\
    & = \lim_{k \ra \infty} \nnearrow  \mu_s(\Phi(-t) (K \cap B^{H^{\sigma_2}}_k) ) = \mu_s \big( \bigcup_{k \in \N} \Phi(-t) (K \cap B^{H^{\sigma_2}}_k) \big) = \mu_s\big(  \Phi(-t) (\bigcup_{k \in \N}(K \cap B^{H^{\sigma_2}}_k)) \big) \\
    & = \mu_s\big(\Phi(-t)(K \cap H^{\sigma_2}(\T))\big)
    \end{split}
\end{equation*}
Besides, since $\mu_s(\hsigt \setminus  H^{\sigma_2}(\T)) = 0$, we have: 
\begin{equation*}
    \int_K G_s(t,u) d\mu_s = \int_{K \cap H^{\sigma_2}(\T)} G_s(t,u) d\mu_s
\end{equation*}
and,
\begin{equation*}
    \mu_s(\Phi(-t)K) = \mu_s\big((\Phi(-t)K) \cap H^{\sigma_2}(\T)\big) = \mu_s\big(\Phi(-t)(K\cap H^{\sigma_2}(\T))\big)
\end{equation*}
Hence, if we prove \eqref{it suffices equality on K cap Ball stronger topology}, then we will obtain \eqref{it suffices equality on K}, and the proof will be complete. Thus, we now move on to the proof of \eqref{it suffices equality on K cap Ball stronger topology}. Let $k\in \N$ and denote $K_2 := K \cap B^{H^{\sigma_2}}_k $.\\

-- Firstly, we prove that:
\begin{equation}\label{first ineq}
    \int_{K_2} G_s(t,u) d\mu_s \leq \mu_s(\Phi(-t) (K_2) )
\end{equation}
Since $K_2$ is compact in $H^{\sigma}(\T)$, we obtain from Proposition \ref{prop GsN cvg uniformly on cpcts to Gs} that:
\begin{equation}\label{int_K G_s = lim int_K Gs,N}
    \int_{K_2} G_s(t,u) d\mu_s = \lim_{N \ra \infty} \int_{K_2} G_{s,N}(t,u) d\mu_s = \lim_{n\ra \infty} \Phi_N(t)_\# \mu_s (K_2) = \lim_{N \ra \infty} \mu_s(\Phi_N(-t)K_2)
\end{equation}
Let $\eps >0$. Thanks to Corollary \ref{appendix set approximation}, we invoke $N_0 \in \N$ such that:
    \begin{equation*}
      N\geq N_0 \implies  \Phi_N(-t)(K_2) \subset \Phi(-t)(K_2) +  B^{H^{\sigma}}_{\eps}
    \end{equation*}
    where $ B^{H^{\sigma}}_{\eps}$ is the closed centered ball in $\hsigt$ of radius $\eps$. As a consequence, we have:
    \begin{equation*}
        \limsup_N \mu_s(\Phi_N(-t)K_2) \leq \mu_s(\Phi(-t)(K_2) +  B^{H^{\sigma}}_{\eps})
    \end{equation*}
    Plugging this into \eqref{int_K G_s = lim int_K Gs,N}, we obtain:
    \begin{equation*}
         \int_{K_2} G_s(t,u) d\mu_s \leq \mu_s(\Phi(-t)(K_2) +  B^{H^{\sigma}}_{\eps})
    \end{equation*}
    Since $\eps>0$ is arbitrary, we deduce that\footnote{Since $\Phi(-t) K_2$ is closed in $\hsigt$, we have $\bigcap_{\eps > 0}(\Phi(-t)(K_2) +  B^{H^{\sigma}}_{\eps}) = \Phi(-t) K_2$}:
        \begin{equation*}
         \int_{K_2} G_s(t,u) d\mu_s \leq \lim_{\eps \ra 0} \ssearrow \mu_s(\Phi(-t)(K_2) +  B^{H^{\sigma}}_{\eps})  = \mu_s\big(\bigcap_{\eps > 0}(\Phi(-t)(K_2) +  B^{H^{\sigma}}_{\eps})\big) = \mu_s(\Phi(-t) K_2)
    \end{equation*}
    So we have proven \eqref{first ineq}. \\
    
-- Secondly, we prove that:
\begin{equation}\label{second ineq}
    \int_{K_2} G_s(t,u) d\mu_s \geq \mu_s(\Phi(-t) (K_2) )
\end{equation}
Let us first observe that $K_2$ is compact in $H^{\sigma_1}(\T)$ : if $\{ u_n\}_{n \in \N} \in K_2^\N$ is a sequence in $K_2 = K \cap B^{H^{\sigma_2}}_k $, then, from the compactness of $B^{H^{\sigma_2}}_k$ in $H^{\sigma_1}(\T)$ (because $\sigma_1 < \sigma_2$), there exists a subsequence $\{ u_{n_j}\}_{j \in \N}$ and an element $u \in B^{H^{\sigma_2}}_k$ such that:
\begin{equation*}
    \norm{u_{n_j}-u}_{H^{\sigma_1}} \tendsto{j \ra \infty} 0
\end{equation*}
In particular, we have $\norm{u_{n_j}-u}_{H^{\sigma}}\tendsto{j \ra \infty} 0$ because $\sigma < \sigma_1 $. Since $K$ is closed in $\hsigt$, it implies that $u \in K$. Then $u \in K \cap B^{H^{\sigma_2}}_k$, and $K_2$ is compact in $H^{\sigma_1}(\T)$. \\

Now, let $\eps >0$. Thanks to Corollary \ref{appendix set approximation}, and the fact that $K_2$ is compact in $H^{\sigma_1}(\T)$ , we are able to invoke $N_1 \in \N$ such that:
    \begin{equation*}
       N \geq N_1 \implies  \Phi(-t)(K_2) \subset \Phi_N(-t)(K_2 + B^{H^{\sigma_1}}_{\eps})
    \end{equation*}
    where $ B^{H^{\sigma_1}}_{\eps}$ is the closed centered ball in $H^{\sigma_1}(\T)$ of radius $\eps$. It is now important to notice that  $K_2 + B^{H^{\sigma_1}}_{\eps}$ is compact in $\hsigt$. It follows from the fact that both $K_2$ and $B^{H^{\sigma_1}}_{\eps}$ are compact in $\hsigt$. As a consequence, we obtain from Proposition \ref{prop GsN cvg uniformly on cpcts to Gs} that for $N \geq N_1$: 
    \begin{equation*}
        \mu_s(\Phi(-t) (K_2)) \leq \mu_s( \Phi_N(-t)(K_2 + B^{H^{\sigma_1}}_{\eps})) = \int_{K_2 + B^{H^{\sigma_1}}_{\eps}} G_{s,N}(t,u) d\mu_s \tendsto{N \ra \infty} \int_{K_2 + B^{H^{\sigma_1}}_{\eps}} G_s(t,u) d\mu_s 
    \end{equation*}
    Since $\eps >0$ is arbitrary, we obtain that:
    \begin{equation*}
        \mu_s(\Phi(-t) (K_2)) \leq \lim_{\eps \ra 0} \ssearrow \int_{K_2 + B^{H^{\sigma_1}}_{\eps}} G_s(t,u) d\mu_s  = \int_{K_2} G_s(t,u) d\mu_s 
    \end{equation*}
    So we have proven \eqref{second ineq}. This completes the proof of Proposition \ref{prop Radon-Nikodym derivative for the transported GM}.
\end{proof}

\section{Weighted Gaussian measures, \texorpdfstring{$L^p$}{TEXT}-estimates on the weight, and transport along the flows}\label{section weighted Gaussian measures}
In Section~\ref{section Transport of Gaussian measures under the flow}, we proved that for every $s>\frac{3}{2}$ and every $t \in \R$, we have:
\begin{equation*}
    \Phi(t)_\# \mu_s = G_s(t,u) \mu_s
\end{equation*}
where $G_s(t,u)$ is the continuous function on $\hsigt$ (for a given $\sigma<s-\frac{1}{2}$ close enough to $s-\frac{1}{2}$) given by:
\begin{equation*}
     G_s(t,u) = \textnormal{exp} \left(R_s(\Phi(-t)u)-R_s(u) - \int_0^{-t} Q_{s}(\Phi(\tau)u) d\tau\right)
\end{equation*}
In particular, it implies that $G_s(t,.)$ belongs to $\L^1(d\mu_s)$ (with a $\L^1(d\mu_s)$-norm equal to $1$). Thus, it is interesting to wonder if  $G_s(t,.)$ belongs to $\L^p(d\mu_s)$ for $p \in (1,\infty)$. We will see that the answer is positive for all $p \in (1,\infty)$ if we restrict $\mu_s$ on bounded sets of $H^1(\T)$, that is at the level where we can make use of the conservation of the Hamiltonian and of the $\L^2(\T)$-norm. To do so, we need to introduce \textit{weighted Gaussian measures}. \\

In Section~\ref{section Poincaré-Dulac normal form reduction and modified energy}, we identified a modified energy in Definition~\ref{def normal form R and E}. Based on this modified energy, we define the weighted Gaussian measures. Formally, the idea is to replace the Gaussian measure $\frac{1}{Z_s}e^{-\frac{1}{2}\norm{u}_{H^s(\T)}^2}du$ by $\frac{1}{Z'_s}e^{-E_s(u)}du$. However, we need to add a cut-off at the energy level, where the Hamiltonian and the $\L^2(\T)$-norm are conserved by the flow. In this section, we introduce the weighted Gaussian measures as density measures with respect to the Gaussian measure $\mu_s$. We also provide $\L^p$-estimates on these densities, ensuring in particular that the weighted Gaussian measure are well-defined probability measures on $\hsigt$, $\sigma < s-\frac{1}{2}$. 

\subsection{Definitions}
We start by invoking the following quantity :
\begin{equation*}
    \cC(u) := \frac{1}{2} \norm{u}_{\L^2}^2 + H(u) 
\end{equation*}
which is conserved by the flow of \eqref{NLS}.
Next, for $R > 0$, and for every $N \in \N$, we define the weighted Gaussian measures as  
\begin{align}\label{weighted Gaussian measure}
    d\rho_{s,R,N} &:= \frac{1}{Z_{s,R,N}}\1_{ \{\cC(u) \leq R \}} e^{-R_{s,N}(u)} d\mu_s, & d\rho_{s,R} &:= \frac{1}{Z_{s,R}} \1_{ \{\cC(u) \leq R \}} e^{-R_s(u)} d\mu_s
\end{align}
where,
\begin{align*}
    Z_{s,R,N} &:= \int_{H^{s-\frac{1}{2}-}} \1_{ \{\cC(u) \leq R \}} e^{-R_{s,N}(u)} d\mu_s, & Z_{s,R} &:= \int_{H^{s-\frac{1}{2}-}} \1_{ \{\cC(u) \leq R \}} e^{-R_s(u)} d\mu_s
\end{align*}
are normalizing constants ensuring that $\rho_{s,R,N}$ and $\rho_{s,R}$ are probability measures (if they are positive and finite, see Remark \ref{rem well defined proba meas}).
We recall that $R_{s,N}(u)$ and $R_s(u)$ have been defined in Definition~\ref{def normal form R and E} (see also ~\eqref{R_s,N diag multilinear form}). 
\begin{notn}
    We also use the notations $\rho_{s,R,\infty}$ and $Z_{s,R,\infty}$ to respectively refer to $\rho_{s,R}$ and $Z_{s,R}$. 
\end{notn}

\begin{rem}[$H^1(\T)$ cut-off]\label{rem cut-off H1}
We can rewrite $\cC(u)$ as :
\begin{equation*}
    \cC(u) = \frac{1}{2} \norm{u}_{H^1}^2 + \frac{1}{6}\norm{u}_{\L^6}^6
\end{equation*}
From the Sobolev embedding $H^1(\T) \hookrightarrow \L^6(\T)$, we have
\begin{equation*}
    \frac{1}{2} \norm{u}_{H^1}^2 \leq \cC(u) \leq C(1+\norm{u}_{H^1})^6
\end{equation*}
It means that the cut-off $\1_{ \{\cC(u) \leq R \}}$ is a $H^1(\T)$ cut-off. In particular, (for $R\geq 2$)
\begin{equation*}
    \cC(u) \leq R \implies \norm{u}_{H^1} \leq \sqrt{2R} \leq R
\end{equation*}
so, 
\begin{equation*}
    \1_{ \{\cC(u) \leq R \}} \leq \1_{B_R^{H^1}}(u)
\end{equation*}
where $B_R^{H^1}$ is the closed center ball of radius $R$ in $H^1(\T)$.\\
The additional nice property is that the quantity $\cC(u)$ is conserved by the flow of \eqref{NLS}.
\end{rem}

\begin{rem}(Passing from $\rho_{s,R,N}$ to $\mu_s$)\label{rem mu_s on balls}
Let $R>0$. The measure $\mu_s|_{\{ \cC \leq R\}}$ coincide with the measure $ Z_{s,N,R}e^{R_{s,N}(u)}\rho_{s,R,N}|_{\{ \cC \leq R\}}$.
    In other words, for every Borel set $A \subset H^1(\T)$ such that $A \subset \{\cC \leq R\}$, we have
    \begin{equation*}
        Z_{s,N,R} e^{R_{s,N}(u)} \rho_{s,R,N}(A) = \mu_s(A)
    \end{equation*}
    Indeed, it results from
    \begin{equation*}
        \begin{split}
        Z_{s,N,R} e^{R_{s,N}(u)}\rho_{s,R,N}(A) & = \int_A Z_{s,N,R}  e^{R_{s,N}(u)} d\rho_{s,R,N} (u) \\
        & = \int_A  Z_{s,N,R} e^{R_{s,N}(u)} \underbrace{\1_{ \{\cC(u) \leq R \}}}_{=1} e^{-R_{s,N}(u)} \frac{1}{Z_{s,N,R}} d\mu_s(u) \\
        & = \int_A d\mu_s(u) = \mu_s(A)
        \end{split}
    \end{equation*}
\end{rem}

Now, we state the following crucial proposition, whose proof is postponed to the dedicated Section~\ref{section Estimates for the weight of the weighted Gaussian measures}.

\begin{prop}\label{prop estimate for the weight of the weighted measure}
    Let $s > \frac{3}{2}$ and $R>0$. Then for any $p \in [1,+\infty)$, there exists a constant $C(s,p,R)>0$ such that for every $N \in \N \cup \{ \infty \}$, we have:
    \begin{equation}\label{estimate weight of wgm}
       \norm{\1_{\{ \cC(u) \leq R \} } e^{|R_{s,N}(u)|}}_{\Lp(d\mu_s)} \leq C(s,p,R)
    \end{equation}
    Moreover, 
    \begin{equation*}
        \norm{\1_{\{ \cC(u) \leq R \} } e^{-R_s(u)}-\1_{\{ \cC(u) \leq R \} } e^{-R_{s,N}(u)}}_{\L^p(d\mu_s)} \tendsto{N \ra \infty} 0
    \end{equation*}
    In particular, $Z_{s,N,R} \tendsto{N} Z_{s,R}$ so that we also have:
        \begin{equation}\label{cvgce density rhoN to density rho}
        \norm{\frac{1}{Z_{s,R}}\1_{\{ \cC(u) \leq R \} } e^{-R_s(u)}-\frac{1}{Z_{s,R,N}}\1_{\{ \cC(u) \leq R \} } e^{-R_{s,N}(u)}}_{\L^p(d\mu_s)} \tendsto{N \ra \infty} 0
    \end{equation}
\end{prop}

\begin{rem}\label{rem well defined proba meas}
    The inequality \eqref{estimate weight of wgm} in  Proposition~\ref{prop estimate for the weight of the weighted measure} ensures that for all $N \in \N \cup \{ \infty\}$ :
    \begin{equation*}
       Z_{s,R,N} = \int_{H^{s-\frac{1}{2}}}  \1_{\{ \cC(u) \leq R \} } e^{-R_{s,N}(u)}d\mu_s < +\infty
        \end{equation*}
        Besides, we also have $Z_{s,R,N}>0$ because one can show that $\mu_s (  \1_{\{ \cC(u) \leq R \} } e^{-R_{s,N}(u)} > 0 ) > 0$. To see this, we write on the one hand:
    \begin{equation*}
        \mu_s \left(  \1_{\{ \cC(u) \leq R \} } e^{-R_{s,N}(u)} > 0 \right) = \mu_s\left(\{ \cC(u) \leq R \} \cap R_{s,N}<+\infty \right) = \mu_s\left(\{ \cC(u) \leq R \} \right)
    \end{equation*}
    and on the other hand, we have that $\mu_s(\{ \cC(u) \leq R \}) > 0$, for any $R>0$, since $\mu_s$ charges all open sets of $H^1(\T)$. For this latter point, we refer to \cite{burq2013probabilistic}, Proposition 1.2.
  Note also that since $Z_{s,R,N} \ra Z_{s,R}$, there exists a constant $C_{s,R}>0$ such that:
    \begin{equation}\label{1/Cs,R leq Zs,R,N leq Cs,R}
        \frac{1}{C_{s,R}} \leq Z_{s,R,N} \leq C_{s,R}
    \end{equation}
    uniformly in $N\in \N \cup \{ \infty \}$.
\end{rem}

\subsection{Transport of weighted Gaussian measures along the flows} \underline{For $N \in \N \cup \{ \infty \}$}, now that we know explicitly the density of $\Phi_N(t)_\# \mu_s$ with respect to $\mu_s$, we are able to obtain the density of $\Phi_N(t)_\# \rho_{s,R,N}$ with respect to $\rho_{s,R,N}$ without re-performing the analysis of Section~\ref{section Transport of Gaussian measures under the truncated flow} and~\ref{section Transport of Gaussian measures under the flow}. \\

We stress the fact that the following proposition holds for $N \in \N$ \textbf{and} $N=\infty$.

\begin{prop}\label{prop Radon-Nikodym derivative for the transported wgm}
    Let $s > \frac{3}{2}$, $R>0$ and $N \in \N \cup\{ \infty \}$. For every $t \in \R$, we have:
    \begin{equation*}
        \Phi_N(t)_\# \rho_{s,R,N} = F_{s,N}(t,u) d\rho_{s,R,N}
    \end{equation*}
    for a function $F_{s,N}(t,.)$ given by the explicit formula:
    \begin{equation*}
        F_{s,N}(t,u)=\textnormal{exp}\left(-\int_0^{-t} Q_{s,N}(\Phi_N(\tau)u)d\tau \right)
    \end{equation*}
    where $Q_{s,N}$ is defined in \eqref{def Qs,N = Q0 + Q1 +Q2} (see also Definition~\ref{def normal form Q}). In particular, $\rho_{s,R}$ is quasi-invariant along the flow of \eqref{NLS}; and when $N \in \N$, $\rho_{s,R,N}$ is quasi-invariant along the flow of \eqref{truncated equation}.
\end{prop}

\begin{rem}
    When $N \in \N$, we have:
    \begin{equation*}
        \int_0^{t} Q_{s,N}(\Phi_N(\tau)u)d\tau  = E_{s,N}(\Pi_N \Phi_N(t)u) - E_{s,N}(\Pi_N u)
    \end{equation*}
    Hence, the a priori ill-defined object $E_s(\Phi(t)u)-E_s(u)$ (on the support of $\mu_s$) can be seen as:
    \begin{equation*}
        E_s(\Phi(t)u)-E_s(u) := \int_0^{t} Q_s(\Phi(\tau)u)d\tau
    \end{equation*}
    which is a continuous function on $\hsigt$ for $\sigma <s-\frac{1}{2}$ close enough to $s-\frac{1}{2}$.
\end{rem}

\begin{proof}[Proof of Proposition~\ref{prop Radon-Nikodym derivative for the transported wgm}]
    Let $N \in \N \cup \{ \infty \}$. We start by applying a general feature for the transport of density measures:
    \begin{equation*}
    \begin{split}
        \Phi_N(t)_\# \rho_{s,R,N} &=  \Phi_N(t)_\# \left(\frac{1}{Z_{s,R,N}}\1_{ \{\cC(u) \leq R \}} e^{-R_{s,N}(u)} d\mu_s \right)\\
        &= \frac{1}{Z_{s,R,N}}\1_{ \{\cC(\Phi_N(t)^{-1}u) \leq R \}} e^{-R_{s,N}(\Phi_N(t)^{-1}u)} \Phi_N(t)_\# d\mu_s
    \end{split}
    \end{equation*}
    Next, thanks to Proposition~\ref{prop Radon-Nikodym derivative for the trc trspted GM} and \ref{prop Radon-Nikodym derivative for the transported GM}, along with the facts that $\Phi_N(t)^{-1}=\Phi_N(-t)$ and that $\cC$ is conserved by the flows, we obtain:
    \begin{equation*}
        \begin{split}
        \Phi_N(t)_\# \rho_{s,R,N} &= \frac{1}{Z_{s,R,N}}\1_{ \{\cC(u) \leq R \}} \textnormal{exp}\left(-R_{s,N}(u) - \int_0^{-t} Q_{s,N}(\Phi_N(\tau)u) d\tau\right)d\mu_s \\
        &=  \textnormal{exp}\left( - \int_0^{-t} Q_{s,N}(\Phi_N(\tau)u) d\tau\right) d\rho_{s,R,N
        }
        \end{split}
    \end{equation*}
    which completes the proof.
\end{proof}

\begin{rem}\label{rem FsN cvg uniformly on cpcts to Fs} Let $s>\frac{3}{2}$ and $\sigma<s-\frac{1}{2}$. Then, for every $t\in \R$, $F_{s,N}(t,.)$ converges to $F_s(t,.)$ uniformly on compact sets of $\hsigt$ (see the proof of Proposition~\ref{prop GsN cvg uniformly on cpcts to Gs}).
\end{rem}

\section{Densities in \texorpdfstring{$L^p$}{TEXT} and convergence in \texorpdfstring{$L^p$}{TEXT} of the truncated densities}\label{section densities in Lp and convergence}
Let $s>\frac{3}{2}$. In Section~\ref{section Transport of Gaussian measures under the truncated flow} and \ref{section Transport of Gaussian measures under the flow}, we have seen that for every $N \in \N$ and every $t\in \R$,
\begin{align*}
    \Phi_N(t)_\# \mu_s & = G_{s,N}(t,.)\mu_s, & &\text{and,} & \Phi(t)_\# \mu_s & = G_s(t,.)\mu_s
\end{align*}
where $G_{s,N}$ and $G_s$ are known explicitly (see Proposition~\ref{prop Radon-Nikodym derivative for the trc trspted GM} and \ref{prop Radon-Nikodym derivative for the transported GM}). In this section, our goal is to prove that for every fixed $R>0$, and every $t\in \R$:
\begin{equation*}
   G_s(t,.), \hsp G_{s,N}(t,.) \in \L^p(d\mu_{s,R}) \hspace{0.2cm} \text{and:} \hspace{0.2cm} \norm{G_s(t,.) -G_{s,N}(t,.) }_{\L^p(d\mu_{s,R})} \tendsto{N \ra \infty} 0
\end{equation*}
where $\mu_{s,R}$ is the \textit{restricted Gaussian measure} define by:
\begin{equation}\label{mu_s,R}
    \mu_{s,R}:= \1_{ \{\cC(u)\leq R \} } \mu_s
\end{equation}
This is equivalent to the fact that:
\begin{equation*}
  \1_{ \{\cC(u)\leq R \} } G_s(t,.), \hsp \1_{ \{\cC(u)\leq R \} }G_{s,N}(t,.) \in \L^p(d\mu_s) \hspace{0.2cm} \text{and:} \hspace{0.2cm} \norm{\1_{ \{\cC(u)\leq R \} }(G_s(t,.) -G_{s,N}(t,.)) }_{\L^p(d\mu_s)} \tendsto{N \ra \infty} 0
\end{equation*}

\begin{rem}
    Since $\cC$ is conserved by the flow, we still have:
    \begin{align}\label{transport restricted GM}
    \Phi_N(t)_\# \mu_{s,R}& = G_{s,N}(t,.)\mu_{s,R}, & &\text{and,} & \Phi(t)_\# \mu_{s,R} & = G_s(t,.)\mu_{s,R}
\end{align}
\end{rem}

In our approach, we do not consider directly $\mu_{s,R}$ but we consider instead the weighted Gaussian measures $\rho_{s,R,N}$. We will then be able to prove a \textit{quantitative inequality} in Proposition~\ref{prop quantitative quasi-invariance for rho_s,R} thanks to suitable $\L^p$-estimates on $Q_{s,N}$.
In a second step, we will be able to go back to $\mu_{s,R}$ by proving the same quantitative inequaltiy for $\mu_{s,R}$. \\
The quantitative inequaltiy \eqref{ineq quantitative quasi-invariance for rho_s,R} is significant in itself; indeed, it is often used to obtain the quasi-invariance without knowing the Radon-Nikodym derivative, see for example \cite{genovese2022quasi}, \cite{oh2017quasi}, \cite{tzvetkov2015quasiinvariant},\cite{genovese2023transport}. Here, we know explicitly the Radon-Nikodym derivative of the transported measure, and from this point it will be easier to establish \eqref{ineq quantitative quasi-invariance for rho_s,R}.

\subsection{Quantitative quasi-invariance}
Recall that for convenience we use the notations $\rho_{s,R,\infty} = \rho_{s,R}, Q_s=Q_{s,\infty}, \text{etc}$. \\
We start this paragraph by providing $L^p$ estimates for $Q_{s,N}$. These estimates will help us to establish the \textit{quantitative quasi-invariance property} in Proposition~\ref{prop quantitative quasi-invariance for rho_s,R}. We postpone the proof of the following proposition to Section~\ref{section Estimates for the differential of the modified energy} where a detailed analysis is provided.

\begin{prop}\label{prop Qs,N L^p estimate wrt mu_s}
    Let $s > \frac{3}{2}$. There exists $\beta \in (0,1)$ such that for every $R>0$, there exists a constant $C(s,R) > 0$, such that for any $p \in [2,+\infty)$,
    \begin{equation}\label{Qs,N L^p estimate wrt mu_s}
         \norm{\1_{\{ \cC{u} \leq R \}} Q_{s,N}(u)}_{\L^p(d\mu_s)} \leq C(s,R)p^\beta
    \end{equation}
    uniformly in $N \in \N \cup \{\infty \}$.\\
\end{prop}
Combining estimate \eqref{Qs,N L^p estimate wrt mu_s} with estimate \eqref{estimate weight of wgm} from Proposition~\ref{prop estimate for the weight of the weighted measure}, we also have:

\begin{prop}\label{prop Qs,N L^p estimate wrt rho_s,R,N}
    Let $s > \frac{3}{2}$. There exists $\beta \in (0,1)$ such that for every $R>0$, there exists a constant $C(s,R) > 0$, such that for any $p \in [1,+\infty)$,
    \begin{equation}\label{Qs,N L^p estimate wrt rho_s,R,N}
        \norm{ Q_{s,N}(u)}_{\L^p(d\rho_{s,R,N})} \leq C(s,R)p^\beta
    \end{equation}
    uniformly in $N \in \N \cup \{\infty \}$.\\
\end{prop}

\begin{proof}[Proof of Proposition~\ref{prop Qs,N L^p estimate wrt rho_s,R,N} assuming \eqref{Qs,N L^p estimate wrt mu_s} and \eqref{estimate weight of wgm}] It results from Cauchy-Schwarz that:
\begin{equation*}
    \begin{split}
    \norm{ Q_{s,N}(u)}_{\L^p(d\rho_{s,R,N})} &= \norm{\frac{1}{Z_{s,R,N}}\1_{\{ \cC{u} \leq R \}} |Q_{s,N}(u)|^p e^{-R_{s,N}(u)}}_{\L^1(d\mu_s)}^{1/p} \\
    & \leq\frac{1}{Z_{s,R,N}}  \norm{\1_{\{ \cC{u} \leq R \}} Q_{s,N}(u)}_{\L^{2p}(d\mu_s)} \norm{\1_{\{ \cC{u} \leq R \}}e^{-R_{s,N}(u)}}_{\L^2(d\mu_s)}^{1/p} \\
    & \leq C(s,R)2^{\beta} p^\beta \cdot C(s,R,2)^{1/p} \leq C(s,R) p^\beta
    \end{split}
\end{equation*}
    where the constant $C(s,R)$ has changed but still depends only on $s$ and $R$. Note also that we used \eqref{1/Cs,R leq Zs,R,N leq Cs,R}.
\end{proof}

The following identity will be our starting point in order to obtain the forthcoming inequality \eqref{ineq quantitative quasi-invariance for rho_s,R}.

\begin{prop}\label{prop formula d/dt rho(phi(t)(A)) wrt to Qs} Let $s > \frac{3}{2}$, $R>0$, $N \in \N \cup \{ \infty\}$ and $t \in \R$. For every Borel set $A \subset H^\sigma(\T)$, we have :
\begin{equation}\label{formula d/dt rho(phi_N(t)(A)) wrt to Qs,N}
    \frac{d}{dt} \rho_{s,R,N}\left(\Phi_N(t)(A) \right) = - \int_{\Phi_N(t)A} Q_{s,N}(u) d\rho_{s,R,N}
\end{equation}    
\end{prop}

\begin{proof}[Proof of Proposition~\ref{prop formula d/dt rho(phi(t)(A)) wrt to Qs}] We use the explicit formula for the density from Proposition~\ref{prop Radon-Nikodym derivative for the transported wgm}, so that we obtain:
\begin{equation*}
    \begin{split}
         & \frac{d}{dt} \rho_{s,R,N}\left(\Phi_N(t)(A) \right) = \frac{d}{dt} \left( \Phi_N(-t)_\# \rho_{s,R,N} \right)(A) = \frac{d}{dt} \int_A e^{-\int_0^t Q_{s,N}(\Phi_N(\tau)u) d\tau} d\rho_{s,R,N} \\
         & = - \int_A Q_{s,N}(\Phi_N(t)u) d \left( \Phi_N(-t)_\# \rho_{s,R,N} \right)(u) = -\int_{\Phi_N(t)A} Q_{s,N}(u) d\rho_{s,R,N}
    \end{split}
\end{equation*}
where the last equality follows from the definition of a push-forward measure.
\end{proof}

\begin{prop}\label{prop quantitative quasi-invariance for rho_s,R}
    Let $s>\frac{3}{2}$, $R>0$ and $t\in \R$. Then, there exists $\beta \in (0,1)$ such that for every $\alpha \in (0,1)$, there exists a constant $C=C_{s,R,\alpha,\beta} > 0$ such that for all Borel set $A \subset H^\sigma(\T)$ :
    \begin{equation}\label{ineq quantitative quasi-invariance for rho_s,R}
        \rho_{s,R,N}(\Phi_N(t)A) \leq \rho_{s,R,N}(A)^{1-\alpha} \textnormal{exp}\left( C(1+|t|)^{\frac{1}{1-\beta}}\right)
    \end{equation}
    uniformly in $N \in \N \cup \{ \infty \}$.
\end{prop}

\begin{proof}
     Let $s>\frac{3}{2}$, $R>0$ and $N \in \N \cup \{ \infty \}$. Let $A \subset H^\sigma(\T)$ be a Borel set. Using \eqref{formula d/dt rho(phi_N(t)(A)) wrt to Qs,N}, Hölder inequality and the energy estimate from Proposition~\ref{prop Qs,N L^p estimate wrt rho_s,R,N}, we obtain :
     \begin{equation*}
         \big| \frac{d}{dt}\rho_{s,R,N}(\Phi_N(t)A) \big| \leq \norm{Q_{s,N}}_{\L^p(d\rho_{s,R,N})} \rho_{s,R,N}(\Phi_N(t)A)^{1-\frac{1}{p}} \leq C_{s,R} \hsp p^\beta  \rho_{s,R,N}(\Phi_N(t)A)^{1-\frac{1}{p}}
     \end{equation*}
     for all $p \in [1,+\infty)$. It means that the function $F(t):=\rho_{s,R,N}(\Phi_N(t)A)$ satisfies the differential inequality :
     \begin{equation*}
        | F'(t) | \leq C_{s,R} \hsp p^\beta F(t)^{1-\frac{1}{p}}
     \end{equation*}
     Integrating this yields :
     \begin{equation*}
        \begin{split}
         F(t) \leq \left( F(0)^{\frac{1}{p}} + C_{s,R}|t| p^{\beta -1} \right)^p &= F(0) \left( 1 + C_{s,R}|t| F(0)^{-\frac{1}{p}}p^{\beta -1} \right)^p \\
         &= F(0) \textnormal{exp}\left(p\textnormal{log}(1+ C_{s,R}|t| F(0)^{-\frac{1}{p}}p^{\beta -1}) \right)
         \end{split}
     \end{equation*}
     Then, using the inequality $\textnormal{log}(1+x)\leq x$ implies : 
     \begin{equation*}
          F(t) \leq F(0) \textnormal{exp}\left( C_{s,R}|t| F(0)^{-\frac{1}{p}}p^\beta) \right)
     \end{equation*}
     Now, we choose $p:=1+\textnormal{log}(\frac{1}{F(0)})$ so that $F(0)^{-\frac{1}{p}} = \textnormal{exp}(\frac{\textnormal{log}(F(0))}{\textnormal{log}(F(0))-1}) \leq \textnormal{exp}(1)$. Hence,
          \begin{equation}\label{ineq optimized in p}
          F(t) \leq F(0) \textnormal{exp}\left( C_{s,R}|t| (1-\textnormal{log}F(0))^\beta \right)
     \end{equation}
     Let us mention the following elementary lemma, whose proof is provided afterwards.
\begin{lem}\label{lem qualitative elementary estimate}
    For every $\alpha \in (0,1)$, there exists a constant $C_{s,R,\alpha,\beta}>0$ such that for all $x \geq 0$ :
    \begin{equation*}
        C_{s,R}|t| \left( 1+x \right)^\beta \leq \alpha x + C_{s,R,\alpha,\beta} \left( 1+|t| \right)^{\frac{1}{1-\beta}}
    \end{equation*}
\end{lem}
Using Lemma~\ref{lem qualitative elementary estimate} in \eqref{ineq optimized in p} with $x=-\textnormal{log}(F(0)) \geq 0$ yields :
\begin{equation*}
    F(t) \leq F(0)^{1-\alpha} \textnormal{exp} \left( C_{s,R,\alpha,\beta} (1+|t|)^{\frac{1}{1-\beta}} \right)
\end{equation*}
which is the desired inequality, so the proof is completed.
\end{proof}

Here, we provide a proof of Lemma~\ref{lem qualitative elementary estimate}:

\begin{proof}[Proof of Lemma~\ref{lem qualitative elementary estimate}]
    Let $\alpha \in (0,1)$. We invoke the function $f(x) := C_{s,R}|t|(1+x)^\beta - \alpha x$. We aim to show that $f(x) \leq C_{s,R,\alpha,\beta} \left( 1+|t| \right)^{\frac{1}{1-\beta}}$. Recall that $\beta \in (0,1)$. We have :
    \begin{equation*}
        f'(x) = \beta C_{s,R}|t|(1+x)^{\beta-1} -\alpha \geq 0 \iff (1+x)^{\beta-1}  \geq \frac{\alpha}{\beta C_{s,R}|t|} \iff x \leq \left( \frac{\beta C_{s,R}|t|}{\alpha} \right)^{\frac{1}{1-\beta}} -1
    \end{equation*}
    This implies that $f$ has a maximum at the point $x=\left( \frac{\beta C_{s,R}|t|}{\alpha} \right)^{\frac{1}{1-\beta}} -1$. Thus, for all $x\geq 0$ :
    \begin{equation*}
        \begin{split}
        f(x) &\leq C_{s,R}|t| \left( \frac{\beta C_{s,R}|t|}{\alpha} \right)^{\frac{\beta}{1-\beta}} -\alpha \left( \frac{\beta C_{s,R}|t|}{\alpha} \right)^{\frac{1}{1-\beta}} +\alpha \\
        & = \alpha + (C_{s,R}|t|)^{\frac{1}{1-\beta}} \left( \left( \frac{\beta}{\alpha} \right)^{\frac{\beta}{1-\beta}} - \alpha \left(\frac{\beta}{\alpha} \right)^{\frac{1}{1-\beta}} \right) \\
        & \leq C_{s,R,\alpha,\beta}(1+|t|)^{\frac{1}{1-\beta}}
        \end{split}
    \end{equation*}
    This completes the proof of Lemma~\ref{lem qualitative elementary estimate}.
\end{proof}

Now, we can go back to the measure $\mu_{s,R}$ (see \eqref{mu_s,R} for the definition). Indeed, we are able to deduce from Proposition~\ref{prop quantitative quasi-invariance for rho_s,R} the analogous proposition for $\mu_{s,R}$.

\begin{prop}\label{prop quantitative quasi-invariance for mu_s,R}
     Let $s>\frac{3}{2}$, $R>0$ and $t\in \R$. Then, there exists $\beta \in (0,1)$ such that for every $\alpha \in (0,1)$, there exists a constant $C=C_{s,R,\alpha,\beta} > 0$ such that for all Borel set $A \subset H^1(\T)$ :
    \begin{equation}\label{ineq quantitative quasi-invariance for mu_s,R}
        \mu_{s,R}(\Phi_N(t)A) \leq \mu_{s,R}(A)^{1-\alpha} \textnormal{exp}\left( C(1+|t|)^{\frac{1}{1-\beta}}\right)
    \end{equation}
    uniformly in $N \in \N \cup \{ \infty \}$.
\end{prop}

\begin{proof}
     Let $N\in \N \cup \{ \infty \}$. Recall that $\mu_{s,R}$ is defined in \eqref{mu_s,R}, and observe that we have:
     \begin{equation*}
         \mu_{s,R}=Z_{s,N,R} e^{R_{s,N}(u)} d\rho_{s,R,N}
     \end{equation*}
    (see also Remark \ref{rem mu_s on balls}). Let $\alpha \in (0,1)$. Let us invoke the $\beta \in (0,1)$ and the constant $C>0$ from Proposition~\ref{prop quantitative quasi-invariance for rho_s,R}. Then, for every $q_1 \in (1,\infty)$:
    \begin{equation*}
        \begin{split}
        \mu_{s,R}(\Phi_N(t)A)= Z_{s,R,N} \int_{\Phi(t)A} e^{R_{s,N}(u)} d\rho_{s,R,N} &\leq Z_{s,R,N} \norm{e^{R_{s,N}}}_{\L^{q_1}(d\rho_{s,R,N})}\rho_{s,R,N}(\Phi(t)A)^{1-\frac{1}{q_1}} \\
        & \leq C_{s,R,q_1} \rho_{s,R,N}(\Phi(t)A)^{1-\frac{1}{q_1}}
        \end{split}
    \end{equation*}
    where in the last inequality we used \eqref{1/Cs,R leq Zs,R,N leq Cs,R} and Proposition~\ref{prop estimate for the weight of the weighted measure}. Using now Proposition~\ref{prop quantitative quasi-invariance for rho_s,R}, we obtain:
    \begin{equation}\label{first ineq mu_s(phi(t)A)}
        \mu_{s,R}(\Phi_N(t)A) \leq \rho_{s,R,N}(A)^{(1-\alpha)(1-\frac{1}{q_1})} \textnormal{exp}\left( C(1+|t|)^{\frac{1}{1-\beta}}\right)
    \end{equation}
    where the constant $C>0$ depends on $s,R,\alpha$ and $q_1$ (and also $\beta$). On the other hand, for every $q_2 \in (1,\infty)$:
    \begin{equation*}
        \begin{split}
        \rho_{s,R,N}(A) = \frac{1}{Z_{s,R,N}} \int_A e^{-R_{s,N}} d\mu_{s,R} &\leq \frac{1}{Z_{s,R,N}} \norm{e^{-R_{s,N}}}_{\L^{q_2}(d\mu_{s,R})} \mu_{s,R}(A)^{1-\frac{1}{q_2}} \\ &\leq C_{s,R,q_2} \mu_{s,R}(A)^{1-\frac{1}{q_2}} 
        \end{split}
    \end{equation*}
     where again in the last inequality we used \eqref{1/Cs,R leq Zs,R,N leq Cs,R} and Proposition~\ref{prop estimate for the weight of the weighted measure}. Then, plugging this into \eqref{first ineq mu_s(phi(t)A)} yields:
    
    \begin{equation*}
        \mu_{s,R}(\Phi_N(t)A) \leq \rho_{s,R,N}(A)^{(1-\alpha)(1-\frac{1}{q_1})(1-\frac{1}{q_2})} \textnormal{exp}\left( C(1+|t|)^{\frac{1}{1-\beta}}\right)
    \end{equation*}
    where now, the constant $C>0$ depends on $s,R,\alpha,q_1$ and $q_2$. Finally, since $\alpha \in (0,1)$ and $q_1,q_2 \in (1,\infty)$ are arbitrary, we have that $(1-\alpha)(1-\frac{1}{q_1})(1-\frac{1}{q_2})$ takes all the values in $(0,1)$. This completes the proof.
\end{proof}

\subsection{Consequences: uniform \texorpdfstring{$L^p$}{TEXT} integrability and convergence
} In this paragraph, we perform an analysis similar to the one from Section 7 of \cite{genovese2023transport}. As a consequence of Proposition~\ref{prop quantitative quasi-invariance for mu_s,R}, we are now able to prove the following proposition. Here, for clarity, we denote $G_{s,N,t}=G_{s,N}(t,.)$ for $N \in \N \cup \{ \infty \}$ (still $G_{s,\infty,t}=G_{s,t}$).

\begin{prop}\label{prop Gs,N uniformly bounded in Lp}
    Let $s>\frac{3}{2}$, $R>0$ and $t \in \R$.  For every $p\in [1,+\infty)$, there exists a constant $C(s,R,p,t) \in (0,+\infty)$ such that :
    \begin{equation*}
        \norm{G_{s,N,t}}_{\L^p(d\mu_{s,R})} \leq C(s,R,p,t)
    \end{equation*}
    uniformly in $N\in \N \cup \{ \infty \}$.
\end{prop}

\begin{proof}Let $p \in [1,+\infty)$ and let us prove that $G_{s,N,t}$ belongs to $\L^p(d\mu_{s,R})$. To do so, we use Cavalieri's principle :
    \begin{equation}\label{cavalieri's principle}
        \norm{G_{s,N,t}}_{L^p(d\mu_{s,R})}^p = p \int_0^{\infty} \lambda^{p-1} \mu_{s,R}(G_{s,N,t} > \lambda) d\lambda
    \end{equation}
    We can estimate $\mu_{s,R}(G_{s,N,t} > \lambda)$ as follows :
    \begin{equation*}
        \begin{split}
           & \mu_{s,R}(G_{s,N,t} > \lambda)= \int \1_{\{G_{s,N,t} > \lambda\} } d\mu_{s,R} = \int \frac{1}{G_{s,N,t}}\1_{\{G_{s,N,t} > \lambda\} }  G_{s,N,t} d\mu_{s,R}\\
            &= \int \frac{1}{G_{s,N,t}}\1_{\{G_{s,N,t} > \lambda\} }   d\left( \Phi_N(t)_\#\mu_{s,R} \right)
            \leq \frac{1}{\lambda} \int \1_{\{G_{s,N,t} > \lambda\} }   d\left( \Phi_N(t)_\#\mu_{s,R} \right) = \frac{1}{\lambda}\mu_{s,R}(\Phi_N(-t)(G_{s,N,t} > \lambda))
        \end{split}
    \end{equation*}
    Here, we can make use of estimate \eqref{ineq quantitative quasi-invariance for mu_s,R} and deduce that for any $\alpha \in (0,1)$, there exists $C=C_{s,R,\alpha}>0$ such that
    \begin{equation*}
        \mu_{s,R}(G_{s,N,t} > \lambda) \leq \frac{1}{\lambda} \mu_{s,R}(G_{s,N,t} > \lambda)^{1-\alpha} \textnormal{exp}\left( C(1+|t|)^{\frac{1}{1-\beta}}\right)
    \end{equation*}
    So, 
        \begin{equation*}
        \mu_{s,R}(G_{s,N,t} > \lambda) \leq \lambda^{-1/\alpha} \textnormal{exp}\left(C'(1+|t|)^{\frac{1}{1-\beta}}\right) =: \lambda^{-1/\alpha} C_{s,R,\alpha,t} 
    \end{equation*}
    (Where $C'=\frac{C}{\alpha}$ still depends only on $s,R$ and $\alpha$). Plugging this into \eqref{cavalieri's principle} yields
    \begin{equation*}
         \norm{G_{s,N,t}}_{L^p(d\mu_{s,R})}^p \leq p + \int_1^{\infty} \lambda^{p-1} \mu_{s,R}(G_{s,N,t} > \lambda) d\lambda \leq p + C_{s,R,\alpha,t} \int_1^{\infty} \lambda^{p-1-1/ \alpha} d\lambda
    \end{equation*}
    Choosing $\alpha \in (0,1)$ such that $p-1-\frac{1}{\alpha}<-1$, that is $\alpha < \frac{1}{p}$, we have that the integral $\int_1^{\infty} \lambda^{p-1-1/ \alpha} d\lambda$ is finite. Hence, denoting by
    \begin{equation*}
        C(s,R,p,t) :=  p + C_{s,R,\alpha,t} \int_1^{\infty} \lambda^{p-1-1/ \alpha} d\lambda
    \end{equation*}
    leads the result.
\end{proof}

\begin{prop}\label{prop Gs,N cvg Gs in Lp}
    Let $s> \frac{3}{2}$ and $R>0$. Then, for every $t\in \R$ and every $p \in [1,+\infty)$, $G_{s,N}(t,.)$ converges to $G_s(t,.)$ in $\L^p(d\mu_{s,R})$.
\end{prop}

\begin{proof}
    Let $q\in (1,+\infty)$. From Proposition~\ref{prop Gs,N uniformly bounded in Lp} and \ref{prop GsN cvg uniformly on cpcts to Gs}, we know that the two following facts hold:
    \begin{equation*}
        \begin{cases}
            \sup_{N \in \N} \norm{G_{s,N}(t,.)}_{L^q(d\mu_{s,R})} < +\infty \\
             \text{$G_{s,N}(t,.)$ converges in measure to $G_s(t,.)$ (with respect to $\mu_{s,R}$)}
        \end{cases}
    \end{equation*}
    Then -- see for example \cite{tzvetkov2008invariant} Remark 3.8 -- it implies that for every $p \in [1,q)$, $G_{s,N}(t,.)$ converges in $\L^p(d\mu_{s,R})$ to $G_s(t,.)$. Since $q \in (1,+\infty)$ is arbitrary, we deduce that this convergence holds for every $p \in [1,+\infty)$.
\end{proof}

We conclude this section by a short remark:
\begin{rem}\label{rem density wrt wgm in L^p and cvgce}
Starting from Proposition~\ref{prop quantitative quasi-invariance for rho_s,R}, it would have also been possible to prove (with the notations of Proposition~\ref{prop Radon-Nikodym derivative for the transported wgm}) that the densities $F_{s,N}(t,.)$, $F_s(t,.)$ of the transported weighted Gaussian measures belongs to $\L^p(d\rho_{s,R})$ and that $F_{s,N}(t,.)$ converges to $F_s(t,.)$ in $\L^p(d\rho_{s,R})$.
\end{rem}

\section{Tools for the energy estimates}\label{section Tools for the energy estimates}
In this section, we gather the main tools that we will use in the forthcoming Section~\ref{section Proofs of the deterministic properties}, \ref{section Estimates for the weight of the weighted Gaussian measures} and \ref{section Estimates for the differential of the modified energy}.

\subsection{Deterministic tools }
\begin{notns}[Number ordering] We will intensively use the following notations:
\begin{itemize} 
    \item Given a set of frequencies $k_1,...,k_m \in \Z$, we denote by $k_{(1)},...,k_{(m)}$ a rearrangement of the $k_j's$ such that 
    \begin{equation*}
        |k_{(1)}| \geq |k_{(2)}| \geq ... \geq |k_{(m)}| 
    \end{equation*}
    \item Similarly, given a set of dyadic integers $N_1,...,N_m \in 2^{\N}$, we denote by $N_{(1)},...,N_{(m)}$ a non-decreasing rearrangement of the $N_j$.
\end{itemize}
\end{notns}

\begin{exmps} For example, we have :
    \begin{itemize}
        \item If $k_1 = -1$, $k_2=0$, $k_3=2$ then $k_{(1)} = k_3$, $k_{(2)}=k_1$, and $k_{(3)}=k_2$
        \item If $N_1 = 8$, $N_2=2$, $N_3=4$ then $N_{(1)}=N_1$, $N_{(2)}=N_3$, and $N_{(3)}= N_2$
    \end{itemize}
\end{exmps}

\begin{lem}[Counting bound]\label{lem counting bound}
    There exists a constant $C>0$ such that for every dyadic integers $N_1,...,N_m$, every $\eps_1,...,\eps_m \in \{-1,+1\}$, and every $\kappa \in \Z$,
    \begin{equation}\label{elementary counting bound}
        \sum_{k_1,...,k_m \in \Z} \1_{\eps_1 k_1 +\eps_2 k_2 +...+\eps_m k_m = \kappa} \cdot \left( \prod_{j=1}^{m} \1_{|k_j| \sim N_j} \right) \leq C N_{(2)}N_{(3)}...N_{(m)}
    \end{equation}
\end{lem}
\begin{proof}
    Without loss of generality, we assume that $N_j = N_{(j)}$. The idea of the proof is to let free the variables $k_{2},...,k_{m}$ and to freeze the variable $k_{1}$ thanks to the constraint $\eps_1 k_1 = \kappa -\eps_2 k_2 -...-\eps_m k_m$:
    \begin{equation*}
        \sum_{\substack{k_1,...,k_m \\ |k_j| \sim N_j}} \1_{\eps_1 k_1 +\eps_2 k_2 +...+\eps_m k_m = \kappa} = \sum_{\substack{k_1,...,k_m \\ |k_j| \sim N_j, j\geq2}}  \underbrace{\sum_{|k_1| \sim N_1} \1_{\eps_1 k_1 = \kappa -\eps_2 k_2 -...-\eps_m k_m}}_{\leq 1} \leq \sum_{\substack{k_1,...,k_m \\ |k_j| \sim N_j, j\geq2}} 1 \lesssim N_2 N_3...N_m
    \end{equation*}
\end{proof}
In the proof of energy estimates, we will use the following estimate on $\psi_{2s}$ as a starting point.

\begin{lem}\label{lemma psi estimate}
    Set 
    \begin{align*}
        \psi_{2s}(\vec{k}) &:= \sum_{j=1}^{6} (-1)^{j-1}|k_j|^{2s}, & \Omega(\vec{k}) &= \sum_{j=1}^6 (-1)^{j-1} k_j^2
    \end{align*}
There exists a constant $C(s)>0$ such that for every $k_1-k_2+k_3-k_4+k_5-k_6 = 0$,
\begin{equation*}
    | \psi_{2s}(\vec{k}) | \leq C(s) | k_{(1)} |^{2s-2} \left( | \Omega(\vec{k}) |+| k_{(3)} |^2\right)
\end{equation*}
\end{lem}

\begin{proof}
    Essentially, we have to consider two cases : when $k_{(1)}=k_1$, $k_{(2)}=k_2$ and $k_{(1)}=k_1$, $k_{(2)}=k_3$. In any case, we can assume that $|k_{(3)}| \leq \frac{1}{2} |k_{(2)}|$. Otherwise, using the fact that the constraint $k_1-k_2+k_3-k_4+k_5-k_6 = 0$ implies that $|k_{(2)}| \sim |k_{(1)}|$, we would deduce that $|k_{(3)}| \sim |k_{(1)}|$. Therefore, the a priori bound $ | \psi_{2s}(\vec{k}) | \lesssim |k_{(1)}|^{2s}$ would guarantee the desired inequality. \\

    \underline{Case 1} : Firstly, we consider the case $k_{(1)}=k_1$ and $k_{(2)}=k_2$. We use the mean value theorem :
    \begin{equation*}
        \begin{split}
            |\psi_{2s}(\vec{k})| & \leq k_1^{2s} - k_2^{2s} + |k_3^{2s}-k_4^{2s}+k_5^{2s}-k_6^{2s}| \\
            & \leq \sup_{t \in [k_2^2,k_1^2]} \frac{d}{dt}(t^s) (k_1^2 - k_2^2) + 4 k_{(3)}^{2s} \\
            & \leq s |k_{(1)}|^{2(s-1)}(\Omega(\vec{k}) - k_3^2+k_4^2-k_5^2+k_6^2)  + 4 k_{(1)}^{2(s-1)} k_{(3)}^2 \\
            & \leq C(s) |k_{(1)}|^{2(s-1)} (|\Omega(\vec{k})| + |k_{(3)}|^2)
        \end{split}
    \end{equation*}
which is the desired inequality. \\

    \underline{Case 2} : Secondly, we consider the case  $k_{(1)}=k_1$ and $k_{(2)}=k_3$. Since we assume that $k_{(3)} \leq \frac{1}{2} |k_{(2)}| = \frac{1}{2} |k_{3}|$, we have
    \begin{equation*}
            \Omega(\vec{k}) \geq k_1^2 + k_3^2 -|k_2^2 +k_4^2 -k_5^2+k_6^2| \geq k_1^2 + k_3^2 - 4 k_{(3)}^2 \geq k_1^2
    \end{equation*}
    Therefore, the a priori bound $ | \psi_{2s}(\vec{k}) | \lesssim |k_{(1)}|^{2s}$ guarantees the desired inequality.
\end{proof}

In order to make our analysis work for the full range $s > \frac{3}{2}$, we will use the followinw Strichartz estimate for the linear propagator of the Schrödinger equation--see \cite{Bourgain1993}.

\begin{thm}[Strichartz estimate]\label{Strichartz estimate}
    For any $\eps >0$, there exists a constant $C_{\eps} > 0$ such that for any function $g$ on $\T$ : 
    \begin{equation*}
        \norm{e^{it\p_x^2}g}_{\L^6(\T_t \times \T_x)} \leq C_{\eps} \norm{g}_{H^{\eps}(\T)}
    \end{equation*}
\end{thm}

As a consequence, we can prove the following useful estimate : 

\begin{lem}\label{Strichartz's trick}
    Let $\eps > 0$. There exists a constant $C_{\eps} > 0$ such that for any dyadic integers $N_1,...,N_6$ and any sequences $\left( f_{k_1}^{(1)} \right)_{k_1 \in \Z}$,...,$\left( f_{k_6}^{(6)} \right)_{k_6 \in \Z}$ of complex numbers that satisfy $f^{(j)}_{k_j} = \1_{|k_j| \sim N_j} f^{(j)}_{k_j} $, we have for any $\kappa \in \Z$ :
    \begin{equation*}
        \sum_{\substack{\linc \\ \Omega(\vec{k})= \kappa}} \prod_{j=1}^6 |f_{k_j}^{(j)}| \leq C_{\eps} N_{(1)}^{\eps} \prod_{j=1}^6 \norm{f^{(j)}}_{l^2}
    \end{equation*}
\end{lem}

\begin{proof}
    We set
    \begin{align*}
        F_0^{(1)}(x) &:= \sum_{k_1 \in \Z} |f_{k_1}^{(1)} |e^{i k_1 x}, &  F^{(1)}(t,x) & := e^{it\p_x^2}e^{it \kappa}F_0^{(1)}(x)
    \end{align*}
    and for $j \geq 2$, 
    \begin{align*}
         F_0^{(j)}(x) &:= \sum_{k_j \in \Z} |f_{k_j}^{(j)}| e^{i k_j x}, &  F^{(j)}(t,x) & := e^{it\p_x^2}F_0^{(j)}(x)
    \end{align*}
    Let us now prove the identity : 
    \begin{equation}\label{sum as an integral}
        \sum_{\substack{\linc \\ \Omega(\vec{k})= \kappa}} \prod_{j=1}^6 |f_{k_j}^{(j)}| = \int_{\T_t} \int_{\T_x} \left( F^{(1)}\cjg{F^{(2)}}F^{(3)}\cjg{F^{(4)}}F^{(5)}\cjg{F^{(6)}} \right) (t,x) \frac{dxdt}{(2\pi)^2}
    \end{equation}
    Firstly, we expand
    \begin{equation*}
        \left( F^{(1)}\cjg{F^{(2)}}F^{(3)}\cjg{F^{(4)}}F^{(5)}\cjg{F^{(6)}} \right) (t,x) = \sum_{k_1,...,k_6 \in \Z} |f_{k_1}^{(1)}...f_{k_6}^{(6)}| e^{ix(k_1-k_2+...-k_6)} e^{-it (\Omgk - \kappa)}
    \end{equation*}
    Secondly, we integrate with respect to $x$ and $t$ 
    \begin{equation*}
        \begin{split}
            \int_{\T_t} \int_{\T_x}   F^{(1)}\cjg{F^{(2)}}& F^{(3)}\cjg{F^{(4)}}F^{(5)}\cjg{F^{(6)}} (t,x) \frac{dxdt}{(2\pi)^2} \\
            & =   \sum_{k_1,...,k_6 \in \Z} |f_{k_1}^{(1)}...f_{k_6}^{(6)}| \left( \int_{\T_t} \int_{\T_x} e^{ix(k_1-k_2+...-k_6)} e^{-it (\Omgk-\kappa)} \frac{dxdt}{(2\pi)^2} \right)
       \end{split}
    \end{equation*}
    Then, the formula \eqref{sum as an integral} follows from the fact that 
    \begin{equation*}
        \int_{\T_t} \int_{\T_x} e^{ix(k_1-k_2+...-k_6)} e^{-it (\Omgk -\kappa)} \frac{dxdt}{(2\pi)^2} = \1_{\linc} \1_{\Omgeqkp}
    \end{equation*}
    Starting now from \eqref{sum as an integral} and using the Hölder's inequality and Theorem~\ref{Strichartz estimate} we get that 
    \begin{equation*}
       \sum_{\substack{\linc \\ \Omega(\vec{k})= \kappa}} \prod_{j=1}^6 |f_{k_j}^{(j)}| \leq \prod_{j=1}^6 \norm{F^{(j)}}_{\L^6(\T_t \times \T_x)} \leq C_{\eps} \prod_{j=1}^6 \norm{F_0^{(j)}}_{H^{\frac{\eps}{6}}(\T_x)}
    \end{equation*}
    We also used the fact that $\norm{F^{(1)}}_{\L^6(\T_t \times \T_x)} = \norm{e^{it\kappa}e^{it\p_x^2}F^{(1)}_0}_{\L^6(\T_t \times \T_x)} = \norm{e^{it\p_x^2}F^{(1)}_0}_{\L^6(\T_t \times \T_x)}$. \\

    On the other hand, with the localizations of the sequences $\left( f^{(j)}_{k_j}\right)_{k_j \in \Z}$, we can write 
    \begin{equation*}
        \norm{F_0^{(j)}}_{H^{\frac{\eps}{6}}(\T_x)} \lesssim N_j^{\frac{\eps}{6}}\norm{f^{(j)}}_{l^2}  \lesssim N_{(1)}^{\frac{\eps}{6}}\norm{f^{(j)}}_{l^2}
    \end{equation*}
    Coming back to the previous estimate we obtain 
    \begin{equation*}
        \sum_{\substack{\linc \\ \Omega(\vec{k})= \kappa}} \prod_{j=1}^6 |f_{k_j}^{(j)}| \leq C'_{\eps} N_{(1)}^{\eps}\prod_{j=1}^6 \norm{f^{(j)}}_{l^2}
    \end{equation*}
    which is the desired estimate.
\end{proof}

Finally, we conclude this section with the following suitable lemma. Although its significance will become apparent later (in the proof of energy estimates in Section~\ref{section Proofs of the deterministic properties}, \ref{section Estimates for the weight of the weighted Gaussian measures} and \ref{section Estimates for the differential of the modified energy}), we prove it now for a better clarity.

\begin{lem}\label{an estimate using Strichartz's trick}
       Let $s\in\R$ and $\eps > 0$. There exists a constant $C_{\eps,s} > 0$ such that for any dyadic integers $N_1,...,N_6$ and any sequences $\left( f_{k_1}^{(1)} \right)_{k_1 \in \Z}$,...,$\left( f_{k_6}^{(6)} \right)_{k_6 \in \Z}$ of complex numbers that satisfy $f^{(j)}_{k_j} = \1_{|k_j| \sim N_j} f^{(j)}_{k_j} $, we have
        \begin{equation*}
            \sum_{\substack{\linc \\ \Omgnz}} \left| \frac{\psi_{2s}(\vec{k})}{\Omega(\vec{k})}\right| \prod_{j=1}^6 | f_{k_j}^{(j)} | \leq C_{\eps,s} N_{(1)}^{2s-2+\eps} N_{(3)}^{2} \prod_{j=1}^6 \norm{f^{(j)}}_{l^2}
        \end{equation*}
    \end{lem}

\begin{proof}
    We start by applying Lemma~\ref{lemma psi estimate} :
    \begin{equation*}
        \begin{split}
            \sum_{\substack{\linc \\ \Omgnz}} \left| \frac{\psi_{2s}(\vec{k})}{\Omega(\vec{k})}\right| \prod_{j=1}^6 | f_{k_j}^{(j)} | & \lesssim \sum_{\substack{\linc \\ \Omgnz}} |k_{(1)}|^{2s-2} \left( 1 + \frac{|k_{(3)}|^2}{|\Omega(\vec{k})|} \right)  \prod_{j=1}^6 | f_{k_j}^{(j)} | \\
            & \lesssim \hspace{.5cm} \textbf{I} + \textbf{II}
        \end{split}
    \end{equation*}
    where we denote 
    \begin{equation*}
        \textbf{I} := N_{(1)}^{2s-2}\sum_{\substack{\linc \\ \Omgnz}} \prod_{j=1}^6 | f_{k_j}^{(j)} |
    \end{equation*}
    and,
    \begin{equation*}
        \textbf{II} := N_{(1)}^{2s-2} N_{(3)}^2\sum_{\substack{\linc \\ \Omgnz}} \frac{1}{|\Omega(\vec{k})|}\prod_{j=1}^6 | f_{k_j}^{(j)} |
    \end{equation*}
    We estimate separately \textbf{I} and \textbf{II}. \\

    \textbullet \textbf{Estimate of I :} Removing the constraint $\Omgnz$ and using the Cauchy-Schwarz inequality in the $k_{(1)},k_{(2)}$ summations, we get that
    \begin{equation*}
        \begin{split}
            \textbf{I} & \leq N_{(1)}^{2s-2} \norm{f^{\left( (1)\right)}}_{l^2} \norm{f^{\left( (2)\right)}}_{l^2} \prod_{j=3}^6 \norm{f^{\left( (j)\right)}}_{l^1} \\
            & \leq N_{(1)}^{2s-2} \left( N_{(3)} N_{(4)}N_{(5)} N_{(6)} \right)^{\frac{1}{2}} \prod_{j=1}^6 \norm{f^{(j)}}_{l^2} \\
            & \leq N_{(1)}^{2s-2}  N_{(3)}^2 \prod_{j=1}^6 \norm{f^{(j)}}_{l^2}
        \end{split} 
    \end{equation*}
    which is even better than the desired bound. \\
    
    \textbullet \textbf{Estimate of II :} Firstly, we observe that $|\Omega(\vec{k})|\lesssim N_{(1)}^2$ so we can write
    \begin{equation*}
        \textbf{II} \leq N_{(1)}^{2s-2} N_{(3)}^2 \sum_{ N_{(1)}^2 \gtrsim |\kappa| \geq 1} \frac{1}{|\kappa|} \sum_{\substack{\linc \\ \Omgeqkp}} \prod_{j=1}^6 |f_{k_j}^{(j)}|
    \end{equation*}
    Secondly, we use Lemma~\ref{Strichartz's trick} that says
    \begin{equation*}
            \sum_{\substack{\linc \\ \Omgeqkp}} \prod_{j=1}^6 |f_{k_j}^{(j)}| \leq C_{\eps} N_{(1)}^{\frac{\eps}{2}}\prod_{j=1}^6 \norm{f^{(j)}}_{l^2}
    \end{equation*}
    Finally, we invoke the well-known estimate 
    \begin{equation*}
        \sum_{ N_{(1)}^2 \gtrsim |\kappa| \geq 1} \frac{1}{|\kappa|} \lesssim \textnormal{log} \left( N_{(1)}\right) \leq C'_{\eps}N_{(1)}^{\frac{\eps}{2}}
    \end{equation*}
    and the combination of those inequalities yields 
    \begin{equation*}
        \textbf{II} \leq C''_{\eps} N_{(1)}^{2s-2+\eps} N_{(3)}^2 \prod_{j=1}^6 \norm{f^{(j)}}_{l^2}
    \end{equation*}
    which is the desired bound. 
    This completes the proof of Lemma~\ref{an estimate using Strichartz's trick}.
    \end{proof}

    \subsection{Some properties of Gaussian measures}
    Our analysis for the energy estimates in Section~\ref{section Estimates for the weight of the weighted Gaussian measures} and \ref{section Estimates for the differential of the modified energy} will also require the following probabilistic tools.
    
\begin{lem}[Moments of Gaussian measures]
    Let $s\in \R$ and $\sigma < s-\frac{1}{2}$. Then, there exists $C=C(s,\sigma)>0$ such that for all $m \geq 1$ :
    \begin{equation}\label{Gaussian moments}
       \Big( \int_{\hsig} \hsignorm{u}^m d\mu_s \Big)^{\frac{1}{m}} \leq C m^{\frac{1}{2}} 
    \end{equation}
\end{lem}
\begin{proof}
From Fernique's theorem (see for example \cite{kuo2006gaussian}, \cite{bogachev1998gaussian}), there exists $\alpha>0$ such that:
\begin{equation*}
    \int_{\hsig} e^{\alpha \norm{u}_{\hsig}^2} d\mu_s < + \infty
\end{equation*}
Then, by Markov's inequality, we obtain the following large deviation estimate:
\begin{equation*}
    \mu_s(\hsignorm{u} > \lambda) = \mu_s(e^{\alpha \hsignorm{u}^2} > e^{\alpha \lambda^2}) \leq e^{-\alpha \lambda^2} \int_{\hsig} e^{\alpha \norm{u}_{\hsig}^2} d\mu_s \leq C e^{-\alpha \lambda^2}
\end{equation*}
Combining this estimate with Cavalieri's principle, we have:
\begin{equation*}
    \int_{\hsig} \hsignorm{u}^m d\mu_s = m \int_0^{+\infty} \lambda^{m-1} \mu_s({\hsignorm{u}>\lambda}) d\lambda \leq C m \int_0^{+\infty} \lambda^{m-1} e^{-\alpha \lambda^2} d\lambda
\end{equation*}
To conclude, we perform $\lfloor \frac{m}{2}\rfloor -1$ integration by parts on the integral on the right hand side:
\begin{equation*}
    \begin{split}
   \int_0^{+\infty} \lambda^{m-1} e^{-\alpha \lambda^2} d\lambda &= \frac{m-2}{2\alpha} \int_0^{+\infty} \lambda^{m-3} e^{-\alpha \lambda^2} d\lambda \\
   &=...= \frac{(m-2)(m-4)...(m-2(\lfloor \frac{m}{2}\rfloor-1))}{(2\alpha)^{\lfloor \frac{m}{2}\rfloor-1} }\int_0^{+\infty}\lambda^{m- 2(\lfloor \frac{m}{2}\rfloor-1)-1} e^{-\alpha \lambda^2}d\lambda \\
  & \leq \left( \frac{m}{2\alpha}\right)^{\lfloor \frac{m}{2}\rfloor-1} \sup_{\tld{m} \in [0,2]} \int_0^{+\infty} \lambda^{\tld{m}}e^{-\alpha \lambda^2} d\lambda
   \end{split}
\end{equation*}
\end{proof}

Next, we state a conditional Wiener chaos estimate which will play a crucial role in the energy estimates. In the sequel, for any complex number $z$, we adopt the notation $z^+ =z$ and $z^-=\cjg{z}$, called respectively positive and negative \textit{signature} of $z$.

\begin{lem}[Conditional Wiener chaos estimate]\label{lem Wiener chaos}
    Let $(\Omega,\cA,\P)$ be a probability space and $\cB$ be a $\sigma$-algebra on $\Omega$ such that $\cB \subset \cA$. Let $m \in \N$ and $\iota_1,...,\iota_m \in \{-,+ \}$. We consider the following expression :
    \begin{equation*}
        F(\omega) := \sum_{k_1,...,k_m} c_{k_1,...,k_m}(\omega) \cdot \prod_{j=1}^m g_{k_j}^{\iota_j}(\omega), \hspace{1cm} \omega \in \Omega
    \end{equation*}
    where, the $g_{k_j}(\omega)$ are complex standard i.i.d Gaussians, independent of the $\sigma$-algebra $\cB$, and the $c_{k_1,...,k_m}(\omega)$ are $\cB$-mesurable complex random variables. Then, there exists $C>0$ such that for every $p\geq 2$, we have :
    \begin{equation*}
        \norm{F}_{\L^p(\Omega | \cB)} \leq C p^{\frac{m}{2}} \norm{F}_{\L^2(\Omega | \cB)}
    \end{equation*}
    where $\L^p(\Omega | \cB)$ is the $\L^p$-norm conditioned to the $\sigma$- algebra $\cB$.
\end{lem}

In the energy estimates, we will apply this lemma with $m=3$, $\cB$ the $\sigma$-algebra generated by low-frequency Gaussians, and the random variables $c_{k_1,..,k_m}$ will be some multi-linear expression of high-frequency Gaussians (independent of the low-frequency Gaussians). For a proof of this lemma, we refer to \cite{Simon+1974} (see also \cite{tzvetkov2010construction} and \cite{thomann2010gibbs}).

\begin{lem}\label{L2 estimate with orthogonality}
    Let $(\Omega,\cA,\P)$ be a probability space. Let $m\in \N$, $\iota_1,...,\iota_m \in \{-,+ \}$, and $g_{k_1},...,g_{k_m}$ be complex standard i.i.d Gaussians. We consider the following multi-linear expression of Gaussians:
    \begin{equation}\label{mult lin gauss}
        G(\omega) := \sum_{\substack{k_1,...,k_m \\ \forall \iota_i \neq \iota_j, \hsp k_i \neq k_j}} c_{k_1,...,k_m} \cdot \prod_{j=1}^m g_{k_j}^{\iota_j}(\omega), \hspace{1cm} \omega \in \Omega 
    \end{equation}
    where $c_{k_1,...,k_m}$ is a sequence of $l^2(\Z^m;\C)$. Then, there exists $C>0$ such that:
    \begin{equation}\label{ineq L2 estimate orhtogonality}
        \norm{G}_{\L^2(\Omega)} \leq C \biggl( \sum_{\substack{k_1,...,k_m \\ \forall \iota_i \neq \iota_j, \hsp k_i \neq k_j}} | c_{k_1,...,k_m} |^2 \biggr)^{\frac{1}{2}}
    \end{equation}
\end{lem}

\begin{proof}
    For more readability, we perform the proof only for $m=3$ (and that is the value of $m$ we will use with this lemma). In addition, we assume that $\iota_1, \iota_3=1$ and $\iota_2=-1$ (the other cases are similar). Thus, we want to prove that:
    \begin{equation*}
        \big\|  \sum_{\substack{k_1,k_2,k_3 \\  k_2 \neq k_1,k_3}} c_{k_1,k_2,k_3} g_{k_1} \cjg{g_{k_2}} g_{k_3} \big\|_{\L^2(\Omega)}^2 \leq C \sum_{\substack{k_1,k_2,k_3 \\  k_2 \neq k_1,k_3}} | c_{k_1,k_2,k_3} |^2 
    \end{equation*}
    We start by expanding the left hand side:
    \begin{equation}\label{expanding L^2 norm}
         \big\|  \sum_{\substack{k_1,k_2,k_3 \\  k_2 \neq k_1,k_3}} c_{k_1,k_2,k_3} g_{k_1} \cjg{g_{k_2}} g_{k_3} \big\|_{\L^2(\Omega)}^2 =  \sum_{\substack{k_1,k_2,k_3,l_1,l_2,l_3 \\  k_2 \neq k_1,k_3 \hsp \& \hsp l_2 \neq l_1,l_3}} c_{k_1,k_2,k_3} \cjg{c_{l_1,l_2,l_3}} \hsp \E[g_{k_1} \cjg{g_{k_2}} g_{k_3} \cjg{g_{l_1}} g_{l_2} \cjg{g_{l_3}}]
    \end{equation}
    From the independence of the Gaussians $g_n$, and the fact that for a complex standard Gaussian variable $g$ we have $\E[g]=\E[g^2]=\E[g^3]=0$, the only non-zero contributions in the sum above is when for each $g_{k_j}^{\pm}$ there exists an $g_{l_i}^{\mp}$, with the opposite signature, such that $k_j = l_i$. Hence, the only non-zero contributions are of the form:
    \begin{center}
    $c_{k_1,k_2,k_3} \cjg{c_{l_1,l_2,l_3}} \hsp   \E[|g_{k_1}|^2|g_{k_2}|^2|g_{k_3}|^2]$ with $k_2=l_2$, $\{ k_1 , k_3\} = \{ l_1, l_3\}$ and $k_2 \neq k_1,k_3$
    \end{center}
    Then, we invoke the following set of indices:
    \begin{equation*}
        \begin{split}
        D_1 &:= \{ (k_1,k_2,k_3,l_1,l_2,l_3) \in \Z^6 \hsp : \hsp k_2 = l_2, \hsp k_1 = l_1,\hsp k_3 = l_3, \hsp   k_2 \neq k_1,k_3, \hsp k_1 \neq k_3 \} \\
        D_2 &:= \{ (k_1,k_2,k_3,l_1,l_2,l_3) \in \Z^6 \hsp : \hsp k_2 = l_2, \hsp k_1 = l_3,\hsp k_3 = l_1, \hsp   k_2 \neq k_1,k_3, \hsp k_1 \neq k_3 \} \\
        D_3 &:= \{ (k_1,k_2,k_3,l_1,l_2,l_3) \in \Z^6 \hsp : \hsp k_2 = l_2, \hsp k_1 = k_3 = l_1 = l_3, \hsp   k_2 \neq k_1,k_3\}
        \end{split}
    \end{equation*}
    so that, coming back to \eqref{expanding L^2 norm}, we have (for any complex standard Gaussian $g$):
    \begin{equation*}
        \begin{split}
        \big\|  \sum_{\substack{k_1,k_2,k_3 \\  k_2 \neq k_1,k_3}} c_{k_1,k_2,k_3} g_{k_1} \cjg{g_{k_2}} g_{k_3} \big\|_{\L^2(\Omega)}^2 &= \E[|g|^2]^3\sum_{D_1} |c_{k_1,k_2,k_3}|^2  + \E[|g|^2]^3\sum_{D_2} c_{k_1,k_2,k_3} \cjg{c_{k_3,k_2,k_1}} \\
        & +\E[|g|^2]\E[|g|^4] \sum_{D_3}  |c_{k_1,k_2,k_1}|^2 
        \end{split}
    \end{equation*}
    This concludes the proof since:
    \begin{equation*}
        \sum_{D_1} |c_{k_1,k_2,k_3}|^2 + \sum_{D_3}  |c_{k_1,k_2,k_1}|^2  = \sum_{\substack{k_1,k_2,k_3 \\ k_2\neq k_1,k_3 \hsp \& \hsp k_1 \neq k_3}}|c_{k_1,k_2,k_3}|^2 + \sum_{\substack{k_1,k_2 \\ k_2 \neq k_1}}|c_{k_1,k_2,k_1}|^2 =  \sum_{\substack{k_1,k_2,k_3 \\  k_2 \neq k_1,k_3}} | c_{k_1,k_2,k_3} |^2
    \end{equation*}
    and,
    \begin{equation*}
        \sum_{D_2} c_{k_1,k_2,k_3} \cjg{c_{k_3,k_2,k_1}} \leq \frac{1}{2} \sum_{\substack{k_1,k_2,k_3 \\ k_2\neq k_1,k_3 \hsp \& \hsp k_1 \neq k_3}} |c_{k_1,k_2,k_3}|^2 + \frac{1}{2} \sum_{\substack{k_1,k_2,k_3 \\ k_2\neq k_1,k_3 \hsp \& \hsp k_1 \neq k_3}} |c_{k_3,k_2,k_1}|^2 \leq \sum_{\substack{k_1,k_2,k_3 \\  k_2 \neq k_1,k_3}} | c_{k_1,k_2,k_3} |^2
    \end{equation*}
\end{proof}
\begin{rem}
    When $\iota_i \neq \iota_j$  \textbf{and} $\hsp k_i = k_j$, we say that $k_i$ and $k_j$ are paired. If one allows such parings in the sum in \eqref{mult lin gauss}, then the inequality \eqref{ineq L2 estimate orhtogonality} does not hold anymore in general. In our analysis in the upcoming sections, we will not encounter such pairings. However, in \cite{sun2023quasi}, this situation occurred, and the authors needed to study these pairing contributions separately. 
\end{rem}

\section{Proofs of the deterministic properties}
\label{section Proofs of the deterministic properties}
This section is dedicated to the proof of the deterministic properties of the energy correction $R_{s,N}$ and of the derivative of the modified energy $Q_{s,N}$. More precisely, we prove here Proposition~\ref{prop R_s,N continuous} and Proposition~\ref{prop continuity of Qs,N and deterministic estimate}. To do so, we are going to use the deterministic tools from Section~\ref{section Tools for the energy estimates}.\\

\paragraph{\textbf{Deterministic estimate for the energy correction}} In this paragraph, we prove Proposition~\ref{prop R_s,N continuous}. Recall that we want to estimate the sum of positive terms in \eqref{energy correction with absolute value}.

\begin{proof}[Proof of Proposition~\ref{prop R_s,N continuous}]
    Let $\sigma < s-\frac{1}{2}$ and let $u^{(1)},...,u^{(6)} \in \hsigt$. First, we decompose the sum in \eqref{energy correction with absolute value} dyadically:
    \begin{equation}\label{starting estimate R}
        \sum_{\substack{\linc \\ \Omgnz}} \big| \frac{\psi_{2s}(\vec{k})}  {\Omega(\vec{k})}\big||u^{(1)}_{k_1}u^{(2)}_{k_2}...u^{(6)}_{k_6}| = \sum_{N_1,...,N_6} \frkR(N_1,...,N_6)
    \end{equation}
    where the summations are performed on the dyadic values of $N_1,...,N_6$ and,
    \begin{equation*}
        \frkR(N_1,...,N_6) := \sum_{\substack{\linc \\ \Omgnz}} \big| \frac{\psi_{2s}(\vec{k})}  {\Omega(\vec{k})}\big| \prod_{j=1}^6 \1_{|k_j| \sim N_j} |u^{(j)}_{k_j}|
    \end{equation*}
    Now, using Lemma~\ref{an estimate using Strichartz's trick} yields:
    \begin{equation}\label{R dyadic estimate L2}
         \frkR(N_1,...,N_6) \lesssim_{\eps} N_{(1)}^{2s-2+\eps} N_{(3)}^2 \prod_{j=1}^{6} \norm{P_{N_j}u^{(j)}}_{\L^2(\T)} \lesssim_{\eps}N_{(1)}^{2s-2+\eps} N_{(3)}^2 (N_{(1)} ... N_{(6)})^{-\sigma}\prod_{j=1}^6 \hsignorm{u^{(j)}} 
    \end{equation}
    where $P_{N}$ is the projector onto frequencies $|k| \sim N$. Note that we have $N_{(1)} \sim N_{(2)}$ because the constraint $k_1-k_2+...-k_6 = 0$ implies that $ |k_{(2)}|  \sim |k_{(1)}|$. Besides, we crudely estimate $(N_{(4)}N_{(5)}N_{(3)})^{-\sigma} \lesssim 1$. Then, we rewrite the inequality above as: 
    \begin{equation}\label{R dyadic estimate Hsig}
         \frkR(N_1,...,N_6) \lesssim_{\eps} N_{(1)}^{2(s-1-\sigma)+\eps} N_{(3)}^{2-\sigma} \prod_{j=1}^6 \hsignorm{u^{(j)}}  
    \end{equation}
    Next, let us observe that for $\sigma < s -\frac{1}{2}$ and $\eps>0$: 

\begin{align*}
    &\begin{cases}
        2(s-1-\sigma)+\eps > -1 \\
        2(s-1-\sigma)+\eps \lra -1 \hspace{0.2cm} \text{as} \begin{cases}
            \sigma \ra s-\frac{1}{2} \\
            \eps \ra 0
        \end{cases}
    \end{cases} 
    & &\text{and,} & &\begin{cases}
        2-\sigma > \frac{5}{2} -s \\
        2-\sigma \lra \frac{5}{2} -s \hspace{0.2cm} \text{as}  \begin{cases}
            \sigma \ra s-\frac{1}{2} \\
            \eps \ra 0
        \end{cases}
    \end{cases}
\end{align*}
In particular, in \eqref{R dyadic estimate Hsig}, $N_{(1)}$ is accompanied by a negative exponent\footnote{In dimension $d\geq 2$, we have $\sigma <s-\frac{d}{2}$, and this scenario would be worse because we would have $2(s-1-\sigma)>0$, so $N_{(1)}$ would be accompanied by a positive exponent. This scenario is encountered in \cite{sun2023quasi}, with $d=3$.}. \\
Finally, if:
\begin{align*}
    &-1 + (\frac{5}{2} - s) < 0, & & \text{that is if:} & & s>\frac{3}{2}, 
\end{align*}
then, for $\sigma<s-\frac{1}{2}$ and $\eps>0$ respectively close enough to $s-\frac{1}{2}$ and 0, we deduce from \eqref{R dyadic estimate Hsig} that\footnote{The notation $N^{-}$ means that the power of $N$ is $-\gamma$ for a certain $\gamma>0$.}:
\begin{equation*}
    \frkR(N_1,...,N_6) \lesssim N_{(1)}^{-}  \cdot \prod_{j=1}^6 \hsignorm{u^{(j)}}
\end{equation*}
which is summable in the $N_j$. Coming back to \eqref{starting estimate R} finishes the proof. 
\end{proof}

\paragraph{\textbf{Deterministic estimate for the derivative of the modified energy}} In this paragraph, we prove Proposition~\ref{prop continuity of Qs,N and deterministic estimate}. Recall that we want to estimate the sums of positive terms in \eqref{T0 with abs value}, \eqref{T1 with abs value} and \eqref{T2 with abs value}.

\begin{proof}[Proof of Proposition~\ref{prop continuity of Qs,N and deterministic estimate}] Let $\sigma < s-\frac{1}{2}$. \\
    \textbullet \textbf{Estimate for $\cT_0$ :} The analysis is similar to the one for $\cR$. Indeed, let us fix $u^{(1)},...,u^{(6)} \in \hsigt$. Then, decomposing the sum in \eqref{T0 with abs value} dyadically and using Lemma~\ref{lemma psi estimate} (with $\Omgz$), leads to:
     \begin{equation*}
        \begin{split}
             \sum_{\substack{\linc \\ \Omgz}} | \psi_{2s} (\vec{k})| |u^{(1)}_{k_1}u^{(2)}_{k_2}...u^{(6)}_{k_6}|  & \lesssim \sum_{\substack{\linc \\ \Omgz}} |k_{(1)}|^{2s-2} |k_{(3)}|^2 |u^{(1)}_{k_1}u^{(2)}_{k_2}...u^{(6)}_{k_6}| \\
            & \lesssim \sum_{N_1,...,N_6} \frkT_0(N_1,...,N_6)
        \end{split}
    \end{equation*}
     where,
    \begin{equation*}
        \frkT_0(N_1,...,N_6) := \sum_{\substack{\linc \\ \Omgz}} N_{(1)}^{2s-2} N_{(3)}^2 \prod_{j=1}^6 \1_{|k_j| \sim N_j} |u^{(j)}_{k_j}|
    \end{equation*}
    Next, from Lemma~\ref{Strichartz's trick} (with $f^{j}_{k_j} = \1_{|k_j| \sim N_j} |u^{(j)}_{k_j}|$), we have:  
    \begin{equation*}
         \frkT_0(N_1,...,N_6) \lesssim_{\eps}  N_{(1)}^{2s-2+\eps} N_{(3)}^2 \prod_{j=1}^{6} \norm{P_{N_j}u^{(j)}}_{\L^2(\T)}
    \end{equation*}
From that point the proof goes the same as the one for $\cR$ (see the estimate \eqref{R dyadic estimate L2}).\\

\textbullet \textbf{Estimate of $\cT_j$ for $j=1,2$:} We only prove the estimate \eqref{T1 with abs value} since the analysis for the estimate \eqref{T2 with abs value} is similar. Here, the computations are a little more delicate. For the sake of readability, we restrict ourselves to prove the estimate:
\begin{equation*}
    \sum_{\substack{\linc \\ \lincba \\\Omgnz}} \big| \frac{\psi_{2s}(\vec{k})}{ \Omega(\vec{k})} \big| | u_{p_1}\cjg{u_{p_2}}...u_{p_5}\cjg{u_{k_2}}...\cjg{u_{k_6}}| \leq C_s \hsignorm{u}^{10}
\end{equation*}
In other words, we just prove the estimate \eqref{T1 with abs value} where all the $u^{(j)}$ and $v^{(l)}$ are equal to a single $u$.
To do so, we will use the following lemma:
\begin{lem}\label{technical lem Q estimate}
    There exists a constant $C > 0$ such that for any dyadic integers $P_1,...,P_5$ and any sequences $\left(g^{P_1}_{p_1} \right)_{p_1 \in \Z}, ... , \left(g^{P_5}_{p_5} \right)_{p_5 \in \Z} $ of complex numbers that satisfy for $j=1,...,5$, $g^{P_j}_{p_j} = \1_{|p_j| \sim P_j} g_{p_j}$, if we set for $k_1 \in \Z$ :
    \begin{equation*}
        f_{k_1} := \sum_{\lincba}  g_{p_1}^{P_1}g_{p_2}^{P_2}g_{p_3}^{P_3}g_{p_4}^{P_4}g_{p_5}^{P_5}
    \end{equation*}
    Then,
    \begin{equation*}
        \norm{f}_{l^2(\Z)} \leq C \left( P_{(2)} P_{(3)}P_{(4)} P_{(5)} \right)^{\frac{1}{2}} \prod_{j=1}^5\norm{g^{P_j}}_{l^2(\Z)}
    \end{equation*}
\end{lem}
For clarity, we postpone the proof of this lemma for the end of this section. Now we are ready to prove the estimate. Once again, we start by decomposing the sum as:
\begin{equation*}
    \sum_{\substack{\linc \\ \lincba \\\Omgnz}} \big| \frac{\psi_{2s}(\vec{k})}{ \Omega(\vec{k})} \big| | u_{p_1}\cjg{u_{p_2}}...u_{p_5}\cjg{u_{k_2}}...\cjg{u_{k_6}}| = \sum_{\substack{M_1,...,M_6 \\ P_1,...,P_5}} \frkT_1(M_1,...,M_6,P_1,...,P_5)
\end{equation*}
 where $M_1,...,M_6,P_1,...,P_5$ are dyadic-valued and
    \begin{equation}\label{T_1 dyadic block}
        \frkT_1(M_1,...,M_6,P_1,...,P_5) := \sum_{\substack{\linc \\ \lincba \\\Omgnz}} \big| \frac{\psi_{2s}(\vec{k})}{ \Omega(\vec{k})} \big| |u_{p_1}^{P_1}...u_{p_5}^{P_5}u_{k_2}^{M_2}...u_{k_6}^{M_6}| \cdot \1_{|k_1| \sim M_1}
    \end{equation}
    with the practical notation $u_l^N := \1_{|l| \sim N} u_l$. Next, we rewrite $\frkT_1$ as:
    \begin{equation}\label{T_1 rewritten}
        \frkT_1(M_1,...,P_5) =  \sum_{\substack{\linc \\ \Omgnz}} \big| \frac{\psi_{2s}(\vec{k})}{\Omega(\vec{k})}\big| \prod_{j=1}^6 | f_{k_j}^{(j)} |
    \end{equation}
     where, for $j=2,...,6$ we denote $f_{k_j}^{M_j} := u_{k_j}^{M_j}$, and:
    \begin{equation*}
        f_{k_1}^{(1)} := \1_{|k_1| \sim M_1} \cdot \sum_{\lincba}  |u_{p_1}^{P_1}...u_{p_5}^{P_5}|
    \end{equation*}
    At this stage, let us recall that $M_{(1)}\geq... \geq M_{(6)}$ and $P_{(1)} \geq ... \geq P_{(5)}$ are respectively a non-increasing rearrangement of $M_1,...,M_6$ and $P_1,...,P_5$. We also introduce:
    \begin{center}
        $N_{(1)} \geq N_{(2)} \geq ... \geq N_{(10)}$ a non-increasing rearrangement of $M_2,...,M_6,P_1,...,P_5$
    \end{center}
    Note that the constraints in the sum in $\frkT_1$ imply that 
    \begin{center}
    $M_{(1)} \sim M_{(2)}$, and: $P_{(1)} \sim P_{(2)}$ or $P_{(1)} \sim M_1$, and: $N_{(1)}\sim N_{(2)}$.
    \end{center}
    Now, we use Lemma~\ref{an estimate using Strichartz's trick} to estimate $\frkT_1$ in \eqref{T_1 rewritten}; we obtain that for any $\eps>0$:
    \begin{equation*}
            \frkT_1(M_1,...,P_5)  \lesssim_{\eps} M_{(1)}^{2s-2+\eps} M_{(3)}^{2} \prod_{j=1}^6 \norm{f^{(j)}}_{l^2}
              \lesssim_{\eps} M_{(1)}^{2s-2+\eps} M_{(3)}^{2}\norm{f^{(1)}}_{l^2} \prod_{j=2}^6 \norm{u^{M_j}}_{\L^2}
    \end{equation*}
    Applying the well-suited estimate of $\norm{f^{(1)}}_{l^2}$ from Lemma~\ref{technical lem Q estimate} leads to
    \begin{equation}\label{dyadic estimate T_1}
        \frkT_1(M_1,...,P_5)  \lesssim_{\eps}  M_{(1)}^{2s-2+\eps} M_{(3)}^{2} (P_{(2)} P_{(3)} P_{(4)} P_{(5)})^{\frac{1}{2}} \prod_{j=1}^5 \norm{u^{P_j}}_{\L^2} \prod_{j=2}^6 \norm{u^{M_j}}_{\L^2}
    \end{equation}
    and it follows that
    \begin{equation*}
        \frkT_1(M_1,...,P_5)  \lesssim_{\eps}  M_{(1)}^{2s-2+\eps} M_{(3)}^{2} \left( M_2 M_3 M_4 M_5 M_6\right)^{- \sigma} P_{(1)}^{-\sigma} \left( P_{(2)} P_{(3)} P_{(4)}P_{(5)}\right)^{\frac{1}{2} - \sigma} \hsignorm{u}^{10} 
    \end{equation*}
    In the remaining part of the proof, we show that it implies that:
    \begin{equation}\label{T_1 desired dyadic estimate}
        \frkT_1(M_1,...,P_5) \lesssim_{\eps} N_{(1)}^{-} \hsignorm{u}^{10}
    \end{equation}
   for $\sigma < s-\frac{1}{2}$ close enough to $s-\frac{1}{2}$, which will complete the estimate for $\cT_1$. We start from the inequality above and we use the fact that $M_1 \lesssim P_{(1)}$ due to the constraint $\lincba$, along with the rough estimates $( M_{(4)}M_{(5)}M_{(6)})^{-\sigma} \leq 1$ and $( P_{(3)} P_{(4)} P_{(5)})^{\frac{1}{2}-\sigma} \leq 1$, in order to obtain:
    \begin{equation}\label{T_1 dyadic estimate in hisg}
        \begin{split}
        \frkT_1(M_1,...,P_5) &\lesssim_{\eps} M_{(1)}^{2s-2+\eps} M_{(3)}^{2} (M_1 M_2 M_3 M_4 M_5 M_6)^{- \sigma} (M_1^{\sigma}P_{(1)}^{-\sigma}) ( P_{(2)} P_{(3)} P_{(4)}P_{(5)})^{\frac{1}{2} - \sigma} \hsignorm{u}^{10} \\
        & \lesssim_{\eps} M_{(1)}^{2(s-1-\sigma)+\eps} M_{(3)}^{2-\sigma} P_{(2)}^{\frac{1}{2} - \sigma} \hsignorm{u}^{10}
        \end{split}
    \end{equation}
    Besides, when $\sigma < s - \frac{1}{2}$ and $\eps > 0$ are respectively arbitrarily close to $s - \frac{1}{2}$ and $0$, we have that $2(s-1-\sigma) + \eps$, $2-\sigma$ and $\frac{1}{2} - \sigma$ are respectively arbitrarily close to 
\begin{align*}
    2(s-1-(s-\frac{1}{2})) &= -1, & 2-(s - \frac{1}{2}) &= \frac{5}{2} -s, & \frac{1}{2} - (s - \frac{1}{2}) = 1-s
\end{align*}
    Hence, $M_{(1)}$ is accompanied by a negative exponent, $P_{(2)}$ is accompanied by a negative exponent (since $s>\frac{3}{2}>1$), and $M_{(3)}$ is accompanied by a negative exponent when $s\geq\frac{5}{2}$ and  by a non-negative one when $s<\frac{5}{2}$. \\
    
    --In the situation where $M_{(1)} \sim N_{(1)}$, we use the rough estimate $P_{(2)}^{\frac{1}{2}-\sigma}\leq 1 $ in \eqref{T_1 dyadic estimate in hisg}, and we obtain:
     \begin{equation*}
         \frkT_1(M_1,...,P_5)\lesssim_{\eps} N_{(1)}^{2(s-1-\sigma)+\eps} M_{(3)}^{2-\sigma}\hsignorm{u}^{10}
     \end{equation*}
    which is conclusive as far as 
\begin{equation*}
    2(s-1-(s-\frac{1}{2})) + 2 - (s-\frac{1}{2}) < 0
\end{equation*}
that is, as far as $s > \frac{3}{2}$. \\
--In the other situation, where $M_{(1)} \ll N_{(1)}$, we necessarily have $P_{(1)} \sim P_{(2)} \sim N_{(1)}$, so we deduce from \eqref{T_1 dyadic estimate in hisg} that:
\begin{equation*}
    \frkT_1(M_1,...,P_5) \lesssim_{\eps} N_{(1)}^{\frac{1}{2}-\sigma} M_{(3)}^{2(s-1-\sigma) + 2-\sigma +\eps} \hsignorm{u}^{10}
\end{equation*}
Then, for $\sigma < s - \frac{1}{2}$ and $\eps > 0$ respectively arbitrarily close to $s - \frac{1}{2}$ and $0$, we have:
\begin{equation*}
    \frkT_1(M_1,...,P_5) \lesssim N_{(1)}^{(1-s)+} M_{(3)}^{(\frac{3}{2}-s)+} \hsignorm{u}^{10} \lesssim N_{(1)}^{-} \hsignorm{u}^{10}
\end{equation*}
since $s>\frac{3}{2}$.\\ 
In conclusion, \eqref{T_1 desired dyadic estimate} is proven and the proof of Proposition~\ref{prop continuity of Qs,N and deterministic estimate} is completed.
\end{proof}

We finish this section by a proof of Lemma~\ref{technical lem Q estimate}.
\begin{proof}[Proof of Lemma~\ref{technical lem Q estimate}]
    We have to estimate : 
    \begin{equation*}
        \norm{f}_{l^2(\Z)}^2 = \sum_{k_1 \in \Z} \bigg|  \sum_{\lincba} g_{p_1}^{P_1}g_{p_2}^{P_2}g_{p_3}^{P_3}g_{p_4}^{P_4}g_{p_5}^{P_5}  \bigg|^2
    \end{equation*}
    Using the Cauchy-Schwarz inequality we obtain:
    \begin{equation*}
            \norm{f}_{l^2(\Z)}^2 \leq \sum_{k_1 \in \Z} \bigl( \sum_{p_1,...,p_5} \1_{\lincba}  \prod_{j=1}^5 \1_{|p_j| \sim P_j} \bigr)
            \bigl( \sum_{p_1,...,p_5}  \1_{\lincba} |g_{p_1}^{P_1}g_{p_2}^{P_2}g_{p_3}^{P_3}g_{p_4}^{P_4}g_{p_5}^{P_5}|^2\bigr)
    \end{equation*}
    Plugging the counting bound \eqref{elementary counting bound} into this formula yields 
    \begin{equation*}
        \norm{f}_{l^2(\Z)}^2 \lesssim P_{(2)}P_{(3)}P_{(4)}P_{(5)}  \sum_{p_1,...,p_5}  |g_{p_1}^{P_1}g_{p_2}^{P_2}g_{p_3}^{P_3}g_{p_4}^{P_4}g_{p_5}^{P_5}|^2  \underbrace{\big( \sum_{k_1 \in \Z} \1_{\lincba} \big)}_{\leq 1}
    \end{equation*}
    Thus,
    \begin{equation*}
        \norm{f}_{l^2(\Z)}^2 \lesssim P_{(2)}P_{(3)}P_{(4)}P_{(5)} \prod_{j=1}^5 \norm{g^{P_j}}_{l^2(\Z)}^2
    \end{equation*}
    which is the desired estimate.
\end{proof}

\section{Estimates for the weight of the weighted Gaussian measures}\label{section Estimates for the weight of the weighted Gaussian measures}
This section is dedicated to estimates on the weight $\1_{\{ \cC(u) \leq R \} } e^{-R_{s,N}(u)}$ of the weighted Gaussian measure $\rho_{s,R,N}$ (defined in \eqref{weighted Gaussian measure}). In particular, we show that : 
\begin{equation*}
    \1_{\{ \cC(u) \leq R \} } e^{-R_{s,N}(u)} \in \L^1(d\mu_s)
\end{equation*}
which ensures that $\rho_{s,R,N}$ is a finite measure on $\hsigt$, $\sigma < s - \frac{1}{2}$. 
More precisely, we prove Proposition~\ref{prop estimate for the weight of the weighted measure}. Yet, before doing so, we will need the two following lemmas :

\begin{lem}\label{lem R_s,N L^p estimate}
    Let $s > \frac{3}{2}$. Then, there exists $\beta \in (0,1)$ such that for every $R>0$, there exists $C(s,R)>0$, such that for any $p \in [2,+\infty)$,
    \begin{equation*}
       \norm{\1_{\{ \cC(u) \leq R \} } R_{s,N}}_{\Lp(d\mu_s)} \leq C(s,R)p^\beta
    \end{equation*}
    uniformly in $N \in \N \cup \{ \infty \}$.
\end{lem}
Recall that we denote $R_{s,\infty} = R_s$ for convenience.
\begin{lem}\label{csqce control on all Lp norm}
    Let $(X,\cA,d\nu)$ be a probability space and let $F : X \lra \C$ be a mesurable function. Assume that there exist constants $C_0 > 0$ and $\beta \in (0,1)$ such that for every $p \in [2,+\infty)$,
    \begin{equation*}
        \norm{F}_{\Lp(d\nu)} \leq C_0 p^\beta
    \end{equation*}
    Then, there exist $\delta > 0$ and $C_1>0$ only depending on $\beta$ and $C_0$ such that :
    \begin{equation*}
        \int_X e^{\delta |F(v)|^{\frac{1}{\beta}}} d\nu(v) \leq C_1
    \end{equation*}
\end{lem} 
We will only prove Lemma~\ref{lem R_s,N L^p estimate}, Lemma~\ref{csqce control on all Lp norm} being just a slightly different version of Lemma 4.5 from \cite{tzvetkov2010construction} (where a proof is given). \bigskip

Before doing the proof of Lemma~\ref{lem R_s,N L^p estimate}, let us briefly show how it implies Proposition~\ref{prop estimate for the weight of the weighted measure} when combining with Lemma~\ref{csqce control on all Lp norm}.

\begin{proof}[Proof of Proposition~\ref{prop estimate for the weight of the weighted measure}] Here, we assume the statements in Lemma~\ref{lem R_s,N L^p estimate} and Lemma~\ref{csqce control on all Lp norm}.\\
--We start by proving the first statement in Proposition~\ref{prop estimate for the weight of the weighted measure}. Lemma~\ref{lem R_s,N L^p estimate} shows that $F:=\1_{\{ \cC(u) \leq R \} } R_{s,N}(u) : \hsigt \lra \C$ satisfies the assumptions of Lemma~\ref{csqce control on all Lp norm} (with $\nu = \mu_s$). Then, applying Lemma~\ref{csqce control on all Lp norm}, we obtain that there exist $\delta = \delta(s,R) > 0$ and $C_1(s,R) > 0$ such that :
\begin{equation*}
    \int_{\hsig} e^{\delta \left|  \1_{\{ \cC(u) \leq R \} } R_{s,N}(u) \right|^{\frac{1}{\beta}}} d\mu_s \leq C_1(s,R)
\end{equation*}
A fortiori we have, 
\begin{equation*}
    \int_{\hsig} \1_{\{ \cC(u) \leq R \} } e^{\delta \left|  R_{s,N}(u) \right|^{\frac{1}{\beta}}} d\mu_s \leq C_1(s,R)
\end{equation*}
Thanks to the fact that $\frac{1}{\beta} > 1$, we deduce from the above inequality that for every $p\in [1,+\infty)$ : 
\begin{equation*}
     \norm{\1_{\{ \cC(u) \leq R \} } e^{|R_{s,N}(u)|}}^p_{\Lp(d\mu_s)} = \int_{\hsig} \1_{\{ \cC(u) \leq R \} } e^{p |R_{s,N}(u)|} d\mu_s \leq C(s,R,p) < +\infty
\end{equation*}
for a certain constant $C(s,R,p)>0$.\\
--As a consequence, we are now able to prove the second statement of Proposition~\ref{prop estimate for the weight of the weighted measure}, that is:  
\begin{equation*}
        \norm{\1_{\{ \cC(u) \leq R \} } e^{-R_s(u)}-\1_{\{ \cC(u) \leq R \} } e^{-R_{s,N}(u)}}_{\L^p(d\mu_s)} \tendsto{N \ra \infty} 0
    \end{equation*}
Firstly, we know from Proposition \ref{prop approx R_s by R_s,N on compact sets} that we have the pointwise convergence:
\begin{equation*}
    \1_{ \{\cC(u) \leq R\}}R_s(u) = \1_{ \{\cC(u) \leq R\}} \cdot \frac{1}{6} \Re \sum_{\substack{\linc \\ \Omgnz}} \frac{\psi_{2s}(\vec{k})}{ \Omega(\vec{k})} u_{k_1}\cjg{u_{k_2}}...\cjg{u_{k_6}} = \lim_{N\ra \infty} \1_{ \{\cC(u) \leq R\}} R_{s,N}(u)
\end{equation*}
Consequently, from the continuity of the exponential, we also have the pointwise convergence:
\begin{equation*}
   \1_{\{ \cC(u) \leq R \} } e^{-R_s(u)}  = \lim_{N \ra \infty} \1_{\{ \cC(u) \leq R \} } e^{-R_{s,N}(u)}
\end{equation*}
In particular, $\1_{\{ \cC(u) \leq R \} } e^{-R_{s,N}(u)}$ converges to $\1_{\{ \cC(u) \leq R \} } e^{-R_s(u)}$ in measure. In addition, for a fixed $q\in(1,+\infty)$, we just proved that the functions $\1_{\{ \cC(u) \leq R \} } e^{-R_{s,N}(u)}$ are uniformly bounded in $\L^q(d\mu_s)$ (with respect to $N \in \N$). Then, (using the same argument as the one from the proof of Proposition \ref{prop Gs,N cvg Gs in Lp}) we can conclude that $\1_{\{ \cC(u) \leq R \} } e^{-R_{s,N}(u)}$ converges to $\1_{\{ \cC(u) \leq R \} } e^{-R_s(u)}$ in $\L^p(d\mu_s)$ for any $p\in [1,q)$; and since $q \in (1,+\infty)$ is arbitrary, the convergence holds for any $p\in[1,+\infty)$. This completes the proof of Proposition~\ref{prop estimate for the weight of the weighted measure}.
\end{proof}

However, it remains to prove Lemma~\ref{lem R_s,N L^p estimate}. Our analysis will rely on a decomposition into two parts of $R_{s,N}$. We will be able to treat the first part deterministically thanks to suitable "exchanges of derivatives". For the second part, those "exchanges of derivatives" will fail because we will be in a \textit{high-high-high-low-low-low} regime of frequency (where the three highest frequencies in the sum \eqref{R_s,N diag multilinear form} defining $R_{s,N}$ are in fact much more higher than the three lowest frequencies). Instead, we will handle the second part using the independence between the high frequency Gaussians and the low frequency Gaussians, using Wiener chaos estimate. 

\subsection{Decomposition} Recall that we have (see Section~\ref{section Deterministic properties of the energy correction}, Proposition \ref{prop R_s,N continuous}):
\begin{equation*}
        R_{s,N}(u) = \frac{1}{6} \Re \hsp \cR(w) 
    \end{equation*}
    where $w:=\Pi_N u$ (with the convention $\Pi_{\infty}=id$), and
    \begin{equation*}
        \cR(w) := \sum_{\substack{\linc \\ \Omgnz}} \frac{\psi_{2s}(\vec{k})}{\Omega(\vec{k})}w_{k_1}\cjg{w_{k_2}}...\cjg{w_{k_6}}
    \end{equation*}

It suffices to show the estimate of Lemma~\ref{lem R_s,N L^p estimate} for $\1_{\{ \cC(u) \leq R \}} \cR(w)$ because $ |R_{s,N}(u)| \leq \frac{1}{6} | \cR(w) | $. \\

Next, we split the set of indices over which we sum. We invoke the following sets of indices :
\begin{equation}
    \begin{split}
    \Lambda_D := \{ (k_1,...,k_6) \in \Z^6 : \hsp & \sum_{j=1}^6 (-1)^{j-1}k_j = 0, \hsp \sum_{j=1}^6 (-1)^{j-1}k_j^2 \neq 0, \\
    & \hsp |k_{(3)}| < |k_{(1)}|^{1-\delta_0} \hsp \textbf{or} \hsp |k_{(4)}| \geq |k_{(3)}|^{\delta_0}  \}
    \end{split}
\end{equation}
and,
\begin{equation}\label{Lambda_W}
    \begin{split}
    \Lambda_W := \{ (k_1,...,k_6) \in \Z^6 : \hsp & \sum_{j=1}^6 (-1)^{j-1}k_j = 0, \hsp \sum_{j=1}^6 (-1)^{j-1}k_j^2 \neq 0, \\
    & \hsp |k_{(3)}| \geq |k_{(1)}|^{1-\delta_0} \hsp \textbf{and} \hsp |k_{(4)}| < |k_{(3)}|^{\delta_0}  \}
    \end{split}
\end{equation}
where $\delta_0 \in (0,1)$\footnote{We will see in the forthcoming proof that there is no constraint on $\delta_0$, so we can chose it as any number in $(0,1)$}. Then, we split $\cR$ as :
\begin{equation*}
    \cR(w) = \cR^{(D)}(w) + \cR^{(W)}(w)
\end{equation*}
where,
\begin{align*}
    \cR^{(D)}(w) &:= \sum_{\Lambda_D} \frac{\psi_{2s}(\vec{k})}{\Omega(\vec{k})}w_{k_1}\cjg{w_{k_2}}...\cjg{w_{k_6}}, & \cR^{(W)}(w) := \sum_{\Lambda_W} \frac{\psi_{2s}(\vec{k})}{\Omega(\vec{k})}w_{k_1}\cjg{w_{k_2}}...\cjg{w_{k_6}}
\end{align*}
To estimate $\cR^{(D)}$ we will use the deterministic tools from Section~\ref{section Tools for the energy estimates}. To estimate $\cR^{(W)}$, we note that in the sum we are in a \textit{high-high-high-low-low-low} regime because $|k_{(1)}|,|k_{(2)}|,|k_{(3)}| \gg |k_{(4)}|,|k_{(5)}|,|k_{(6)}|$. It will then be possible to make use of the independence between the Gaussians $g_{k_{(1)}},g_{k_{(2)}},g_{k_{(3)}}$ and $g_{k_{(4)}}, g_{k_{(5)}}, g_{k_{(6)}}$ via Wiener chaos. \\

\subsubsection{Absence of pairing:}

\begin{defn}[Pairing]\label{def pairing}
    Consider a constraint under the form :
    \begin{equation*}
        \eps_1 k_1 + \eps_2 k_2 + ... + \eps_m k_m = 0 
    \end{equation*}
    where $k_j \in \Z$ and $\eps_j \in \{-,+ \}$. We say that $k_j$ and $k_l$ are paired if 
    \begin{equation*}
        \eps_j k_j + \eps_l k_l = 0
    \end{equation*}
\end{defn}

\begin{rem}\label{rem no high freq pairing, first generation}
    In $\Lambda_W$ (see \eqref{Lambda_W} above), there is (for $k_{(1)}$ large enough) no pairing within the three highest frequencies (relatively to the constraint $k_1 - k_2 + ... -k_6 = 0$). Indeed, suppose there is a pairing between two of the three highest frequencies. Then, the constraint would take the form:
    \begin{center}
        the remaining high frequency = sum of three low frequencies
    \end{center}
    which is impossible because in $\Lambda_W$ we have:
    \begin{center}
        $|$high frequencies$|$ $\gg$ $|$low frequencies$|$
    \end{center}
\end{rem}

\subsection{Proof of the estimate}
We are now ready to prove Lemma~\ref{lem R_s,N L^p estimate}.

\begin{proof}[Proof of Lemma~\ref{lem R_s,N L^p estimate}]
We organize the proof in two steps. In the first step, we provide a (deterministic) estimate for $\1_{\{ \cC(u) \leq R \} } \cR$. In the second step, we show that this estimate is only conclusive for the contribution $\cR^{(D)}$, and that for the remaining contribution $\cR^{(W)}$, further analysis using Wiener chaos is required.\\

    \underline{Step 1, Deterministic estimate :}  We rely on the estimate \eqref{R dyadic estimate L2} we obtained in Section~\ref{section Proofs of the deterministic properties} (which we do not reprove here). Thus, we start our analysis from the following estimates:
    \begin{align}\label{bound R_N1...N6}
        \left| \cR(w) \right| &\lesssim \sum_{N_1,...,N_6} \frkR(N_1,...,N_6) & &\text{and:} & \frkR(N_1,...,N_6) & \lesssim_{\eps} N_{(1)}^{2s-2+\eps}N_{(3)}^2 \prod_{j=1}^6 \norm{P_{N_j}w}_{\L^2(\T)}
    \end{align}
    where $N_1,...,N_6$ are dyadic valued, $P_N$ is the projector onto frequencies $|k| \sim N$, and:
    \begin{equation*}
        \frkR(N_1,...,N_6) := \sum_{\substack{\linc \\ \Omgnz}} \Big| \frac{\psi_{2s}(\vec{k})}{\Omega(\vec{k}) } \Big| \prod_{j=1}^6 \1_{|k_j| \sim N_j} |w_{k_j}|
    \end{equation*}
    We decompose $\frkR$ as $\frkR = \frkR^{(D)} + \frkR^{(W)}$, according to the decomposition of $\cR$.  \\
    
 Now, let $1<\sigma'<\sigma<s-\frac{1}{2}$. We have: 
    \begin{equation}\label{P_Nj w L2}
       \prod_{j=1}^6 \norm{P_{N_j}w}_{\L^2(\T)} \lesssim \prod_{j=1}^6 \norm{P_{N_j}u}_{\L^2(\T)} \lesssim (N_{(1)}^{-\sigma}\hsignorm{u})(N_{(2)}^{-\sigma'} \norm{u}_{H^{\sigma'}})(N_{(3)}N_{(4)}N_{(5)}N_{(6)})^{-1} \norm{u}_{H^1}^4
    \end{equation}
   Recall that the constraints in the sum in $\frkR$ above implies that $N_{(1)} \sim N_{(2)}$. Moreover, if we write: 
    \begin{equation*}
        \sigma' = \alpha 1 + (1-\alpha) \sigma
    \end{equation*}
    for $\alpha \in (0,1)$, then, by interpolating $H^{\sigma'}$ between $H^1$ and $\hsig$, we obtain from \eqref{P_Nj w L2} that:
    \begin{equation*}
        \begin{split}
            \prod_{j=1}^6 \norm{P_{N_j}u}_{\L^2(\T)} & \lesssim N_{(1)}^{-2\sigma +\alpha(\sigma-1)}(N_{(3)}N_{(4)}N_{(5)}N_{(6)})^{-1} \norm{u}_{H^1}^{4+\alpha} \norm{u}^{2-\alpha}_{\hsig} \\
            & \lesssim N_{(1)}^{-2\sigma +\alpha(s-\frac{3}{2})}(N_{(3)}N_{(4)}N_{(5)}N_{(6)})^{-1} \norm{u}_{H^1}^{4+\alpha} \norm{u}^{2-\alpha}_{\hsig}, \hspace{0.3cm} (\textnormal{because $\alpha(\sigma-1) < \alpha(s-3/2)$ })
        \end{split}
    \end{equation*}
    Plugging this into \eqref{bound R_N1...N6}, we obtain 
    \begin{equation*}
        \frkR(N_1,...,N_6) \lesssim_{\eps} N_{(1)}^{2(s-\sigma-1)+\eps+\alpha(s-\frac{3}{2})}N_{(3)} (N_{(4)}N_{(5)}N_{(6)})^{-1} \norm{u}_{H^1}^{4+\alpha} \norm{u}^{2-\alpha}_{\hsig}
    \end{equation*}
    In particular, thanks to Remark \ref{rem cut-off H1},
    \begin{equation*}
        \1_{\{ \cC(u) \leq R \}} \frkR(N_1,...,N_6) \lesssim_{\eps} N_{(1)}^{2(s-\sigma-1)+\eps+\alpha(s-\frac{3}{2})}N_{(3)} (N_{(4)}N_{(5)}N_{(6)})^{-1} R^{4+\alpha} \norm{u}^{2-\alpha}_{\hsig}
    \end{equation*}
    And since we will not need the smallness provided by $(N_{(5)}N_{(6)})^{-1}$, we simply write :
      \begin{equation*}
        \1_{\{ \cC(u) \leq R \}} \frkR(N_1,...,N_6) \lesssim_{\eps,R} N_{(1)}^{2(s-\sigma-1)+\eps+\alpha(s-\frac{3}{2})}N_{(3)} N_{(4)}^{-1} \norm{u}^{2-\alpha}_{\hsig}
    \end{equation*}
    Using now \eqref{Gaussian moments}, we conclude that 
    \begin{equation}\label{cR deterministic estimate}
        \norm{\1_{\{ \cC(u) \leq R \}} \frkR(N_1,...,N_6)}_{\Lp(d\mu_s)} \leq C N_{(1)}^{2(s-\sigma-1)+\eps+\alpha(s-\frac{3}{2})}N_{(3)} N_{(4)}^{-1} p^{\frac{2-\alpha}{2}}
    \end{equation}
    Where the constant $C>0$ depends on $R$, $s$, $\sigma$ and $\eps$.
    This is the estimate we will work with later in the proof. We stress the fact that \eqref{cR deterministic estimate} is true for $\frkR$ and also for $\frkR^{(D)}$ and $\frkR^{(W)}$ (with the exact same proof). \bigskip

    \underline{Step 2, Estimates for $\cR^{(D)}$ and $\cR^{(W)}$:} \\
    \textbullet We start with $\cR^{(D)}$. Recalling that $\sigma < s -\frac{1}{2}$, we see that the exponent of $N_{(1)}$ in \eqref{cR deterministic estimate} satisfies:
    \begin{equation}\label{power of N(1)}
        \begin{cases}
            2(s-\sigma-1)+\eps+\alpha(s-\frac{3}{2}) > -1 \\
            2(s-\sigma-1)+\eps+\alpha(s-\frac{3}{2}) \lra -1 \hspace{0.2cm} \text{as} \begin{cases}
                \eps \ra 0 \\
                \alpha \ra 0 \\
                \sigma \ra s -\frac{1}{2}
            \end{cases}
        \end{cases}
    \end{equation}
    Moreover, for a fixed $\delta_0 \in (0,1)$, the conditions on the indices lying in $\Lambda_D$ imply that $N_{(3)} \lesssim N_{(1)}^{1-\delta_0}$ or $N_{(4)} \gtrsim N_{(3)}^{\delta_0}$. In both cases, we deduce from \eqref{cR deterministic estimate} that for $\eps$ and $\alpha$ close enough to $0$ and $\sigma$ close enough to $s-\frac{1}{2}$, we have\footnote{The notation $N^{-}$ means that the power of $N$ is $-\gamma$ for a certain $\gamma>0$.}
    \begin{equation*}
        \norm{\1_{\{ \cC(u) \leq R \}} \frkR^{(D)}(N_1,...,N_6)}_{\Lp(d\mu_s)} \leq C N_{(1)}^-p^{\frac{2-\alpha}{2}}
    \end{equation*}
    Consequently,
        \begin{equation*}
        \norm{\1_{\{ \cC(u) \leq R \}} \cR^{(D)}}_{\Lp(d\mu_s)} \leq C p^{\frac{2-\alpha}{2}}
    \end{equation*}
   
\textbullet We continue our analysis with the term $\cR^{(W)}$. Once again, we start by decomposing $\cR^{(W)}$ dyadically as $\cR^{(W)}(w) = \sum_{N_1,...,N_6}  \cR_{N_1,...,N_6}^{(W)}(w)$, where\footnote{For the analysis below, we need to keep the complex conjugation bars, so here we don't use $\frkR^{(W)}$}:
    \begin{equation*}
        \cR_{N_1,...,N_6}^{(W)}(w) := \sum_{\substack{\linc \\ \Omgnz \\ |k_{(4)}| \leq |k_{(3)}|^{\delta_0} }} \frac{\psi_{2s}(\vec{k})}{\Omega(\vec{k})}w_{k_1}\cjg{w_{k_2}}...\cjg{w_{k_6}} \prod_{j=1}^6 \1_{|k_j| \sim N_j}
    \end{equation*}

Henceforth, we denote $w_{k_j}^{N_j} := \1_{|k_j|\sim N_j} w_{k_j}$ for better readability. \\
    
    Without loss of generality, we assume that $\{N_1,N_2,N_3 \} = \{N_{(1)}, N_{(2)}, N_{(3)} \}$, meaning that the three highest frequencies are $k_1,k_2,k_3$. The other cases are similar or simpler. \\
    We denote $\cB_{\ll N_{(3)}}$ the $\sigma$-algebra generated by the Gaussians $\left( g_k \right)_{|k|\leq N_{(3)}/100}$. We only need to consider the contribution when $N_{(1)}$ is large because when $N_{(1)}$ is small, all the $N_j$'s are small, and we can use \eqref{cR deterministic estimate} without fearing issues of summability in the $N_j$'s. In particular, we assume that $N_{(1)}$ is sufficiently large so that $N_{(3)}$ (which is $\gtrsim N_{(1)}^{1-\delta_0}$) is large enough to ensure that $N_{(4)}$ (which is $\lesssim N_{(3)}^{\delta_0}$) is $\leq N_{(3)}/100$. As a consequence, we have that the random variables :
    \begin{equation*}
        \begin{split}
           & w_{k_4}^{N_4},w_{k_4}^{N_4},w_{k_5}^{N_6} \hspace{0.2cm} \textnormal{are $\cB_{\ll N_{(3)}}$mesurable} \\
           & w_{k_1}^{N_1},w_{k_2}^{N_2},w_{k_3}^{N_3} \hspace{0.2cm} \textnormal{are independent of $\cB_{\ll N_{(3)}}$}
        \end{split}
    \end{equation*}
    With this set up, we can now begin the analysis. Thanks to Remark \ref{rem cut-off H1}, we have 
    \begin{equation*}
        \begin{split}
            \norm{\1_{\{ \cC(u) \leq R \}} \cR_{N_1,...,N_6}^{(W)}}_{\L^p(d\mu_s)} & \leq \norm{\1_{ B_R^{H^1}}(u) \cdot \cR_{N_1,...,N_6}^{(W)}}_{\L^p(d\mu_s)} \\
            & \leq \norm{\1_{ B_R^{H^1} } (P_{N_{(3)/100}}u) \cdot \cR_{N_1,...,N_6}^{(W)}}_{\L^p(d\mu_s)}
        \end{split}
    \end{equation*}
    And, $\1_{ B_R^{H^1} } (P_{N_{(3)/100}}u)$ is $\cB_{\ll N_{(3)}}$ mesurable. So, using the $\L^p$-norm conditioned to the $\sigma$-algebra $\cB_{\ll N_{(3)}}$, denoted $\L^p(d\mu_s|\cB_{\ll N_{(3)}})$, we obtain :
    \begin{equation}\label{R L^p less than L^p_cond Linfty}
        \begin{split}
            \norm{\1_{\{ \cC(u) \leq R \}} \cR_{N_1,...,N_6}^{(W)}}_{\L^p(d\mu_s)} & \leq \Big\| \big\| \1_{ B_R^{H^1}}(P_{N_{(3)/100}}u) \cdot \cR_{N_1,...,N_6}^{(W)} \big\|_{\L^p(d\mu_s|\cB_{\ll N_{(3)}})}\Big\|_{\L^\infty(d\mu_s)} \\
            & \leq  \Big\| \big\| \cR_{N_1,...,N_6}^{(W)}\big\|_{\L^p(d\mu_s|\cB_{\ll N_{(3)}})} \cdot \1_{ B_R^{H^1}}(P_{N_{(3)/100}}u)\Big\|_{\L^\infty(d\mu_s)}
        \end{split}
    \end{equation}
    Now, the conditional Wiener-chaos estimate from Lemma~\ref{lem Wiener chaos} (with $m=3$), followed by  Lemma~\ref{L2 estimate with orthogonality} combined with the absence of paring (see Remark \ref{rem no high freq pairing, first generation}), allows us to obtain : 
    \begin{equation}\label{R application Wiener chaos}
        \begin{split}
            \norm{ \cR_{N_1,...,N_6}^{(W)}}_{\L^p(d\mu_s|\cB_{\ll N_{(3)}})} & \lesssim p^{\frac{3}{2}} \norm{ \cR_{N_1,...,N_6}^{(W)}}_{\L^2(d\mu_s|\cB_{\ll N_{(3)}})} \\
            & \lesssim p^{\frac{3}{2}} (N_{(1)} N_{(2)} N_{(3)})^{-s} \bigg( \sum_{k_1,k_2,k_3} \Big| \sum_{k_4,k_5,k_6} C(\vec{k}) \cdot \frac{\psi_{2s}(\vec{k})}{\Omega(\vec{k})}\cjg{w_{k_4}^{N_4}}w_{k_5}^{N_5}\cjg{w_{k_6}^{N_6}}  \Big|^2 \bigg)^{\frac{1}{2}}
        \end{split}
    \end{equation}
    where we gathered all the constraints into the term :
    \begin{equation*}
        C(\vec{k}) := \1_{\linc} \cdot \1_{\Omgnz} \cdot \prod_{j=1}^6 \1_{|k_j| \sim N_j}
    \end{equation*}
    Next, by Cauchy-Schwarz,
    \begin{equation*}
         \sum_{k_1,k_2,k_3} \Big| \sum_{k_4,k_5,k_6} C(\vec{k}) \cdot \frac{\psi_{2s}(\vec{k})}{\Omega(\vec{k})}\cjg{w_{k_4}^{N_4}}w_{k_5}^{N_5}\cjg{w_{k_6}^{N_6}}  \Big|^2 \lesssim  \sum_{k_1,k_2,k_3} \Big( \sum_{k_4,k_5,k_6} C(\vec{k}) \cdot \Big| \frac{\psi_{2s}(\vec{k})}{\Omega(\vec{k})} \Big|^2 |w_{k_4}^{N_4}w_{k_5}^{N_5}|^2 \Big) \Big( \sum_{k_4,k_5,k_6} C(\vec{k}) |w_{k_6}^{N_6}|^2 \Big) 
    \end{equation*}
    From Lemma~\ref{lemma psi estimate} we get $\left| \frac{\psi_{2s}(\vec{k})}{\Omega(\vec{k})} \right| \lesssim |k_{(1)}|^{2s-2}(1+\frac{|k_{(3)}|^2}{|\Omega(\vec{k})|}) \lesssim |k_{(1)}|^{2s-2} |k_{(3)}|^2 \lesssim N_{(1)}^{2s-2} N_{(3)}^2$. Then, on the one hand :
    \begin{equation*}
        \begin{split}
            \sum_{k_4,k_5,k_6} C(\vec{k}) \cdot \Big| \frac{\psi_{2s}(\vec{k})}{\Omega(\vec{k})} \Big|^2 |w_{k_4}^{N_4}w_{k_5}^{N_5}|^2 & \lesssim (N_{(1)}^{2s-2} N_{(3)}^2)^2  \sum_{k_4,k_5,k_6} C(\vec{k}) \cdot |w_{k_4}^{N_4}w_{k_5}^{N_5}|^2 \\
            & \lesssim (N_{(1)}^{2s-2} N_{(3)}^2)^2 \sum_{k_4,k_5} |w_{k_4}^{N_4}w_{k_5}^{N_5}|^2 \cdot \underbrace{\sum_{k_6} \1_{k_6=k_1-k_2...+k_5}}_{\leq 1} \\
            & \lesssim (N_{(1)}^{2s-2} N_{(3)}^2 \norm{w^{N_4}}_{\L^2}\norm{w^{N_5}}_{\L^2})^2
        \end{split}
    \end{equation*}
    and on the other hand,
    \begin{equation*}
    \begin{split}
            \sum_{k_1,k_2,k_3} \Big( \sum_{k_4,k_5,k_6} C(\vec{k}) |w_{k_6}^{N_6}|^2 \Big) = \sum_{k_6} |w_{k_6}^{N_6}|^2 \sum_{k_1,k_2,k_3,k_4,k_5} C(\vec{k})
            \lesssim N_{(2)}N_{(3)}N_{(4)}N_{(5)} \cdot \norm{w^{N_6}}_{\L^2}^2
    \end{split}
    \end{equation*}
    where we used the counting bound from Lemma~\ref{lem counting bound}. We deduce from the two inequalities above that
    \begin{equation*}
    \begin{split}
        \bigg( \sum_{k_1,k_2,k_3} \Big| \sum_{k_4,k_5,k_6} C(\vec{k}) \cdot \frac{\psi_{2s}(\vec{k})}{\Omega(\vec{k})}\cjg{w_{k_4}^{N_4}}w_{k_5}^{N_5}\cjg{w_{k_6}^{N_6}}  \Big|^2 \bigg)^{\frac{1}{2}} & \lesssim N_{(1)}^{2s-2} N_{(3)}^2 N_{(1)}^{\frac{1}{2}} N_{(3)}^{\frac{1}{2}}  (N_{(4)}N_{(5})^{\frac{1}{2}}\norm{w^{N_4}}_{\L^2}\norm{w^{N_5}}_{\L^2}\norm{w^{N_6}}_{\L^2} \\
        & \lesssim N_{(1)}^{2s-2+\frac{1}{2}} N_{(3)}^{\frac{5}{2}} (N_{(4)}N_{(5})^{-\frac{1}{2}}N_{(6)}^{-1}\norm{w^{N_4}}_{H^1}\norm{w^{N_5}}_{H^1}\norm{w^{N_6}}_{H^1} \\
        & \lesssim N_{(1)}^{2s-\frac{3}{2}} N_{(3)}^{\frac{5}{2}} \norm{w^{N_4}}_{H^1}\norm{w^{N_5}}_{H^1}\norm{w^{N_6}}_{H^1} 
    \end{split}
    \end{equation*}
    Coming back to \eqref{R application Wiener chaos}, we deduce that 
    \begin{equation*}
         \norm{ \cR_{N_1,...,N_6}^{(W)}}_{\L^p(d\mu_s|\cB_{\ll N_{(3)}})} \lesssim p^{\frac{3}{2}} N_{(1)}^{-\frac{3}{2}} N_{(3)}^{\frac{5}{2}-s} \norm{w^{N_4}}_{H^1}\norm{w^{N_5}}_{H^1}\norm{w^{N_6}}_{H^1} 
    \end{equation*}
    And plugging this into \eqref{R L^p less than L^p_cond Linfty}, we obtain:
    \begin{equation*}
        \norm{\1_{\{ \cC(u) \leq R \}} \cR_{N_1,...,N_6}^{(W)}}_{\L^p(d\mu_s)} \lesssim_R p^{\frac{3}{2}}  N_{(1)}^{-\frac{3}{2}} N_{(3)}^{\frac{5}{2}-s}
    \end{equation*}
    Interpolating the above inequality with \eqref{cR deterministic estimate} (more precisely \eqref{cR deterministic estimate} with $\frkR^{(W)}$ instead of $\frkR$), we have that for any $\theta \in (0,1)$ : 
    \begin{equation}\label{R interpolation}
        \norm{\1_{\{ \cC(u) \leq R \}} \cR_{N_1,...,N_6}^{(W)}}_{\L^p(d\mu_s)} \lesssim p^{\frac{3}{2}\theta +\frac{2-\alpha}{2}(1-\theta)}N_{(1)}^{-\frac{3}{2}\theta + \left(2(s-1-\sigma) + \eps + \alpha(s-\frac{3}{2})\right)(1-\theta)}N_{(3)}^{(\frac{5}{2}-s)\theta + (1-\theta)}
    \end{equation}
Finally, we use the following lemma to finish the proof of Lemma~\ref{lem R_s,N L^p estimate}:

    \begin{lem}\label{lem theta exponent < 1 on p and negative exponent on N_{(1)}}
        There exist $\sigma<s-\frac{1}{2}$ close enough to $s-\frac{1}{2}$, $\eps>0$ close enough to $0$, $\theta \in (0,1)$ and $\alpha \in (0,1)$, such that 
        \begin{align*}
                \frac{3}{2}\theta +\frac{2-\alpha}{2}(1-\theta) &< 1, & 
                N_{(1)}^{-\frac{3}{2}\theta + \left(2(s-1-\sigma) + \eps + \alpha(s-\frac{3}{2})\right)(1-\theta)}N_{(3)}^{(\frac{5}{2}-s)\theta + (1-\theta)} &\lesssim N_{(1)}^{-}
        \end{align*}
    \end{lem}
Let us provide a proof of this lemma:

    \begin{proof}[Proof of Lemma~\ref{lem theta exponent < 1 on p and negative exponent on N_{(1)}}]
Firstly, we have:
\begin{equation}\label{cd theta power of p}
    \begin{split}
    \frac{3}{2}\theta +\frac{2-\alpha}{2}(1-\theta) < 1 & \iff  \theta \left(\frac{1+\alpha}{2}\right) < 1 + \frac{\alpha -2 }{2} = \frac{\alpha}{2} \\
    & \iff \theta < \frac{\alpha}{1+\alpha}
    \end{split}
\end{equation}

Secondly, regarding the exponent of $N_{(3)}$, we have:
\begin{equation}\label{cd theta positive power of N3}
    (\frac{5}{2}-s)\theta + (1-\theta) \geq 0 \iff \theta \leq \frac{1}{s-\frac{3}{2}}
\end{equation}

Let us take $\theta \leq \frac{1}{s-\frac{3}{2}}$ (always true when $s \leq \frac{5}{2}$). Then,
\begin{equation*}
    N_{(1)}^{-\frac{3}{2}\theta + \left(2(s-1-\sigma) + \eps + \alpha(s-\frac{3}{2})\right)(1-\theta)}N_{(3)}^{(\frac{5}{2}-s)\theta + (1-\theta)} \lesssim N_{(1)}^{-\frac{3}{2}\theta + \left(2(s-1-\sigma) + \eps + \alpha(s-\frac{3}{2})\right)(1-\theta)}N_{(1)}^{(\frac{5}{2}-s)\theta + (1-\theta)}
\end{equation*}
We want to have appropriate parameters such that the conditions on $\theta$ given in \eqref{cd theta power of p} and \eqref{cd theta positive power of N3} are satisfied along with the following one: 
\begin{equation*}
    -\frac{3}{2}\theta + \left(2(s-1-\sigma) + \eps + \alpha(s-\frac{3}{2})\right)(1-\theta) + (\frac{5}{2}-s)\theta + (1-\theta) < 0 
\end{equation*}
Since $2(s-1-\sigma) + \eps \lra -1$ as $\sigma \ra s-\frac{1}{2}$ and $\eps \ra 0$, the previous condition will be true for $\sigma$ and $\eps$ respectively close enough to $s-\frac{3}{2}$ and $0$ if:
\begin{equation*}
    -\frac{3}{2}\theta + \left(-1 + \alpha(s-\frac{3}{2})\right)(1-\theta) + (\frac{5}{2}-s)\theta + (1-\theta) < 0 
\end{equation*}
And this condition is equivalent to the following one:
\begin{equation}\label{cd power of N1}
\begin{split}
    & \theta \left( 1-s - \alpha(s-\frac{3}{2}) \right) < - \alpha(s-\frac{3}{2}) \\
   \iff &\theta > \frac{\alpha(s-\frac{3}{2})}{s-1 + \alpha(s-\frac{3}{2})} = \frac{\alpha}{\frac{s-1}{s-\frac{3}{2}} + \alpha} =  \frac{\alpha}{1 + \frac{1}{2(s-\frac{3}{2})} + \alpha}
\end{split}
\end{equation}

To conclude, if we first take $\alpha$ small enough such that $\frac{\alpha}{1+\alpha}<\frac{1}{s-\frac{3}{2}}$, and then $\theta$ such that
\begin{equation*}
    \frac{\alpha}{1 + \frac{1}{2(s-3/2)} + \alpha}<\theta<\frac{\alpha}{1+\alpha},
\end{equation*}
the conditions given in \eqref{cd theta power of p}, \eqref{cd theta positive power of N3} and \eqref{cd power of N1} are satisfied. We illustrate this with the following drawing :

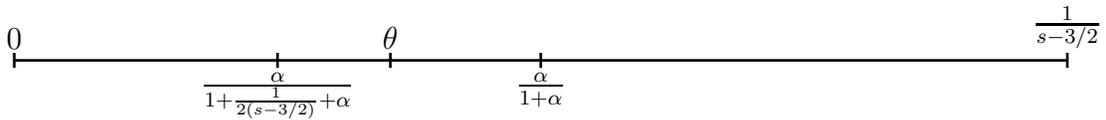
\begin{center}
\begin{tikzpicture}
    \draw[line width=1.pt] (0,0) -- (14,0);
    \draw[line width=1.pt] (0, -0.1) -- (0, 0) node[above] {$0$} -- (0, 0.1);
    \draw[line width=1.pt] (14, -0.1) -- (14, 0) node[above] {$\frac{1}{s-3/2}$} -- (14, 0.1);
    \draw[line width=1.pt] (3.5, -0.1) -- (3.5, 0) node[below] {$\frac{\alpha}{1 + \frac{1}{2(s-3/2)} + \alpha}$} -- (3.5, 0.1);
    \draw[line width=1.pt] (7, -0.1) -- (7, 0) node[below] {$\frac{\alpha}{1 + \alpha}$} -- (7, 0.1);
    \draw[line width=1.pt] (5, -0.1) -- (5, 0) node[above] {$\theta$} -- (5, 0.1);
\end{tikzpicture}
\end{center}
With those parameters, Lemma~\ref{lem theta exponent < 1 on p and negative exponent on N_{(1)}} is proven.
    \end{proof}
    Hence, coming back to \eqref{R interpolation}, the proof of Lemma~\ref{lem R_s,N L^p estimate} is completed.
\end{proof}

\section{Estimates for the modified energy derivative at 0}\label{section Estimates for the differential of the modified energy}
This section is dedicated to $L^p(d\mu_s)$ estimates on $Q_{s,N}$ (defined in \eqref{def Qs,N = Q0 + Q1 +Q2}, see also Definition~\ref{def normal form Q}). More precisely, we prove Proposition~\ref{prop Qs,N L^p estimate wrt mu_s}. The strategy of our proof is the same as the one of Lemma~\ref{lem R_s,N L^p estimate}. We first obtain a deterministic estimate that will be conclusive except for a frequency regime  \textit{high-high-high-low-...-low} (where the three highest frequencies are in fact much more higher than the others). This will lead us to decompose $Q_{s,N}$ into two parts, one that captures the \textit{high-high-high-low-...-low} regime, and one that captures the other regime of frequencies. We will handle the \textit{high-high-high-low-...-low} regime using the independence of Gaussians via Wiener-chaos estimate. It will be slightly more complicated than the proof of Lemma~\ref{lem R_s,N L^p estimate} because of the presence of more indices. However, in our situation, we will not encounter the possibility of a "pairing between generations"\footnote{For example, considering the constraint $\linc \hsp \& \hsp \lincba$, then according to Definition~\ref{def pairing}, a "pairing between generations" corresponds to the situation when one of the $k_j$ is paired with one of the $p_l$.} which could have required the "remarkable cancellation" that has been presented in \cite{sun2023quasi} (sections 5 and 7). The reason why we do not encounter such a pairing stems from the fact that we perform the Wiener chaos estimate with respect to three high-frequency Gaussians, that is with $m=3$ in Lemma~\ref{lem Wiener chaos}. In doing so, we prevent a  pairing between generations from occurring (see Paragraph \ref{absence of pairing}). In \cite{sun2023quasi}, the Wiener chaos estimates are performed with respect to two high-frequency Gaussians, that is with $m=2$ in Lemma \ref{lem Wiener chaos}, and in this situation, a pairing between generation may occur. \bigskip

\paragraph{\textbf{Preliminaries}} Fix $N \in \N \cup \{ \infty \}$.
Recall that from \eqref{def Qs,N = Q0 + Q1 +Q2} we have:
    \begin{equation*}
        Q_{s,N}(u) = Q_{s,N}(w) = \Im(-\frac{1}{6}\cQ_0(w) + \frac{1}{2}\cQ_1(w) - \frac{1}{2}\cQ_2(w))
    \end{equation*}
    where $w:= \Pi_N u$ (with the convention $\Pi_\infty =id$), and:
    \begin{equation*}
    \cQ_0(w) := \sum_{\substack{\linc \\ \Omgz}} \psi_{2s}(\vec{k})w_{k_1}\cjg{w_{k_2}}...\cjg{w_{k_6}}
\end{equation*}
and, 
    \begin{equation*}
    \cQ_1(w) := \sum_{\substack{\linc \\ \lincba \\\Omgnz}} \frac{\psi_{2s}(\vec{k})}{ \Omega(\vec{k})} w_{p_1}\cjg{w_{p_2}}...w_{p_5}\cjg{w_{k_2}}...\cjg{w_{k_6}} 
\end{equation*}
and,
\begin{equation*}
    \cQ_2(w) :=\sum_{\substack{\linc \\ \lincbb \\\Omgnz}} \frac{\psi_{2s}(\vec{k})}{ \Omega(\vec{k})} w_{k_1}\cjg{w_{q_1}}w_{q_2}...\cjg{w_{q_5}}w_{k_3}... \cjg{w_{k_6}}
\end{equation*}
It suffices to show the estimates for $\cQ_0,\cQ_1$ and $\cQ_2$ because $|Q_{s,N}(u)| \leq |\cQ_0(w)| + |\cQ_1(w)| + |\cQ_2(w)|$. \\

\underline{Estimate for $\cQ_0$ :}
Actually, the estimate : 
\begin{equation}\label{Q0 estimate}
    \norm{\1_{\{\cC(u) \leq R\}} \cQ_0}_{\L^p(d\mu_s)} \leq C(s,R)p^\beta
\end{equation}
has somehow already been proven in the proof of Lemma~\ref{lem R_s,N L^p estimate}. Indeed, the proof is very similar, and we only sketch the beginning of it.

\begin{proof}[Sketch of the proof of \eqref{Q0 estimate}]
    We decompose $\cQ_0$ dyadically and we use Lemma~\ref{lemma psi estimate} (with $\Omgz$), and we obtain :
    \begin{equation*}
            \left| \cQ_0(w) \right| \lesssim \sum_{N_1,...,N_6} \frkQ_0 (N_1,...,N_6)
    \end{equation*}
    where,
    \begin{equation*}
        \frkQ_0(N_1,...,N_6) := \sum_{\substack{\linc \\ \Omgz}} N_{(1)}^{2s-2} N_{(3)}^2 \prod_{j=1}^6 \1_{|k_j| \sim N_j} |w_{k_j}|
    \end{equation*}     
    To estimate $\frkQ_0(N_1,...,N_6)$, we use Lemma~\ref{Strichartz's trick} (with $f^{j}_{k_j} = \1_{|k_j| \sim N_j} |w_{k_j}|$), which yields:
    \begin{equation*}
         \frkQ_0(N_1,...,N_6) \lesssim_{\eps} N_{(1)}^{2s-2+\eps} N_{(3)}^2 \prod_{j=1}^{6} \norm{P_{N_j}w}_{\L^2(\T)}
    \end{equation*}
with $P_{N}$ the projector onto frequencies $|k| \sim N$. This estimate is the analogue of \eqref{bound R_N1...N6}. And from this point, the proof of \eqref{Q0 estimate} goes exactly the same as the proof of Lemma~\ref{lem R_s,N L^p estimate}. 
\end{proof}
To sum up, in order to prove Proposition~\ref{prop Qs,N L^p estimate wrt mu_s}, it remains to establish the estimates in the following lemma :
\begin{lem}\label{lem Q1,Q2 estimate}
Let $s > \frac{3}{2}$. Then, there exists $\beta \in (0,1)$ such that for every $R>0$, there exists a constant $C(s,R) > 0$, such that for any $p \in [2,+\infty)$,
\begin{equation}
    \norm{\1_{\{\cC(u) \leq R\}} \cQ_j}_{\L^p(d\mu_s)} \leq C(s,R)p^\beta
\end{equation}
for $j=1,2$.
\end{lem}

Since the analysis for $\cQ_1$ and $\cQ_2$ is the same, we only prove the estimate for $\cQ_1$.

\subsubsection{Notations and remarks on set of indices:}
Before decomposing $\cQ$, let us first introduce some notations :
\begin{notns}\label{notations Q estimate}
    Given a set of frequency $k_1,k_2,...,k_6,p_1,p_2,...,p_5 \in \Z$, we denote :
    \begin{itemize}
        \item $k_{(1)},...,k_{(6)}$ a rearrangement of the $k_j$ such that 
    \begin{equation*}
        |k_{(1)}| \geq |k_{(2)}| \geq ... \geq |k_{(6)}| 
    \end{equation*}
    \item $p_{(1)},...,p_{(6)}$ a rearrangement of the $p_j$ such that 
    \begin{equation*}
        |p_{(1)}| \geq |p_{(2)}| \geq ... \geq |p_{(5)}| 
    \end{equation*}
     \item $n_{(1)},n_{(2)},...,n_{(10)}$ a rearrangement of $k_2,...,k_6,p_1,...,p_5$ such that 
    \begin{equation*}
        |n_{(1)}| \geq |n_{(2)}| \geq ... \geq |n_{(10)}| 
    \end{equation*} 
    \end{itemize}
    Also, in the sequel we will use :
    \begin{itemize}
        \item the letter $M_j$ for the localization of the frequency $k_j$,
        \item the letter $P_j$ for the localization of the frequency $p_j$,
    \end{itemize}
    Finally, $N_{(1)} \geq N_{(2)} \geq ... \geq N_{(10)}$ will be a non-increasing rearrangement of $P_1,...,P_5,M_2,...,M_6$. 
\end{notns}

\begin{rem}
  -- when $k_1-k_2+...-k_6=0$, we have $ |k_{(1)}| \sim |k_{(2)}|$.\\
  -- when $k_1=p_1-p_2+...+p_5$, we have $|p_{(1)}| \sim |p_{(2)}|$ or $|p_{(1)}| \sim |k_1|$.\\
  -- when $k_1-k_2+...-k_6=0$ and  $k_1=p_1-p_2+...+p_5$, we have $|n_{(1)}| \sim |n_{(2)}|$.
\end{rem}

\subsection{Decomposition}
For convenience, we will denote $\cQ_1$ simply as $\cQ$. \\

    Next, in the same spirit as the decomposition in Section~\ref{section Estimates for the weight of the weighted Gaussian measures}, we decompose the set of indices over which we sum in $\cQ$. We invoke the following set of indices : 
    \begin{equation}
        \begin{split}\label{I_D}
             \cI_D := \{ (k_1,...,k_6,p_1,...,p_5)& \in \Z^{11} : \hsp \sum_{j=1}^6 (-1)^{j-1}k_j = 0, \hspace{0.2cm} k_1 = \sum_{j=1}^5 (-1)^{j-1}p_j, \hspace{0.2cm} \sum_{j=1}^6 (-1)^{j-1}k_j^2 \neq 0, \\
    & (n_{(1)},n_{(2)}) \in \{ p_1,...,p_5 \}^2 \hspace{0.2cm} \textbf{or} \hspace{0.2cm} |n_{(3)}| < |n_{(1)}|^{1-\delta_0} \hsp \textbf{or} \hspace{0.2cm} |n_{(4)}| \geq |n_{(3)}|^{\delta_0}  \}
        \end{split}
    \end{equation}
    and,
     \begin{equation}\label{I_W}
        \begin{split}
             \cI_W := \{ (k_1,...,k_6,p_1,...&,p_5) \in \Z^{11} : \hsp \sum_{j=1}^6 (-1)^{j-1}k_j = 0, \hspace{0.2cm} k_1 = \sum_{j=1}^5 (-1)^{j-1}p_j, \hspace{0.2cm} \sum_{j=1}^6 (-1)^{j-1}k_j^2 \neq 0, \\
    & (n_{(1)},n_{(2)}) \notin \{ p_1,...,p_5 \}^2 \hspace{0.2cm} \textbf{and} \hspace{0.2cm} |n_{(3)}| \geq |n_{(1)}|^{1-\delta_0} \hsp \textbf{and} \hspace{0.2cm} |n_{(4)}| < |n_{(3)}|^{\delta_0}  \}
        \end{split}
    \end{equation}
for a fixed $\delta_0 \in (0,1)$\footnote{Once again, $\delta_0$ can be any number in $(0,1)$}. Note that what differs from the decomposition in Section~\ref{section Estimates for the weight of the weighted Gaussian measures} is the additional constraint $(n_{(1)},n_{(2)}) \notin \{ p_1,...,p_5 \}^2$ in $\cI_W$. This is because if $(n_{(1)},n_{(2)}) \in \{ p_1,...,p_5 \}^2$, we will not need to use Wiener-chaos. \\
In the same vein, we split $\cQ$ as : 
\begin{equation*}
    \cQ(w) = \cQ^{(D)}(w) + \cQ^{(W)}(w)
\end{equation*}
where, 
\begin{align}
    \cQ^{(D)}(w) &:= \sum_{\cI_D} \frac{\psi_{2s}(\vec{k})}{ \Omega(\vec{k})} w_{p_1}\cjg{w_{p_2}}...w_{p_5}\cjg{w_{k_2}}...\cjg{w_{k_6}}, & \cQ^{(W)}(w) := \sum_{\cI_W}\frac{\psi_{2s}(\vec{k})}{ \Omega(\vec{k})} w_{p_1}\cjg{w_{p_2}}...w_{p_5}\cjg{w_{k_2}}...\cjg{w_{k_6}}
\end{align}

\subsubsection{Absence of pairing}\label{absence of pairing}
Recall that we have defined what a "pairing" is in Definition \ref{def pairing}. \\
In $\cI_W$ (see \eqref{I_W} above), there is (for $n_{(1)}$ large enough) no pairing within the three highest frequencies, relatively to the constraint: 
\begin{center}
    $p_1 - p_2 ... +p_5 - k_2 + k_3 ... -k_6 = 0$
\end{center}
 Indeed, suppose there is a pairing between two of the three highest frequencies. Then, the constraint would take the form:
    \begin{center}
        the remaining high frequency = sum of seven low frequencies
    \end{center}
    which is impossible because in $\cI_W$ we have:
    \begin{center}
        $|$high frequencies$|$ $\gg$ $|$low frequencies$|$
    \end{center}

\subsection{Proof of the estimate}

We are now prepared to prove Lemma~\ref{lem Q1,Q2 estimate} (for $j=1$). The forthcoming proof is organized as follows. Firstly, we establish an estimate that will be conclusive only for the contribution $\cQ^{(D)}$. For the contribution $\cQ^{(W)}$, we will also exploit the independence between high frequency Gaussians and low frequency Gaussians using Wiener chaos estimate. 

\begin{proof}[Proof of Lemma~\ref{lem Q1,Q2 estimate}] 
\underline{Step 1, Deterministic estimate :}
    Let $w \in \Hgmt$. For convenience, for any dyadic number $N$ we denote $w_k^N := \1_{|k|\sim N}|w_k|$.  We start by taking the absolute value and summing over dyadic blocks
    \begin{equation*}
        \left| \cQ(w) \right| \lesssim \sum_{M_1,...,M_6,P_1,...,P_5} \frkQ(M_1,...,M_6,P_1,...,P_5)
    \end{equation*}
    where $M_1,...,M_6,P_1,...,P_5$ are dyadic-valued and 
    \begin{equation*}
        \frkQ(M_1,...,M_6,P_1,...,P_5) := \sum_{\substack{\linc \\ \lincba \\\Omgnz}} \big| \frac{\psi_{2s}(\vec{k})}{ \Omega(\vec{k})} \big| |w_{p_1}^{P_1}...w_{p_5}^{P_5}w_{k_2}^{M_2}...w_{k_6}^{M_6}| \cdot \1_{|k_1| \sim M_1}
    \end{equation*}  
        We decompose $\frkQ$ as $\frkQ = \frkQ^{(D)} + \frkQ^{(W)}$, according to the decomposition of $\cQ$.\\
        
    In regard to \eqref{T_1 dyadic block} and \eqref{dyadic estimate T_1} from the proof of Proposition~\ref{prop continuity of Qs,N and deterministic estimate}, we have:
    \begin{equation*}
             \frkQ(M_1,...,P_5) \lesssim
              M_{(1)}^{2s-2+\eps} M_{(3)}^{2}  \left( P_{(2)} P_{(3)} P_{(4)} P_{(5)}\right)^{\frac{1}{2}}  \prod_{j=1}^5 \norm{w^{P_j}}_{\L^2} \prod_{j=2}^6 \norm{w^{M_j}}_{\L^2} 
    \end{equation*}
    Using Notations \ref{notations Q estimate}, it means that 
        \begin{equation}\label{first bound Q_M1...M6,P1...P5}
             \frkQ(M_1,...,P_5) \lesssim
              M_{(1)}^{2s-2+\eps} M_{(3)}^{2}  \left( P_{(2)} P_{(3)} P_{(4)} P_{(5)}\right)^{\frac{1}{2}}  \prod_{j=1}^{10} \norm{w^{N_j}}_{\L^2}
    \end{equation}
     Moreover, for $1<\sigma'<\sigma<s-\frac{1}{2}$, we have : 
    \begin{equation}\label{w^N(j) L2}
        \prod_{j=1}^{10} \norm{w^{N_j}}_{\L^2} \lesssim (N_{(1)}^{-\sigma}\hsignorm{u})(N_{(2)}^{-\sigma'} \norm{u}_{H^{\sigma'}})(N_{(3)}N_{(4)}...N_{(10)})^{-1} \norm{u}_{H^1}^8
    \end{equation}
    Then, since $N_{(1)} \sim N_{(2)}$, and writing $\sigma'$ as: 
    \begin{equation*}
        \sigma' = \alpha 1 + (1-\alpha) \sigma
    \end{equation*}
    for $\alpha \in (0,1)$, we deduce from \eqref{w^N(j) L2}, by interpolating $H^{\sigma'}$ between $H^1$ and $\hsig$, that: 
    \begin{equation*}
        \begin{split}
            \prod_{j=1}^{10} \norm{w^{N_j}}_{\L^2} & \lesssim N_{(1)}^{-2\sigma +\alpha(\sigma-1)}(N_{(3)}N_{(4)}...N_{(10)})^{-1} \norm{u}_{H^1}^{8+\alpha} \norm{u}^{2-\alpha}_{\hsig} \\
            & \lesssim N_{(1)}^{-2\sigma +\alpha(s-\frac{3}{2})}(N_{(3)}N_{(4)}...N_{(10)})^{-1} \norm{u}_{H^1}^{8+\alpha} \norm{u}^{2-\alpha}_{\hsig}, \hspace{0.3cm} (\textnormal{because $\alpha(\sigma-1) < \alpha(s-3/2)$ })
        \end{split}
    \end{equation*}
    
        Plugging this into \eqref{first bound Q_M1...M6,P1...P5}, we obtain 
    \begin{equation*}
        \frkQ(M_1,...,P_5) \lesssim N_{(1)}^{-2\sigma +\alpha(s-\frac{3}{2})} M_{(1)}^{2s-2+\eps} M_{(3)}^{2}  \left( P_{(2)} P_{(3)} P_{(4)} P_{(5)}\right)^{\frac{1}{2}}(N_{(3)}N_{(4)}...N_{(10)})^{-1} \norm{u}_{H^1}^{8+\alpha} \norm{u}^{2-\alpha}_{\hsig}
    \end{equation*} 
    Using the facts $M_{(1)}\lesssim N_{(1)}$ and $N_{(1)} \sim N_{(2)}$, and also Remark \ref{rem cut-off H1}, we deduce that 
    \begin{equation}\label{Q pre-deterministic estimate}
        \begin{split}
       \1_{\{\cC(u) \leq R\}} \frkQ&(M_1,...,P_5) \\
       & \lesssim_{\eps,R} N_{(1)}^{2(s-\sigma-1) + \eps +\alpha(s-\frac{3}{2})} M_{(3)}^{2}  \left( P_{(2)} P_{(3)} P_{(4)} P_{(5)}\right)^{\frac{1}{2}}N^2_{(1)}(\overbrace{N_{(1)}N_{(2)}N_{(3)}N_{(4)}...N_{(10)}}^{=M_2...M_6 P_1...P_5})^{-1} \norm{u}^{2-\alpha}_{\hsig} \\
        & \lesssim N_{(1)}^{2(s-\sigma) + \eps +\alpha(s-\frac{3}{2})} M_{(3)}^{2} (M_2...M_6)^{-1} P_{(1)}^{-1}\left( P_{(2)} P_{(3)} P_{(4)} P_{(5)}\right)^{-\frac{1}{2}} \norm{u}^{2-\alpha}_{\hsig} \\
        & \lesssim \begin{cases}
            N_{(1)}^{2(s-\sigma) - \frac{3}{2} + \eps +\alpha(s-\frac{3}{2})} \norm{u}^{2-\alpha}_{\hsig}, \hspace{1.9cm} \textnormal{if $(N_{(1)},N_{(2)}) \in \{ P_1,...,P_5\}^2$} \\
            N_{(1)}^{2(s-\sigma-1) + \eps +\alpha(s-\frac{3}{2})} N_{(3)} N_{(4)}^{-\frac{1}{2}} \norm{u}^{2-\alpha}_{\hsig}, \hspace{0.4cm} \textnormal{if $(N_{(1)},N_{(2)}) \notin \{ P_1,...,P_5\}^2$}
        \end{cases}
        \end{split}
    \end{equation} 

Here, we used estimates on $ M_{(3)}^{2} (M_2...M_6)^{-1} P_{(1)}^{-1}\left( P_{(2)} P_{(3)} P_{(4)} P_{(5)}\right)^{-\frac{1}{2}}$ in terms of the $N_j$, which are gathered in the following lemma. 
\begin{lem}\label{lem M3^2 Mj^-1 P1^-1 Pj^-1/2}
\begin{itemize}
    \item[--]If  $(N_{(1)},N_{(2)}) \in \{ P_1,...,P_5\}^2$, then 
    \begin{equation*}
        M_{(3)}^{2} (M_2...M_6)^{-1} P_{(1)}^{-1}\left( P_{(2)} P_{(3)} P_{(4)} P_{(5)}\right)^{-\frac{1}{2}} \lesssim N_{(1)}^{-\frac{3}{2}}
    \end{equation*}
    \item[--] If $(N_{(1)},N_{(2)}) \notin \{ P_1,...,P_5\}^2$, then 
    \begin{equation*}
        M_{(3)}^{2} (M_2...M_6)^{-1} P_{(1)}^{-1}\left( P_{(2)} P_{(3)} P_{(4)} P_{(5)}\right)^{-\frac{1}{2}} \lesssim N_{(1)}^{-2}N_{(3)}N_{(4)}^{-\frac{1}{2}}
    \end{equation*}
\end{itemize}
\end{lem}
For clarity, we postpone the proof of Lemma~\ref{lem M3^2 Mj^-1 P1^-1 Pj^-1/2} for the end of this section. \\
To conclude the first step of the proof, using \eqref{Gaussian moments} in \eqref{Q pre-deterministic estimate} yields:

    \begin{equation}\label{Q deterministic estimate}
       \norm{\1_{\{\cC(u) \leq R\}} \frkQ(M_1,...,P_5)}_{\L^p(d\mu_s)} \lesssim \begin{cases}
            N_{(1)}^{2(s-\sigma) - \frac{3}{2} + \eps +\alpha(s-\frac{3}{2})} p^{\frac{2-\alpha}{2}}, \hspace{1.9cm} \textnormal{if  $(N_{(1)},N_{(2)}) \in \{ P_1,...,P_5\}^2$} \\
            N_{(1)}^{2(s-\sigma-1) + \eps +\alpha(s-\frac{3}{2})} N_{(3)} N_{(4)}^{-\frac{1}{2}} p^{\frac{2-\alpha}{2}}, \hspace{0.4cm} \textnormal{if $(N_{(1)},N_{(2)}) \notin \{ P_1,...,P_5\}^2$}
        \end{cases}
    \end{equation}
where the constant depends on $s,R,\sigma$ and $\eps$. This is the deterministic estimate we will work with later in the proof. We stress the fact that \eqref{Q deterministic estimate} is true for $\frkQ$ and also for $\frkQ^{(D)}$ and $\frkQ^{(W)}$ (with the exact same proof).\\ 

\underline{Step 2, Estimates for $\cQ^{(D)}$ and $\cQ^{(W)}$ :} \\

\textbullet Let us begin with the estimate for $\cQ^{(D)}$. Recall that in $\cQ^{(D)}$ we sum over the set $\cI_D$ defined in \eqref{I_D}. And, the constraints in $\cI_D$ imply that the dyadic integers (from the dyadic decomposition) must satisfy one of the following conditions :
\begin{align*}
    (N_{(1)},N_{(2)}) &\in \{P_1,...,P_5\}^2 &  &\textbf{or} & N_{(3)} &\lesssim N_{(1)}^{1-\delta_0} & &\textbf{or} & N_{(4)} &\gtrsim N_{(3)}^{\delta_0}
\end{align*}

In the first situation, when $(N_{(1)},N_{(2)}) \in \{P_1,...,P_5\}^2$, \eqref{Q deterministic estimate} 
(with $\frkQ^{(D)}$ instead of $\frkQ$) yields:
\begin{equation*}
    \norm{\1_{\{\cC(u) \leq R\}} \frkQ^{(D)}(M_1,...,P_5)}_{\L^p(d\mu_s)} \lesssim  N_{(1)}^{2(s-\sigma) - \frac{3}{2} + \eps +\alpha(s-\frac{3}{2})} p^{\frac{2-\alpha}{2}}
\end{equation*}

And, the exponent of $N_{(1)}$ satisfy :
\begin{equation*}
        \begin{cases}
            2(s-\sigma) - \frac{3}{2} +\eps+\alpha(s-\frac{3}{2}) > -\frac{1}{2} \\
            2(s-\sigma) - \frac{3}{2}+\eps+\alpha(s-\frac{3}{2}) \lra -\frac{1}{2} \hspace{0.2cm} \text{as} \begin{cases}
                \eps \ra 0 \\
                \alpha \ra 0 \\
                \sigma \ra s -\frac{1}{2}
            \end{cases}
        \end{cases}
\end{equation*}
Thus, for $\eps$ and $\alpha$ close enough to $0$ and $\sigma$ close enough to $s-\frac{1}{2}$, we have 
\begin{equation*}
    \norm{\1_{\{\cC(u) \leq R\}} \frkQ^{(D)}(M_1,...,P_5)}_{\L^p(d\mu_s)} \lesssim  N_{(1)}^- p^{\frac{2-\alpha}{2}}
\end{equation*}

Now, in the second situation when $(N_{(1},N_{(2)}) \notin \{P_1,...,P_5\}^2$ \textbf{and :} $N_{(3)} \lesssim N_{(1)}^{1-\delta_0}$ \textbf{or}  $N_{(4)} \gtrsim N_{(3)}^{\delta_0}$ , \eqref{Q deterministic estimate} yields :
\begin{equation*}
    \norm{\1_{\{\cC(u) \leq R\}} \frkQ^{(D)}(M_1,...,P_5)}_{\L^p(d\mu_s)} \lesssim  N_{(1)}^{2(s-\sigma -1) + \eps +\alpha(s-\frac{3}{2})} N_{(3)} N_{(4)}^{-\frac{1}{2}}p^{\frac{2-\alpha}{2}}
\end{equation*}

And since the exponent of $N_{(1)}$ satisfy :
 \begin{equation*}
        \begin{cases}
            2(s-\sigma-1)+\eps+\alpha(s-\frac{3}{2}) > -1 \\
            2(s-\sigma-1)+\eps+\alpha(s-\frac{3}{2}) \lra -1 \hspace{0.2cm} \text{as} \begin{cases}
                \eps \ra 0 \\
                \alpha \ra 0 \\
                \sigma \ra s -\frac{1}{2}
            \end{cases}
        \end{cases}
    \end{equation*}
we deduce that for a fixed $\delta_0 \in (0,1)$ and for $\eps$ and $\alpha$ close enough to $0$ and $\sigma$ close enough to $s-\frac{1}{2}$, we have again
\begin{equation*}
    \norm{\1_{\{\cC(u) \leq R\}} \frkQ^{(D)}(M_1,...,P_5)}_{\L^p(d\mu_s)} \lesssim  N_{(1)}^- p^{\frac{2-\alpha}{2}}
\end{equation*}

Finally, summing over the dyadic integers $M_1,...,M_6,P_1,...,P_5$, this estimate for both situations leads to
\begin{equation*}
    \norm{\1_{\{\cC(u) \leq R\}} \cQ^{(D)}}_{\L^p(d\mu_s)} \lesssim p^{\frac{2-\alpha}{2}}
\end{equation*}
which is the desired estimate. \\

\textbullet Next, in order to finish the proof of Lemma~\ref{lem Q1,Q2 estimate}, we need the same estimate for the contribution $\cQ^{(W)}$. For this term, while the estimate \eqref{Q deterministic estimate} is not conclusive, we are in a situation where we can make use of Wiener chaos. Recall that in $\cQ^{(W)}$, we sum over the set $\cI_W$ defined in \eqref{I_W}. Even though the following method is very similar to the one used for estimating $\cR^{(W)}$ in Section~\ref{section Estimates for the weight of the weighted Gaussian measures} (see the proof of Lemma~\ref{lem R_s,N L^p estimate}), we re-perform the analysis in detail. \\

As usual, we start by decomposing $\cQ^{(W)}$ as\footnote{For the analysis below, we need to keep the complex conjugation bars, so here we don't use $\frkQ^{(W)}$} $\cQ^{(W)} = \sum_{M_1,...,M_6,P1,...,P_5} \cQ^{(W)}_{M_1,...,M_6,P1,...,P_5}$, where :
\begin{equation*}
    \cQ^{(W)}_{M_1,...,M_6,P1,...,P_5} := \sum_{\cI_W}\frac{\psi_{2s}(\vec{k})}{ \Omega(\vec{k})} w_{p_1}\cjg{w_{p_2}}...w_{p_5}\cjg{w_{k_2}}...\cjg{w_{k_6}} \left(\prod_{j=1}^6 \1_{|k_j|\sim M_j} \right) \left(\prod_{j=1}^5 \1_{p_j \sim P_j}\right)
\end{equation*}
Henceforth, we denote $w_{k_j}^{M_j} := \1_{|k_j|\sim N_j} w_{k_j}$ and $w_{p_j}^{P_j} := \1_{|p_j|\sim P_j} w_{p_j}$ for better readability. \\

Without loss of generality, we only prove the estimate for the contribution in $\cQ^{(W)}$ where the three highest frequencies are $k_2,k_3$ and $k_4$, because the other cases are identical. \\
    Again we denote $\cB_{\ll N_{(3)}}$ the $\sigma$-algebra generated by Gaussians $\left( g_k \right)_{|k|\leq N_{(3)}/100}$. Since we only need to consider the contribution when $N_{(1)}$ is large, we can assume that $N_{(1)}$ is sufficiently large to ensure $N_{(4)}\leq N_{(3)}/100$. This follows from the constraints in $\cI_W$ which imply $N_{(3)} \gtrsim N_{(1)}^{1-\delta_0}$ and $N_{(4)} \lesssim N_{(3)}^{\delta_0}$. Thus, for $N_{(1)}$ large enough, we have $N_{(3)}^{\delta_0} \ll N_{(3)}$, ensuring that $N_{(4)}\leq N_{(3)}/100$. \\
 As a consequence, we have that the random variables :
    \begin{equation*}
        \begin{split}
        & w_{k_2}^{M_2}, w_{k_3}^{M_3}, w_{k_4}^{M_4} \hspace{0.2cm} \textnormal{are independent of $\cB_{\ll N_{(3)}}$}, \\
        & w_{k_5}^{M_5}, w_{k_6}^{M_6},w_{p_1}^{P_1},w_{p_2}^{P_2},w_{p_3}^{P_3},w_{p_4}^{P_4},w_{p_5}^{P_5} \hspace{0.2cm} \textnormal{are $\cB_{\ll N_{(3)}}$ mesurable}.
        \end{split}
    \end{equation*}
    Now, identically to \eqref{R L^p less than L^p_cond Linfty}, we have
    \begin{equation}\label{Q L^p less than L^p_cond Linfty}
            \norm{\1_{\{ \cC(u) \leq R \}} \cQ_{M_1,...,M_6,P_1,...,P_5}^{(W)}}_{\L^p(d\mu_s)} \leq  \Big\| \big\| \cQ_{M_1,...,M_6,P_1,...,P_5}^{(W)}\big\|_{\L^p(d\mu_s|\cB_{\ll N_{(3)}})} \cdot \1_{ B_R^{H^1}}(P_{N_{(3)/100}}u)\Big\|_{\L^\infty(d\mu_s)}
    \end{equation}
    
    And, the conditional Wiener-chaos estimate from Lemma~\ref{lem Wiener chaos} (with $m=3$), followed by  Lemma~\ref{L2 estimate with orthogonality} combined with the absence of paring (see Paragraph \ref{absence of pairing}), allows us to obtain : 
    \begin{equation}\label{Q application Wiener chaos}
        \begin{split}
            &\norm{ \cQ_{M_1,...,M_6,P_1,...,P_5}^{(W)}}_{\L^p(d\mu_s|\cB_{\ll N_{(3)}})}  \lesssim p^{\frac{3}{2}} \norm{ \cQ_{M_1,...,M_6,P_1,...,P_5}^{(W)}}_{\L^2(d\mu_s|\cB_{\ll N_{(3)}})} \\
            & \lesssim p^{\frac{3}{2}} (N_{(1)} N_{(2)} N_{(3)})^{-s} \bigg( \sum_{k_2,k_3,k_4} \Big| \sum_{\substack{k_1,k_5,k_6 \\ p_1,p_2,p_3,p_4,p_5}} C(\vec{k},\vec{p}) \cdot \frac{\psi_{2s}(\vec{k})}{\Omega(\vec{k})}w_{p_1}^{P_1}\cjg{w_{p_2}^{P_2}}w_{p_3}^{P_3}\cjg{w_{p_4}^{P_4}}w_{p_5}^{P_5}w_{k_5}^{M_5}\cjg{w_{k_6}^{M_6}}  \Big|^2 \bigg)^{\frac{1}{2}}
        \end{split}
    \end{equation}
    where we gathered all the constraints into the term :
    \begin{equation*}
        C(\vec{k},\vec{p}) := \1_{\linc} \cdot \1_{\lincba} \cdot \1_{\Omgnz} \cdot \prod_{j=1}^6 \1_{|k_j| \sim M_j} \prod_{j=1}^5 \1_{|p_j| \sim P_j}
    \end{equation*}
    Next, by Cauchy-Schwarz,
    \begin{equation*}
        \begin{split}
        &\sum_{k_2,k_3,k_4}  \bigg|\sum_{\substack{k_1,k_5,k_6 \\ p_1,p_2,p_3,p_4,p_5}} C(\vec{k},\vec{p}) \cdot\frac{\psi_{2s}(\vec{k})}{\Omega(\vec{k})}w_{p_1}^{P_1}...w_{p_5}^{P_5}w_{k_5}^{M_5}\cjg{w_{k_6}^{M_6}} \bigg|^2 \\
        & \lesssim  \sum_{k_2,k_3,k_4} \bigg( \sum_{\substack{k_1,k_5,k_6 \\ p_1,p_2,p_2,p_4,p_5}} C(\vec{k},\vec{p}) \cdot \Big| \frac{\psi_{2s}(\vec{k})}{\Omega(\vec{k})} \Big|^2 |w_{p_1}^{P_1}...w_{p_5}^{P_5}w_{k_5}^{M_5}|^2 \bigg) \bigg( \sum_{\substack{k_1,k_5,k_6 \\ p_1,p_2,p_2,p_4,p_5}} C(\vec{k},\vec{p}) |w_{k_6}^{M_6}|^2 \bigg) 
        \end{split}
    \end{equation*}
    From Lemma~\ref{lemma psi estimate} we get $\left| \frac{\psi_{2s}(\vec{k})}{\Omega(\vec{k})} \right| \lesssim |k_{(1)}|^{2s-2}(1+\frac{|k_{(3)}|^2}{|\Omega(\vec{k})|}) \lesssim |k_{(1)}|^{2s-2} |k_{(3)}|^2 \lesssim N_{(1)}^{2s-2} N_{(3)}^2$. Then, on the one hand:
    \begin{equation*}
        \begin{split}
            \sum_{\substack{k_1,k_5,k_6 \\ p_1,p_2,p_2,p_4,p_5}} C(\vec{k},\vec{p}) \cdot & \Big| \frac{\psi_{2s}(\vec{k})}{\Omega(\vec{k})} \Big|^2 |w_{p_1}^{P_1}...w_{p_5}^{P_5}w_{k_5}^{M_5}|^2  \lesssim (N_{(1)}^{2s-2} N_{(3)}^2)^2  \sum_{\substack{k_1,k_5,k_6 \\ p_1,p_2,p_2,p_4,p_5}} C(\vec{k},\vec{p}) \cdot |w_{p_1}^{P_1}...w_{p_5}^{P_5}w_{k_5}^{M_5}|^2 \\
            & \lesssim (N_{(1)}^{2s-2} N_{(3)}^2)^2 \sum_{\substack{k_5 \\ p_1,p_2,p_2,p_4,p_5}} |w_{p_1}^{P_1}...w_{p_5}^{P_5}w_{k_5}^{M_5}|^2 \cdot \underbrace{\sum_{k_1,k_6} \1_{k_6=k_1-k_2...+k_5} \cdot \1_{\lincba}}_{\leq 1} \\
            & \lesssim (N_{(1)}^{2s-2} N_{(3)}^2 \norm{w^{P_1}}_{\L^2}...\norm{w^{P_5}}_{\L^2}\norm{w^{M_5}}_{\L^2})^2 \\
            & \lesssim \left(N_{(1)}^{2s-2} N_{(3)}^2 (P_1...P_5 M_5)^{-1}\norm{w^{P_1}}_{H^1}...\norm{w^{P_5}}_{H^1}\norm{w^{M_5}}_{H^1}\right)^2
        \end{split}
    \end{equation*}
    and on the other hand,
    \begin{equation*}
            \sum_{\substack{k_1,k_2,k_3,k_4,k_5,k_6 \\ p_1,p_2,p_2,p_4,p_5}} C(\vec{k},\vec{p}) |w_{k_6}^{N_6}|^2 = \sum_{k_6} |w_{k_6}^{N_6}|^2 \cdot \sum_{\substack{k_1,k_2,k_3,k_4,k_5\\ p_1,p_2,p_2,p_4,p_5}} C(\vec{k},\vec{p})
            \lesssim N_{(2)}N_{(3)}P_1...P_5 M_5 \cdot (M_6^{-1}\norm{w^{M_6}}_{H^1})^2
    \end{equation*}
    where we used the counting bound from Lemma~\ref{lem counting bound}. We deduce from the two inequalities above that
    \begin{equation*}
    \begin{split}
        \biggl( \sum_{k_2,k_3,k_4} \big| \sum_{\substack{k_1,k_5,k_6 \\ p_1,p_2,p_3,p_4,p_5}} &C(\vec{k},\vec{p}) \cdot \frac{\psi_{2s}(\vec{k})}{\Omega(\vec{k})}w_{p_1}^{P_1}\cjg{w_{p_2}^{P_2}}w_{p_3}^{P_3}\cjg{w_{p_4}^{P_4}}w_{p_5}^{P_5}w_{k_5}^{M_5}\cjg{w_{k_6}^{M_6}}  \big|^2 \biggr)^{\frac{1}{2}}  \\
        & \lesssim N_{(1)}^{2s-2} N_{(3)}^2 N_{(1)}^{\frac{1}{2}} N_{(3)}^{\frac{1}{2}}\norm{w^{P_1}}_{H^1}...\norm{w^{P_5}}_{H^1}\norm{w^{M_5}}_{H^1} \norm{w^{M_6}}_{H^1} \\
        & \lesssim N_{(1)}^{2s-\frac{3}{2}} N_{(3)}^{\frac{5}{2}} \norm{w^{P_1}}_{H^1}...\norm{w^{P_5}}_{H^1}\norm{w^{M_5}}_{H^1} \norm{w^{M_6}}_{H^1} 
    \end{split}
    \end{equation*}
    Coming back to \eqref{Q application Wiener chaos}, we deduce that 
    \begin{equation*}
         \norm{ \cQ_{M_1,...,M_6,P_1,...,P_5}^{(W)}}_{\L^p(d\mu_s|\cB_{\ll N_{(3)}})} \lesssim p^{\frac{3}{2}} N_{(1)}^{-\frac{3}{2}} N_{(3)}^{\frac{5}{2}-s} \norm{w^{P_1}}_{H^1}...\norm{w^{P_5}}_{H^1}\norm{w^{M_5}}_{H^1} \norm{w^{M_6}}_{H^1}  
    \end{equation*}
    And plugging this into \eqref{Q L^p less than L^p_cond Linfty}, we obtain :
    \begin{equation*}
        \norm{\1_{\{ \cC(u) \leq R \}} \cQ_{M_1,...,M_6,P_1,...,P_5}^{(W)}}_{\L^p(d\mu_s)} \lesssim_R p^{\frac{3}{2}}  N_{(1)}^{-\frac{3}{2}} N_{(3)}^{\frac{5}{2}-s}
    \end{equation*}
    Interpolating the above inequality with \eqref{Q deterministic estimate} (more precisely \eqref{Q deterministic estimate} with $\frkQ^{(W)}$ instead of $\frkQ$), we can conclude in the same way as what we did with $\cR^{(W)}$ in the proof of Lemma~\ref{lem R_s,N L^p estimate} (in particular using again Lemma~\ref{lem theta exponent < 1 on p and negative exponent on N_{(1)}}). 
    \end{proof}

In order to complete the proof of Lemma~\ref{lem Q1,Q2 estimate}, we provide in the next paragraph a proof of the technical Lemma~\ref{lem M3^2 Mj^-1 P1^-1 Pj^-1/2}.\\

\paragraph{\textbf{Proof of Lemma~\ref{lem M3^2 Mj^-1 P1^-1 Pj^-1/2}}}

\begin{proof}[Proof of Lemma~\ref{lem M3^2 Mj^-1 P1^-1 Pj^-1/2}]Recall that we use Notations \ref{notations Q estimate}. \\
\textbullet Assume first that $(N_{(1)},N_{(2)}) \in \{ P_1,...,P_5\}^2$. Then, $P_{(1)}^{-1} P_{(2)}^{-\frac{1}{2}} \lesssim N_{(1)}^{-\frac{3}{2}}$. Moreover, we have
\begin{equation*}
    M_{(3)}^{2} (M_2...M_6)^{-1} \leq M_{(3)}^{2} (M_{(2)}M_{(3)}...M_{(6)})^{-1} \leq (M_{(4)}M_{(5)}M_{(6)})^{-1} \lesssim 1
\end{equation*}
Combining these two inequalities, we obtain 
\begin{equation*}
     M_{(3)}^{2} (M_2...M_6)^{-1} P_{(1)}^{-1}\left( P_{(2)} P_{(3)} P_{(4)} P_{(5)}\right)^{-\frac{1}{2}} \lesssim  N_{(1)}^{-\frac{3}{2}}(P_{(3)}P_{(4)}P_{(5)})^{-\frac{1}{2}} \lesssim  N_{(1)}^{-\frac{3}{2}}
\end{equation*}
which is the desired estimate. \bigskip

\textbullet Now, assume that $(N_{(1)},N_{(2)}) \notin \{ P_1,...,P_5\}^2$. It means that :
\begin{equation*}
\begin{split}
    (N_{(1)},N_{(2)}) \in \{M_2,...,M_6\}^2 \hspace{0.2cm} & \textbf{or} 
         \left( \hspace{0.2cm} N_{(1)} \in \{M_2,...,M_6\} \hspace{0.2cm} \textbf{and} \hspace{0.2cm} N_{(2)} \in \{P_1,...,P_5\} \hspace{0.2cm} \right) \hspace{0.2cm} \\
         & \textbf{or} \left( \hspace{0.2cm}
         N_{(2)} \in \{M_2,...,M_6\} \hspace{0.2cm} \textbf{and} \hspace{0.2cm} N_{(1)} \in \{P_1,...,P_5\} \hspace{0.2cm} \right)
\end{split}
\end{equation*}
-- Fristly, we suppose that $(N_{(1)},N_{(2)}) \in \{M_2,...,M_6\}^2$. Then, we note that $N_{(3)} \in \{M_2,...,M_6 \}$ \textbf{or} $ P_{(1)} = N_{(3)}$. And also,
\begin{equation*}
    N_{(4)} \in \{M_2,...,M_6\}  \hspace{0.4cm} \textbf{or}  \hspace{0.4cm} N_{(4)} \in \{P_1,...,P_5\}
\end{equation*}
Considering these facts (and still the fact that $N_{(1)} \sim N_{(2)}$), we deduce that
\begin{equation*}
     M_{(3)}^{2} (M_2...M_6)^{-1} P_{(1)}^{-1}\left( P_{(2)} P_{(3)} P_{(4)} P_{(5)}\right)^{-\frac{1}{2}} \lesssim M_{(3)}^2 N_{(1)}^{-2}N_{(3)}^{-1}N_{(4)}^{-\frac{1}{2}}
\end{equation*}
Using then the fact that $M_{(3)} \lesssim N_{(3)}$, we have 
\begin{equation*}
     M_{(3)}^{2} (M_2...M_6)^{-1} P_{(1)}^{-1}\left( P_{(2)} P_{(3)} P_{(4)} P_{(5)}\right)^{-\frac{1}{2}} \lesssim  N_{(1)}^{-2}N_{(3)} N_{(4)}^{-\frac{1}{2}}
\end{equation*}
which is the desired inequality. \\
--Secondly, we suppose that $N_{(1)} \in \{M_2,...,M_6\} \hspace{0.2cm} \textbf{and} \hspace{0.2cm} N_{(2)} \in \{P_1,...,P_5\}$. We consider two cases. \\
$a)$ The first one is when $M_1 \in \{ M_{(1)},M_{(2)},M_{(3)}\}$. In that case, $M_{(4)},M_{(5)},M_{(6)},P_{(2)},...,P_{(5)}$ are seven indices out of the ten indices $N_j$ (that would not have been the case if $M_{(1)} \in \{M_{(4)},M_{(5)},M_{(6)}\}$, because according to Notations \ref{notations Q estimate}, $M_1$ is not one of the $N_j$). As a consequence, 
\begin{equation*}
   \max \{M_{(4)},M_{(5)},M_{(6)},P_{(2)},...,P_{(5)}\} \geq N_{(4)} \hspace{0.3cm}  \textnormal{so,} \hspace{0.3cm}(M_{(4)}M_{(5)}M_{(6)})^{-1}(P_{(2)}...P_{(5)})^{-\frac{1}{2}} \lesssim N_{(4)}^{-\frac{1}{2}}
\end{equation*}
On the other hand, using the fact that $M_{(1)} \sim M_{(2)}$, we have 
\begin{equation*}
     M_{(3)}^{2} (M_2 M_3...M_6)^{-1} \lesssim  M_{(3)}^{2} (M_{(2)}M_{(3)}...M_{(6)})^{-1} \lesssim M_{(1)}^{-1} M_{(3)} (M_{(4)}M_{(5)}M_{(6)})^{-1}
\end{equation*}
Combining these two inequalities and recalling that $M_{(1)},P_{(1)} \sim N_{(1)}$, we obtain
\begin{equation*}
    M_{(3)}^{2} (M_2...M_6)^{-1} P_{(1)}^{-1}\left( P_{(2)} P_{(3)} P_{(4)} P_{(5)}\right)^{-\frac{1}{2}} \lesssim N_{(1)}^{-2} M_{(3)}N_{(4)}^{-\frac{1}{2}} \lesssim N_{(1)}^{-2} N_{(3)}N_{(4)}^{-\frac{1}{2}}
\end{equation*}
$b)$ The second case we consider is when $M_1 \in \{ M_{(4)},M_{(5)},M_{(6)}\}$. In that case, $\{ M_{(1)},M_{(2)},M_{(3)}\} \subset \{M_2,M_3,...,M_6 \}$. Thus,
\begin{equation*}
    (M_2...M_6)^{-1} \lesssim (M_{(1)}M_{(2)}M_{(3)}M_{(5)}M_{(6)})^{-1}
\end{equation*}
so,
\begin{equation*}
    M_{(3)}^2(M_2...M_6)^{-1} \lesssim N_{(1)}^{-2}M_{(3)}
\end{equation*}
and,
\begin{equation*}
     M_{(3)}^{2} (M_2...M_6)^{-1} P_{(1)}^{-1}\left( P_{(2)} P_{(3)} P_{(4)} P_{(5)}\right)^{-\frac{1}{2}} \lesssim N_{(1)}^{-3}M_{(3)} \lesssim N_{(1)}^{-3}N_{(3)}
\end{equation*}
which is an even better estimate than the one we desired. \\
--Finally, the case  $N_{(2)} \in \{M_2,...,M_6\} \hspace{0.2cm} \textbf{and} \hspace{0.2cm} N_{(1)} \in \{P_1,...,P_5\}$ is identical to the previous one. \\
Hence, the proof of Lemma~\ref{lem M3^2 Mj^-1 P1^-1 Pj^-1/2} is finished.

\end{proof}

\appendix 

\section{Construction and properties of the flow and the truncated flow}\label{appendix Construction and properties of the flow and the truncated flow}
In this appendix, we study the local and global wellposedness in $H^{\sigma}(\T)$, $\sigma \geq 1$, for both equation \eqref{NLS} and truncated equation \eqref{truncated equation}. We also study approximation properties of the truncated flow along with its structure.

\subsection{Local and global wellposedness}
We begin this paragraph by recalling some facts about the space :
\begin{equation*}
    W(\T) := \{u \in \cD'(\T) : \hsp \underset{k \in \Z}{\sum} \hsp|\widehat{u}(k)| < +\infty\}
\end{equation*}
of absolutely convergent Fourier series, called the \textit{Wiener algebra}, and equipped with the norm 
\begin{equation*}
    \norm{u}_{W(\T)} := \underset{k \in \Z}{\sum} \hsp |\widehat{u}(k)|
\end{equation*} 
\begin{enumerate}
    \item $W(\T)$ is a Banach algebra with the estimate
    \begin{equation*}
        \norm{uv}_{W(\T)} \leq \norm{u}_{W(\T)} \norm{v}_{W(\T)}
    \end{equation*}
    \item For every $\sigma \geq 0$, $W(\T) \cap \hsigt$ is an algebra. Furthermore, there exists $C_{\sigma}>0$ such that 
    \begin{equation*}
        \norm{uv}_{\hsig} \leq C_{\sigma} (\norm{u}_{\hsig} \norm{v}_{W(\T)} + \norm{v}_{\hsig} \norm{u}_{W(\T)})
    \end{equation*}
    In particular, if $\sigma > \frac{1}{2}$, $\hsigt$ is a Banach algebra and there exists $C'_{\sigma}>0$ such that
    \begin{equation*}
        \norm{uv}_{\hsig} \leq C'_{\sigma} \norm{u}_{\hsig} \norm{v}_{\hsig}
    \end{equation*}
\end{enumerate}

We are now ready to prove the following local existence theorem :

\begin{thm}\label{thm local wellposedness}
    Let $\sigma \geq 1$. Let $t_0 \in \R$. Both equation \eqref{NLS} and truncated equation \eqref{truncated equation} are locally well-posed in $\hsigt$, in the sense that for any $R_0>0$, there exists $T_0>0$ such that for every $\norm{u}_{\hsig} \leq R_0$ there exists a unique $u \in \cC([ t_0-T_0,t_0+T_0],\hsigt)$ that satisfies the Duhamel formula :
    \begin{equation}
        u(t) = e^{i(t-t_0)\p_x^2}u_0 -i \int_{t_0}^t e^{i(t-\tau)\p_x^2}(|u(\tau)|^4u(\tau))d\tau
    \end{equation}
    for all $|t-t_0|\leq T_0$. And, for every $N \in \N$, there exists a unique $v_N \in \cC([ t_0-T_0,t_0+T_0],\hsigt) $ that satisfies the Duhamel formula :
    \begin{equation}
        v_N(t) = e^{i(t-t_0)\p_x^2}u_0 -i \int_{t_0}^t e^{i(t-\tau)\p_x^2}\Pi_N(|\Pi_N v_N(\tau)|^4 \Pi_N v_N(\tau))d\tau
    \end{equation}
    for all $|t-t_0|\leq T_0$.
\end{thm}

\begin{rem}
    It is important to notice that the existence time $T_0 >0$ is the same for both NLS and truncated equation and does not depend on the integer $N \in \N$. 
\end{rem}

\begin{proof}
To show Theorem~\ref{thm local wellposedness}, we apply a fixed point argument. Fix $R_0>0$ and $|| u_0||_{\hsig} \leq R_0 $. For $T>0$, we denote 
    \begin{equation*}
        X_T := \cC([t_0-T,t_0,+T],\hsigt)
    \end{equation*}
which is a Banach space when endowed with the sup norm 
\begin{equation*}
    \norm{u}_{X_T} := \underset{\tau \in [t_0-T,t_0,+T]}{\textnormal{sup}} \norm{u(\tau)}_{\hsig} 
\end{equation*} 
For every $N\in \N \cup \{ \infty\}$, let us consider the well-defined map 
    \begin{equation*}
        \begin{split}
           \Gamma_N : \hspace{0.1cm} & X_T \longrightarrow X_T \\
            & u \longmapsto e^{i(t-t_0)\p_x^2}u_0 - i \int_{t_0}^t e^{i(t -\tau)\p_x^2}\Pi_N (|\Pi_N u(\tau)|^4 \Pi_N u(\tau)) d\tau
        \end{split}
    \end{equation*}
where by convention $\Pi_{\infty}= id$. We recall the crucial point that $||\Pi_N u ||_{\hsig} \leq \norm{u}_{\hsig}$. Thanks to that point, every estimate below are uniform in $N$ so that every constant below does not depend on N.\\

Using now the facts that $\hsigt$ is an algebra (since $\sigma \geq 1 > \frac{1}{2}$) and $e^{i(t-t_0)\p_x^2}$ is a linear isometry on $\hsigt$, we have for some constant $C_0 > 0$,
\begin{equation}\label{gamma u estimate lwp}
    \norm{\Gamma_N u}_{X_T} \leq \norm{u_0}_{\hsig} + C_0 T \norm{u}_{X_T}^5 
\end{equation}

In addition, using the multilinearity of $(u_1,u_2,u_3,u_4,u_5) \longmapsto u_1 \cjg{u_2} u_3 \cjg{u_4}u_5 $ we have for some constant $C_1 > 0$,
\begin{equation} \label{gamma u - gamma v estimate lwp}
    \begin{split}
    \norm{\Gamma_N u - \Gamma_N v}_{X_T} & = \lnorm \int_{t_0}^t e^{i(t -\tau)\p_x^2} \Pi_N(|\Pi_Nu(\tau)|^4\Pi_Nu(\tau)- |\Pi_Nv(\tau)|^4 \Pi_Nv(\tau))d\tau \rnorm_{X_T} \\
    & \leq C_1 T (\norm{u}^4_{X_T} + \norm{v}^4_{X_T})\norm{u-v}_{X_T}
    \end{split}
\end{equation}
Let $R > 2R_0$ (for example $R= 1 + 2R_0$) and $T_0 := \frac{1}{3 \text{max}(C_0,C_1) R^4}$. Fix $0<T \leq T_0$ and denote by $\cjg{B}_R(T)$ the closed centered ball in $X_T$ of radius $R$. We get from \eqref{gamma u estimate lwp} and \eqref{gamma u - gamma v estimate lwp} and our choice of the parameters $R$ and $T_0$ that for any $u,v \in \cjg{B}_R(T)$,
\begin{equation}
  ||\Gamma_N u||_{\cjg{B}_R(T)} \leq \frac{R}{2} + \frac{R}{2} \leq R 
\end{equation}
and,
\begin{equation}
|| \Gamma_N u - \Gamma_N v ||_{\cjg{B}_R(T)} \leq 2C_1 T R^4  \norm{u-v}_{\cjg{B}_R(T)} \leq \frac{2}{3} \norm{u-v}_{\cjg{B}_R(T)}
\end{equation}
These two estimates imply that the map $\Gamma_N$ is a contraction from (the complete space) $\cjg{B}_R(T)$ to itself. Hence, applying the Banach's fixed point theorem to the map 
\begin{equation*}
    \Gamma_N : \cjg{B}_R(T) \longrightarrow \cjg{B}_R(T)
\end{equation*}
leads to the existence part of Theorem~\ref{thm local wellposedness}. However, through classical considerations, we can prove that the uniqueness holds in the entire space $\cC([ t_0-T,t_0+T],\hsigt)$. So the proof is completed.
\end{proof}

Putting together the local solutions from Theorem~\ref{thm local wellposedness}, we obtain the two following corollaries.

\begin{cor}\label{cor maximal solution NLS} Let $\sigma \geq1$. For every $u_0 \in \hsigt$, there exists a unique maximal solution $u \in \cC(I_{max}(u_0),\hsigt)$ of \eqref{NLS} where $I_{max}(u_0)$ is an open interval containing 0. Moreover, if $I_{max}(u_0)$ is strictly included in $\R$, then 
\begin{equation*}
    \norm{u(t)}_{\hsig} \longrightarrow +\infty
\end{equation*}
as $t \longrightarrow \p I_{max}(u_0)$.
\end{cor}
\begin{cor}\label{cor maximal solution truncated equation}
   Let $\sigma \geq1$. For every $u_0 \in \hsigt$ and every $N \in \N$, there exists a unique maximal solution $u_N \in \cC(I_{max,N}(u_0),\hsigt)$ of \eqref{truncated equation} where $I_{max,N}(u_0)$ is an open interval containing 0. Moreover, if $I_{max,N}(u_0)$ is strictly included in $\R$, then 
\begin{equation*}
    ||u_N(t)||_{\hsig} \longrightarrow +\infty
\end{equation*}
as $t \longrightarrow \p I_{max,N}(u_0)$.
\end{cor}

Thus, to obtain that the solutions of equations \eqref{NLS} and \eqref{truncated equation} are global, it suffices to show that the $\hsigt$-norm of the solutions do not blow up in finite time. Fortunately, we have : 

\begin{prop}\label{prop exponential bound} Let $u_0 \in \hsigt$. There exists a constant $C_0 > 0$ only depending on $\norm{u}_{H^1(\T)}$ and $\sigma\geq 1$, such that
\begin{equation}\label{exponential bound for NLS solution}
    \norm{u(t)}_{\hsig} \leq \norm{u}_{\hsig} e^{C_0|t|}, \hsp \hsp\textnormal{for all} \hsp t \in I_{max}(u_0)
\end{equation}
    and,
\begin{equation}\label{exponential bound for truncated solution}
   ||u_N(t)||_{\hsig} \leq \norm{u}_{\hsig} e^{C_0|t|},  \hsp \hsp \textnormal{for all} \hsp t \in I_{max,N}(u_0)
\end{equation}
where $u$ is the maximal solution of \eqref{NLS} and $u_N$ is the maximal solution of \eqref{truncated equation}. Moreover, we can explicitly choose $C_0 \leq C_{\sigma} (1+ \norm{u}_{H^1(\T)})^{12}$ where $C_{\sigma} >0$ depends only on $\sigma$.
\end{prop}

\begin{proof}
    We use the convention $\Pi_{\infty} = id$, $u_{\infty} = u$ and $I_{max,\infty}(u_0) = I_{max}(u_0)$. Let $N \in \N \cup \{ \infty \}$ and $t \in I_{max,N}(u_0)$. Passing to the $\hsig$-norm in the Duhamel formula
    \begin{equation*}
        u_N(t) = e^{i(t)\p_x^2}u_0 -i \int_{0}^t e^{i(t-\tau)\p_x^2}\Pi_N(|\Pi_N u_N(\tau)|^4 \Pi_N u_N(\tau))d\tau
    \end{equation*}
    we get that, for some constant $C>0$ (only depending on $\sigma$),
    \begin{equation}\label{exponential bound, grönwall}
        \begin{split}
            ||u_N(t)||_{\hsig} & \leq \norm{u}_{\hsig} + \mid \int_0^t \norm{|\Pi_N u_N(\tau)|^4 \Pi_N u_N(\tau)}_{\hsig} d\tau\mid \\
            & \leq \norm{u}_{\hsig} + C \mid \int_0^t \norm{u_N(\tau)}_{\hsig}\norm{u_N(\tau)}^4_{W(\T)} d\tau \mid 
        \end{split}
    \end{equation}
The crucial point of the proof is that $\norm{u_N(\tau)}_{W(\T)}$ is bounded by the $H^1(\T)$-norm of the initial data. Indeed, using the mass and Hamiltonian conservation 
\begin{equation*}
    \begin{split}
        \norm{u_N(\tau)}_{W(\T)} & \leq \norm{u_N(\tau)}_{H^1(\T)} \\
        & \leq (\norm{u_N(\tau)}^2_{\Ltwo} + \norm{\p_xu_N(\tau)}^2_{\Ltwo})^{\frac{1}{2}} \\
        & \lesssim \left( \norm{u_N(\tau)}^2_{\Ltwo} + \left(\frac{1}{2}\norm{\p_xu_N(\tau)}^2_{\Ltwo} + \frac{1}{6}\norm{u_N(\tau)}^6_{\Lsix}\right) \right)^{\frac{1}{2}} \\
        & \lesssim (\norm{u}^2_{\Ltwo} + \frac{1}{2}\norm{\p_xu_0}^2_{\Ltwo} + \frac{1}{6}\norm{u}^6_{\Lsix})^{\frac{1}{2}} \\
        & \leq C'_0
     \end{split}   
\end{equation*}
where we can choose $C'_0 \sim (1 + \norm{u}_{H^1(\T)})^3$, because $H^1(\T)$ embeds continuously in $\L^6(\T)$. Now, plugging this bound into \eqref{exponential bound, grönwall}, we get 
\begin{equation*}
    ||u_N(t)||_{\hsig} \leq \norm{u}_{\hsig} + C_0 \mid \int_0^t  \norm{u_N(\tau)}_{\hsig} d\tau \mid
\end{equation*}
where we can choose $C_0 \sim (1 + \norm{u}_{H^1(\T)})^{12}$. Then, applying Gronwall's inequality, we obtain 
\begin{equation*}
    ||u_N(t)||_{\hsig} \leq \norm{u}_{\hsig} e^{C_0|t|}
\end{equation*}
which is the desired bound.
\end{proof}
\begin{thm}
    For every $u_0 \in \hsigt$, each maximal  solution of \eqref{NLS} and \eqref{truncated equation} is global. In other words,
    \begin{equation*}
        I_{max}(u_0), I_{max,N}(u_0) = +\infty
    \end{equation*}
    for all $N \in \N$.
\end{thm}

\begin{proof} The estimates from Proposition~\ref{prop exponential bound} imply that the $\hsig$-norm of both solutions does not blow up in finite time. Therefore, Corollaries \ref{cor maximal solution NLS} and \ref{cor maximal solution truncated equation} lead to the result.
\end{proof}

%%%%%%%%%%%%%%%%%%%%%%%%%%%%%%%%%%%%%%%%%%%%%%%%%%%%%%%%%%%%%%%%%%%%%%%%%%%%%%%%%%%%%%%%%%%%%%%%%%%%%%%%%%%%%%%%%%%%%%%%%%%%%%%%%%%%%%

\subsection{Regularity of the flows and approximation properties }
We recall that we denote by $\Phi(t)$ the flow of \eqref{NLS} and by $\Phi_N(t)$ the flow of the truncated equation \eqref{truncated equation}. We also use the notation $\Phi_{\infty}(t) = \Phi(t)$.
\begin{prop}\label{appendix prop first approximation prop}
    Let $\sigma \geq 1$. Let $R>0$ and $T>0$. There exists a constant $\Lambda(R,T) > 0$ only depending on $T,R$ and $\sigma$, such that for any $u_0 \in \hsigball{R}$,
    \begin{equation*}
        \sup_{|t| \leq T} \hsignorm{\Phi(t)u_0} + \sup_{|t| \leq T} \hsignorm{\Phi_N(t)u_0} \leq \Lambda(R,T), \hsp \hsp \forall N \in \N
    \end{equation*}
\end{prop}

\begin{proof}
    This is a consequence of Proposition~\ref{prop exponential bound}. For example we can take 
    \begin{equation*}
        \Lambda(R,T) = R e^{C_{\sigma}(1+R)^{12}T}
    \end{equation*}
    for a certain constant $C_{\sigma}>0$ only depending on $\sigma$.
\end{proof}

We continue this paragraph with two significant regularity results for the flow of \eqref{NLS} and \eqref{truncated equation}.

\begin{prop}\label{prop continuity of the flow}Let $t \in \R$.
For every $N\in \N \cup \{ \infty \}$, the map
    \begin{equation*}
        \begin{split}
            \Phi_N(t) : \hsp & \hsigt \lra \hsigt \\
                        & u \longmapsto \Phi_N(t)u
        \end{split}
    \end{equation*}
is continuous.
\end{prop}
\begin{proof}For better readability, we only provide the proof for $\Phi(t)$.
    Let $u,v \in \hsigt$. Without loss of generality, we can assume that $\hsignorm{u} \leq R $ and $\hsignorm{v} \leq R$ for some $R>0$. We invoke $\Lambda(R,t)>0$ from the Proposition~\ref{appendix prop first approximation prop}. Firstly, using the Duhamel formula and passing to the $\hsigt$-norm we get that for all $s \leq t$
    \begin{equation*}
        \hsignorm{\Phi(s)v - \Phi(s)u} \leq \hsignorm{v-u} + C \Lambda(R,t)^4 \mid \int_0^s \hsignorm{\Phi(\tau)v - \Phi(\tau)u}d\tau \mid
    \end{equation*}
    Secondly, the Gronwall's inequality yields 
    \begin{equation*}
        \hsignorm{\Phi(t)v - \Phi(t)u} \leq \hsignorm{v-u} e^{C \Lambda(R,t)^4 |t|}
    \end{equation*}
    In particular,
    \begin{equation*}
        \hsignorm{\Phi(t)v - \Phi(t)u} \tendsto{v \ra u} 0
    \end{equation*}
\end{proof}

We can go even further with the following proposition:

\begin{prop}\label{prop bicontinuity of the flow}
    For every $N\in \N \cup \{ \infty \}$, the map
    \begin{equation*}
        \begin{split}
            \Phi_N : \hsp \R &\times \hsigt \lra \hsigt \\
             & (t,u) \longmapsto \Phi_N(t)u
        \end{split}
    \end{equation*}
    is continuous.
\end{prop}

\begin{proof} For better readability, we only provide the proof for $\Phi$.
We fix $t \in \R$ and $u \in \hsigt$ and we show that 
\begin{equation*}
    \hsignorm{\Phi(t+h)(u+\delta u) - \Phi(t)u} \lra 0, \hspace{0.5cm} \text{as $h \ra 0 $ and $\delta u \ra 0$}
\end{equation*}
We denote $F(\tau,v) := |\Phi(\tau)v|^4\Phi(\tau)v$. Let us compute the difference $\Phi(t+h)(u+\delta u) - \Phi(t)u$ using the Duhamel formula. 
\begin{equation*}
    \begin{split}
        &\Phi(t+h)(u+\delta u) - \Phi(t)u  =  \Bigl(\Phi(t+h)(u+\delta u) - \Phi(t+h)u\Bigr) + \Bigl(\Phi(t+h)u -\Phi(t)u\Bigr)\\
    & = e^{i(t+h)\p_x^2} \delta u - i \int_0^{t+h} e^{i(t+h-\tau)\p_x^2}(F(\tau,u+\delta u)-F(\tau,u)) d\tau \\
   & + e^{it\p_x^2}(e^{ih\p_x^2}-1)u -i \int_0^{t+h} e^{i(t+h-\tau)\p_x^2} F(\tau,u) d\tau + i \int_0^t e^{i(t-\tau)\p_x^2} F(\tau,u) d\tau \\
   & = e^{i(t+h)\p_x^2} \delta u  + e^{it\p_x^2}(e^{ih\p_x^2} -1)u - i \int_0^{t+h} e^{i(t+h-\tau)\p_x^2}(F(\tau,u+\delta u)-F(\tau,u)) d\tau \\
   & -i \int_t^{t+h} e^{i(t+h-\tau)\p_x^2} F(\tau,u) d\tau  -i \int_0^{t} e^{i(t-\tau)\p_x^2}(e^{ih\p_x^2}-1) F(\tau,u) d\tau
    \end{split}
\end{equation*}

Furthermore, without loss of generality we can assume that $t+h \leq t+1 =: T$ and $\hsignorm{u+\delta u} \leq \hsignorm{u}+1 =: R$. Now, we invoke $\Lambda(R,T)>0$ from Proposition~\ref{appendix prop first approximation prop}. Passing to the $\hsigt$-norm in the above formula we have 
\begin{equation*}
    \begin{split}
         \hsignorm{\Phi(t+h)(u+\delta u) - \Phi(t)u} & \leq \hsignorm{\delta u} + \hsignorm{(e^{ih\p_x^2}-1)u} \\
        & + C \Lambda(R,T)^4 \mid \int_0^T \hsignorm{\Phi(\tau)(u+\delta u)- \Phi(\tau)u} d\tau \mid \\
        & + |h| \Lambda(R,T)^5 + \mid \int_0^t \hsignorm{(e^{ih\p_x^2}-1)F(\tau,u)} d\tau \mid 
    \end{split}
\end{equation*}
Using the continuity of $v \longmapsto \Phi(\tau)v$ (Proposition~\ref{prop continuity of the flow}) and the dominated convergence theorem, we deduce that the right hand side tends to $0$ as $h \ra 0 $ and $\delta u \ra 0$. So the proof of Proposition~\ref{prop bicontinuity of the flow} is complete.
\end{proof}

\begin{prop}[Approximation property]\label{appendix second approximation prop}
    Let $\sigma \geq 1$. Let $K$ be a compact subset of $\hsigt$ and $T>0$. Then, uniformly in $|t| \leq T$ and $u_0 \in K$,
    \begin{equation*}
         \lim_{N \rightarrow +\infty} \hsignorm{\Phi(t)u_0-\Phi_N(t)u_0} = 0,
    \end{equation*}
In other words,
\begin{equation*}
    \sup_{|t|\leq T} \hsp \sup_{u_0 \in K} \hsp \hsignorm{\Phi(t)u_0-\Phi_N(t)u_0} \tendsto{N \ra +\infty} 0  
\end{equation*}
\end{prop}

\begin{proof}
    First, from the compactness of $K$, we invoke $R>0$ such that $K \subset \hsigball{R}$, where $\hsigball{R}$ is the closed centered ball of radius $R$ in $\hsigt$. We then invoke $\Lambda(R,T)>0$ from Proposition~\ref{appendix prop first approximation prop}. Now, we set $R_1 := 1 + 2\Lambda(R,T)$ and invoke $\delta = \delta(R_1) > 0$ the local existence time associated to $R_1$ from the local theory (Theorem~\ref{thm local wellposedness}). These parameters ensure that for every $N \in \N \cup \{ \infty\}$ and $u_0 \in \hsigball{R}$, the Duhamel map,
\begin{equation*}
        \begin{split}
           \Gamma_N :  \hsp &\cjg{B}_{R_1}(\delta)  \longrightarrow \cjg{B}_{R_1}(\delta) \\
            & u \longmapsto e^{it\p_x^2}u_0 - i \int_{0}^t e^{i(t -\tau)\p_x^2}\Pi_N (|\Pi_N u(\tau)|^4 \Pi_N u(\tau)) d\tau
        \end{split}
    \end{equation*}
is a contraction, with a universal contraction coefficient $0<\gamma<1$ (for example $\gamma = \frac{2}{3}$ as in the proof of Theorem~\ref{thm local wellposedness}), where $\cjg{B}_{R_1}(\delta)$ is the closed centered ball in $\cC([-\delta,\delta],\hsigt)$ of radius $R_1$. \\

\underline{Step 1: Local-time convergence}  \\

Firstly, we prove the local property
\begin{equation}\label{local approximation estimate}
    \sup_{|t|\leq \delta} \hsp \sup_{u_0 \in K} \hsp \hsignorm{\Phi(t)u_0-\Phi_N(t)u_0} \tendsto{N \ra +\infty} 0
\end{equation}
Let $u_0 \in \hsigball{R}$. We denote $u(t) = \Phi(t)u_0 $ and $u_N(t) = \Phi_N(t)u_0$. Since $u$ and $u_N$ are respectively the fixed point of $\Gamma_{\infty}$ and $\Gamma_N$, we have
\begin{equation*}
    \begin{split}
        \norm{u -u_N}_{\cjg{B}_{R_1}(\delta)} & = \norm{\Gamma_{\infty}u - \Gamma_N u_N}_{\cjg{B}_{R_1} (\delta)} \\
        & \leq \norm{\Gamma_{\infty}u - \Gamma_N u}_{\cjg{B}_{R_1}(\delta)} + \norm{\Gamma_N u - \Gamma_N u_N}_{\cjg{B}_{R_1}(\delta)} \\
        & \leq \norm{\Gamma_{\infty}u - \Gamma_N u}_{\cjg{B}_{R_1}(\delta)} + \gamma \norm{u -u_N}_{\cjg{B}_{R_1}(\delta)}
    \end{split}
\end{equation*}
Therefore,
\begin{equation*}
    \norm{u -u_N}_{\cjg{B}_{R_1}(\delta)} \leq \frac{1}{1-\gamma} \norm{\Gamma_{\infty}u - \Gamma_N u}_{\cjg{B}_{R_1}(\delta)}
\end{equation*}
Hence, to prove \eqref{local approximation estimate}, it suffices to show that
\begin{equation*}
    \sup_{u_0 \in K} \norm{\Gamma_{\infty}u - \Gamma_N u}_{\cjg{B}_{R_1}(\delta)} \tendsto{N \ra +\infty} 0
\end{equation*}
Next, for every $|t|\leq \delta$ we have,
\begin{equation*}
    \begin{split}
         & \Gamma_{\infty}u(t) - \Gamma_N u(t) = -i \int_0^t e^{i(t-\tau)\p_x^2}(|u(\tau)|^4u(\tau) - \Pi_N(|\Pi_Nu(\tau)|^4\Pi_Nu(\tau)))d\tau \\
         & = -i \int_0^t e^{i(t-\tau)\p_x^2}\Pi_N^{\perp}(|u(\tau)|^4u(\tau))d\tau -i \int_0^t e^{i(t-\tau)\p_x^2}\Pi_N(|u(\tau)|^4u(\tau) - 
         |\Pi_Nu(\tau)|^4\Pi_Nu(\tau))d\tau
    \end{split}
\end{equation*}
So passing to the sup norm we get,
\begin{equation*}
    \begin{split}
        \norm{\Gamma_{\infty}u - \Gamma_N u}_{\cjg{B}_{R_1}(\delta)} & \leq  \delta  \hsp \sup_{|\tau|\leq \delta} \hsignorm{\Pi_N^{\perp}(|u(\tau)|^4u(\tau))} \\
        & + C \delta \hsp \sup_{|\tau| \leq \delta} \hsignorm{u(\tau)- \Pi_N u(\tau)}(\hsignorm{u(\tau)}^4+ \hsignorm{\Pi_N u(\tau)}^4) \\
        & \leq C \delta(\sup_{|\tau|\leq \delta} \hsignorm{\Pi_N^{\perp}(|u(\tau)|^4u(\tau))} + \sup_{|\tau|\leq \delta}\hsignorm{\Pi_N^{\perp}u(\tau)} \Lambda(R,T)^4 )
    \end{split}
\end{equation*}
Taking the supremum over $u_0 \in K$ we then obtain,
\begin{equation}\label{estimate sup over t and K}
    \begin{split}
         \sup_{u_0 \in K} & \norm{\Gamma_{\infty}u - \Gamma_N u}_{\cjg{B}_{R_1}(\delta)} \\ 
        & \leq C\delta (\sup_{u_0 \in K} \sup_{|\tau|\leq \delta} \hsignorm{\Pi_N^{\perp}(|u(\tau)|^4u(\tau))} + \Lambda(R,T)^4 \sup_{u_0 \in K}  \sup_{|\tau|\leq \delta} \hsignorm{\Pi_N^{\perp}u(\tau)} )
  \end{split}
\end{equation}
Besides, based on a classical result in functional analysis (see Lemma~\ref{lem cvgce on compact set of bdd lin map}), $\Pi_N^{\perp}$ satisfies the key property of converging uniformly to $0$ on compact sets as $N \ra + \infty$. At the same time, the two following sets 
\begin{align*}
    K_1 :&= \{\Phi(\tau)u_0 : \hsp u_0 \in K, \hsp |\tau|\leq \delta \},  &  K_2:&= \{|\Phi(\tau)u_0|^4\Phi(\tau)u_0 : \hsp u_0 \in K, \hsp |\tau|\leq \delta \}
\end{align*}
are compacts in $\hsigt$. Indeed, it results from the facts that the map 
\begin{equation*}
    \begin{split}
         & \R \times \hsigt \lra \hsigt \\
                    & (t,u_0) \longmapsto \Phi(t)u_0
    \end{split} 
\end{equation*}
and the map
\begin{equation*}
    \begin{split}
        & \hsigt^5 \lra \hsigt \\
         (u_1,u_2&,u_3,u_4,u_5) \longmapsto u_1 \cjg{u_2} u_3 \cjg{u_4} u_5
    \end{split}
\end{equation*}
are continuous and the fact that the image of a compact set under a continuous map is compact. Now, rewriting \eqref{estimate sup over t and K} we get
\begin{equation*}
    \begin{split}
         \sup_{u_0 \in K} \norm{\Gamma_{\infty}u -\Gamma_N u}_{\cjg{B}_{R_1}(\delta)} & \leq C\delta (\sup_{w_2 \in K_2} \Pi_N^{\perp} w_2 + \Lambda(R,T)^4 \sup_{w_1 \in K_1} \Pi_N^{\perp} w_1) \\
         & \tendsto{N \ra \infty} 0
    \end{split}
\end{equation*}
And this implies the desired local property \eqref{local approximation estimate}. \\

\underline{Step 2: Long-time convergence} \\

Secondly, we complete the proof of Proposition~\ref{appendix second approximation prop} by iterating this local argument. Let $m := \lfloor \frac{T}{\delta}\rfloor + 1$, and for any integer $|k| \leq m$, let $I_k := [k\delta-\delta,k\delta+\delta]$. We show that for any $|k| \leq$ m,
\begin{equation}\label{local approximation propagated}
    \sup_{t \in I_k} \hsp \sup_{u_0 \in K} \hsp \hsignorm{\Phi(t)u_0-\Phi_N(t)u_0} \tendsto{N \ra +\infty} 0  
\end{equation}
If we do so, the proof of Proposition~\ref{appendix second approximation prop} will be completed. At this stage, we know that \eqref{local approximation propagated} is true when $k=0$. Now, we assume that \eqref{local approximation propagated} is true for some integer $|k|\leq m-1$ and we show that \eqref{local approximation propagated} is still true for every integer $|k'| = |k|+1$. Also, without loss of generality, we assume that $k \geq 0$ and $k'=k+1$. The crucial point here is that since we chose $R_1 = 1 + 2\Lambda(R,T)$, we have that for every $N \in \N \cup \{ \infty\}$ and $u_0 \in \hsigball{R}$, the Duhamel map
\begin{equation*}
        \begin{split}
           \Gamma_N :  \hsp &\cjg{B}_{R_1}^{(k)}(\delta)  \longrightarrow \cjg{B}^{k}_{R_1}(\delta) \\
            & u \longmapsto e^{i(t-t_{k})\p_x^2}u_N(t_k) - i \int_{t_{k}}^t e^{i(t -\tau)\p_x^2}\Pi_N (|\Pi_N u(\tau)|^4 \Pi_N u(\tau)) d\tau
        \end{split}
    \end{equation*}
is a contraction with a universal contraction coefficient $0<\gamma<1$, where $t_{k}:= k \delta$, $u_{\infty}(t_k):=u(t_{k})$, and $\cjg{B}_{R_1}^{(k)}(\delta)$ is the closed centered ball in $\cC([t_k-\delta,t_k+\delta],\hsigt)$ of radius $R_1$. \\

Now, doing the same calculations as in the first step of the proof we obtain,
\begin{equation*}
    \norm{u -u_N}_{\cjg{B}^{(k)}_{R_1}(\delta)} \leq \frac{1}{1-\gamma} \norm{\Gamma_{\infty}u - \Gamma_N u}_{\cjg{B}^{(k)}_{R_1}(\delta)}
\end{equation*}
and, 
\begin{equation*}
    \begin{split}
         \sup_{u_0 \in K} & \norm{\Gamma_{\infty}u - \Gamma_N u}_{\cjg{B}^{(k)}_{R_1}(\delta)} \leq \sup_{u_0 \in K} \sup_{t \in I_k} \hsignorm{e^{i(t-t_k)\p_x^2}(u(t_k)-u_N(t_k))} \\
        & + C\delta (\sup_{u_0 \in K} \sup_{\tau \in I_k} \hsignorm{\Pi_N^{\perp}(|u(\tau)|^4u(\tau))} + \Lambda(R,T)^4 \sup_{u_0 \in K}  \sup_{\tau \in I_k} \hsignorm{\Pi_N^{\perp}u(\tau)} )
  \end{split}
\end{equation*}
On the right hand side, the first term : 
\begin{equation*}
    \sup_{u_0 \in K} \sup_{t \in I_k} \hsignorm{e^{i(t-t_k)\p_x^2}(u(t_k)-u_N(t_k))} = \sup_{u_0 \in K} \hsignorm{u(t_k)-u_N(t_k)}
\end{equation*}
tends to $0$ as $N \lra +\infty$ by our assumption. And, we handle the remaining term in the same way as in the first part of the proof. Finally \eqref{local approximation propagated} is true for all $|k|\leq m$, so the proof of Proposition~\ref{appendix second approximation prop} is completed.
\end{proof}

\begin{cor}\label{appendix set approximation}
    Let $\sigma \geq 1$, $T>0$ and let $K$ be a compact subset of $\hsigt$. Then, 
    \begin{enumerate}
        \item for any $\eps > 0$, there exists $N_0 \in \N$ such that for all $N\geq N_0$
    \begin{equation*}
        \Phi_N(t)(K) \subset \Phi(t)(K) + \hsigball{\eps}
    \end{equation*}
    for all $|t| \leq T$.
    \item for any $\eps > 0$, there exists $N_1 \in \N$ such that for all $N\geq N_1$
    \begin{equation*}
        \Phi(t)(K) \subset \Phi_N(t)(K + \hsigball{\eps})
    \end{equation*}
    for all $|t| \leq T$.
    \end{enumerate}
\end{cor}

\begin{proof}Let $\eps >0$.
    \begin{itemize}
        \item For the first point, we take $u_0 \in K$ and we write 
        \begin{equation}\label{phiN = phi + (phiN - phi)}
            \Phi_N(t)u_0 = \Phi(t)u_0 + \left( \Phi_N(t)u_0 - \Phi(t)u_0 \right)
        \end{equation}
        From Proposition \ref{appendix second approximation prop}, there exists $N_0 \in \N$ such that for all $N \geq N_0$ 
        \begin{equation*}
           \sup_{|t| \leq T} \sup_{v_0 \in K} \hsignorm{\Phi_N(t)v_0 - \Phi(t)v_0} \leq \eps
        \end{equation*}
        Thus, for all $N \geq N_0$, $\Phi_N(t)u_0 - \Phi(t)u_0 \in \hsigball{\eps}$ for all $|t|\leq T$ and all $u_0 \in K$. Coming back to \eqref{phiN = phi + (phiN - phi)}, it implies that for all $N \geq N_0$
        \begin{equation*}
            \Phi_N(t)u_0 \in \Phi(t)(K) + \hsigball{\eps}
        \end{equation*}
        for all $u_0 \in K$ and all $|t|\leq T$.
        \item The second point is a consequence of the first one. Let $|t|\leq T$. From the continuity of the map $\Phi(t)$ (see Proposition \ref{prop continuity of the flow}) we have that $\Phi(t)(K)$ is a compact subset of $\hsigt$. So from the first point there exists $N_1 \in \N$ such that for all $N\geq N_1$ 
        \begin{equation*}
            \Phi_N(-\tau)\left(\Phi(t)(K) \right) \subset \Phi(-\tau)\Phi(t)(K) + \hsigball{\eps} 
        \end{equation*}
        for all $|\tau| \leq T$. In particular, for $\tau = t$ we have :
        \begin{equation*}
            \Phi_N(-t)\left(\Phi(t)(K) \right) \subset \Phi(-t)\Phi(t)(K) + \hsigball{\eps} = K + \hsigball{\eps}
        \end{equation*}
        Applying $\Phi_N(t)$ to this we obtain 
        \begin{equation*}
            \Phi(t)(K) \subset \Phi_N(t) \left(  K + \hsigball{\eps} \right)
        \end{equation*}
    \end{itemize}
    This concludes the proof of Corollary \ref{appendix set approximation}.
\end{proof}

\subsection{Structure of the truncated flow}
We set
\begin{align*}
    E_N &:= \Pi_N \L^2(\T) \\
    E_N^{\perp} &:= \Pi_N^{\perp} \L^2(\T) = (Id - \Pi_N)  \L^2(\T)
\end{align*}

\begin{prop}\label{structure of the truncated flow}
\begin{enumerate}[leftmargin=0cm] We have the following properties ;
      \item The truncated flow $\Phi_N(t)$ commute with the frequency projector $\Pi_N$, that is
    \begin{equation*}
     \Phi_N(t) \circ \Pi_N = \Pi_N \circ \Phi_N(t)
     \end{equation*}
    As a consequence, the truncated flow $\Phi_N(t)$ maps the finite-dimensional space $E_N$ to itself. Moreover, the induced map 
    \begin{equation*}
        \begin{split}
        \tld{\Phi}_N(t) : \hsp & E_N \lra E_N \\
                          & u_0 \longmapsto \Phi_N(t)(u_0)
        \end{split}
    \end{equation*}
    is the flow of the ODE
    \begin{equation}\label{FNLS}
    \begin{cases}
        i\p_t u + \p_x^2 u = \Pi_N \left(|u|^4u \right) \\
        u|_{t=0} = u_0 \in E_N
    \end{cases}
    \tag{FNLS}
\end{equation}
which can be seen as the finite-dimensional Hamiltonian equation on $E_N$ : 
\begin{equation*}
    \begin{cases}
        i\p_t u = \frac{\p H_N}{\p \cjg{u}}(u) \\
        u|_{t=0} = u_0 \in E_N
    \end{cases}
\end{equation*}
with Hamiltonian $H_N(u) := \frac{1}{2} \norm{\p_x^2u}^2_{\L^2(\T)} + \frac{1}{6} \norm{u}_{\L^6(\T)}^6$, for $u \in E_N$.

\item The truncated flow $\Phi_N(t)$ commute with $\Pi_N^{\perp}$, that is
    \begin{equation*}
     \Phi_N(t) \circ \Pi_N^{\perp} = \Pi_N^{\perp} \circ \Phi_N(t)
     \end{equation*}
    As a consequence, the truncated flow $\Phi_N(t)$ maps the space $E_N^{\perp}$ to itself. Moreover, the induced map 
    \begin{equation*}
        \begin{split}
        \Phi^{\perp}_N(t) : \hsp & E_N^{\perp} \lra E_N^{\perp} \\
                          & u_0 \longmapsto \Phi_N(t)(u_0)
        \end{split}
    \end{equation*}
    is the solution of the linear Schrödinger equation 
\begin{equation*}
    \begin{cases}
        i\p_t u + \p_x^2 u = 0 \\
        u|_{t=0} = u_0 \in E_N^{\perp}
    \end{cases}
\end{equation*}
Hence, $\Phi^{\perp}_N(t)$ coincide with the linear operator $e^{it\p_x^2}$ on $E_N^{\perp}$.

\item  The truncated flow $\Phi_N(t)$ can be factorized as $(\tld{\Phi}_N(t),e^{it\p_x^2})$ on $E_N \times E_N^{\perp}$. In other words, for every $u_0 \in \hsigt$,
 \begin{equation*}
     \Phi_N(t)(u_0) = \tld{\Phi}_N(t) \Pi_N u_0  +  e^{it\p_x^2} \Pi_N^{\perp}u_0
 \end{equation*}
\end{enumerate}
\end{prop}
\bigskip

Let us now prove these three points.

\begin{proof} Let $u_0 \in L^2(\T)$.
    \begin{enumerate}[leftmargin=0cm]
        \item \textbullet \hsp $\Phi_N(t)$ is the flow of \eqref{truncated equation}. And, when we apply $\Pi_N$ to equation \eqref{truncated equation}, we see that $\Pi_N \Phi_N(t) u_0$ is the solution of the equation 
        \begin{equation*}
            \begin{cases}
                 i\p_t u + \p_x^2 u  = \Pi_N \left(|\Pi_Nu|^4\Pi_Nu \right) =\Pi_N \left(|u|^4u \right) \\
                 u|_{t=0} = \Pi_N u_0 
            \end{cases}
        \end{equation*}
        On the other hand, from the definition of the flow of \eqref{truncated equation}, the solution of the equation above is none other than $\Phi_N(t)\Pi_N u_0$ itself. This means that 
        \begin{equation*}
            \Pi_N \Phi_N(t) u_0=\Phi_N(t)\Pi_N u_0
        \end{equation*}
\textbullet \hsp Next, we show that \eqref{FNLS} can be seen as a finite-dimensional Hamiltonian equation on $E_N$. Every element $u \in E_N$ can be decompose as 
\begin{equation*}
    u = v + i w
\end{equation*}
where 
\begin{equation*}
    \begin{split}
        v &= \sum_{|k|\leq N} \Re(u_k)e^{ikx} =: \sum_{|k|\leq N} v_ke^{ikx} \in E_N \\
        w &= \sum_{|k|\leq N}\Im(u_k) e^{ikx} =: \sum_{|k|\leq N} w_ke^{ikx} \in E_N
    \end{split}
\end{equation*}
Furthermore, we invoke the operators 
\begin{equation*}
    \begin{split}
    &  \frac{\p}{\p \cjg{u}_k} := \frac{\p}{\p v_k} + i \frac{\p}{\p w_k} \\
    & \frac{\p}{\p \cjg{u}} := \frac{\p}{\p v} + i \frac{\p}{\p w} =:  \sum_{|k|\leq N}  e^{ikx}  \frac{\p}{\p v_k} + i \sum_{|k|\leq N}  e^{ikx}  \frac{\p}{\p w_k}
    \end{split}
\end{equation*}
and also the function 
\begin{equation*}
    \begin{split}
         H_N : \hsp & E_N \lra \R \\
            & u \longmapsto \frac{1}{2} \norm{\p_x^2u}^2_{\L^2(\T)} + \frac{1}{6} \norm{u}_{\L^6(\T)}^6
    \end{split}
\end{equation*} 
By performing elementary computations, we see that we can rewrite \eqref{FNLS} as 
\begin{equation*}
        \begin{cases}
        i\p_t u= \frac{\p H_N}{\p \cjg{u}}(u) \\
        u|_{t=0} = u_0 
    \end{cases}
    \tag{FNLS}
\end{equation*}
This means that \eqref{FNLS} is an Hamiltonian equation on $E_N$ with associated Hamiltonian $H_N$.

\bigskip

        \item Once again, applying $\Pi_N^{\perp}$ to equation \eqref{truncated equation}, we see that $\Pi_N^{\perp} \Phi_N(t)u_0$ is the solution of the equation 
        \begin{equation*}
            \begin{cases}
                i\p_t u + \p_x^2 u = 0 \\
        u|_{t=0} = \Pi_N^{\perp} u_0 
            \end{cases}
        \end{equation*}
        This means that 
        \begin{equation}\label{free Schro on hf}
            \Pi_N^{\perp} \Phi_N(t)u_0 = e^{it\p_x^2}\Pi_N^{\perp}u_0
        \end{equation}
        Furthermore, $\Pi_N \Phi_N(t) \Pi_N^{\perp}u_0$ is the solution of the equation 
        \begin{equation*}
            \begin{cases}
                 i\p_t u + \p_x^2 u  = \Pi_N \left(|\Pi_Nu|^4\Pi_Nu \right) \\
                 u|_{t=0} = 0
            \end{cases}
        \end{equation*}
        Thus, $\Pi_N \Phi_N(t) \Pi_N^{\perp}u_0$ is none other than $0$, and we obtain from \eqref{free Schro on hf} that
        \begin{equation*}
            \Phi_N(t) \Pi_N^{\perp}u_0 = \Pi_N^{\perp} \Phi_N(t) \Pi_N^{\perp}u_0 = e^{it\p_x^2}\Pi_N^{\perp}u_0 = \Pi_N^{\perp} \Phi_N(t)u_0
        \end{equation*}
        \item As a consequence of the two previous points, we can write 
        \begin{equation*}
            \begin{split}
                \Phi_N(t)u_0 &= \Pi_N \Phi_N(t)u_0 + \Pi_N^{\perp}\Phi_N(t)u_0 \\
                & = \Phi_N(t)\Pi_N u_0 + \Phi_N(t)\Pi_N^{\perp}u_0 \\
                & = \underbrace{\tld{\Phi}_N(t) \Pi_N u_0}_{\in E_N} + \underbrace{e^{it\p_x^2} \Pi_N^{\perp}u_0}_{\in E_N^{\perp}}
            \end{split}
        \end{equation*}
    \end{enumerate}
\end{proof}

\bibliographystyle{siam}
\bibliography{refs}

\begin{thebibliography}{10}

\bibitem{bernier2024dynamics}
{\sc J.~Bernier, B.~Grébert, and T.~Robert}, {\em Dynamics of quintic nonlinear schr{\"o}dinger equations in $h^{2/5+}(\mathbb{T})$}, arXiv preprint arXiv:2305.05236,  (2024).

\bibitem{bogachev1998gaussian}
{\sc V.~Bogachev}, {\em Gaussian Measures}, Mathematical surveys and monographs, American Mathematical Society, 1998.

\bibitem{Bourgain1993}
{\sc J.~Bourgain}, {\em Fourier transform restriction phenomena for certain lattice subsets and applications to nonlinear evolution equations. {I}. {S}chr\"{o}dinger equations}, Geom. Funct. Anal., 3 (1993), pp.~107--156.

\bibitem{bourgain1994periodic}
\leavevmode\vrule height 2pt depth -1.6pt width 23pt, {\em Periodic nonlinear {S}chr\"{o}dinger equation and invariant measures}, Comm. Math. Phys., 166 (1994), pp.~1--26.

\bibitem{bourgain2004remark}
{\sc J.~Bourgain}, {\em A remark on normal forms and the ``{$I$}-method'' for periodic {NLS}}, J. Anal. Math., 94 (2004), pp.~125--157.

\bibitem{burq2024almost}
{\sc N.~Burq and L.~Thomann}, {\em Almost {S}ure {S}cattering for the {O}ne {D}imensional {N}onlinear {S}chr\"{o}dinger {E}quation}, Mem. Amer. Math. Soc., 296 (2024).

\bibitem{burq2013probabilistic}
{\sc N.~Burq and N.~Tzvetkov}, {\em Probabilistic well-posedness for the cubic wave equation}, J. Eur. Math. Soc. (JEMS), 16 (2014), pp.~1--30.

\bibitem{coe2024sharp}
{\sc J.~Coe and L.~Tolomeo}, {\em Sharp quasi-invariance threshold for the cubic szeg\"{o} equation}, arXiv preprint arXiv:2404.14950,  (2024).

\bibitem{debussche2021quasi}
{\sc A.~Debussche and Y.~Tsutsumi}, {\em Quasi-invariance of {G}aussian measures transported by the cubic {NLS} with third-order dispersion on {$\mathbb{T}$}}, J. Funct. Anal., 281 (2021), pp.~Paper No. 109032, 23.

\bibitem{forlano_and_soeng2022transport}
{\sc J.~Forlano and K.~Seong}, {\em Transport of {G}aussian measures under the flow of one-dimensional fractional nonlinear {S}chr\"{o}dinger equations}, Comm. Partial Differential Equations, 47 (2022), pp.~1296--1337.

\bibitem{forlano2022quasi}
{\sc J.~Forlano and L.~Tolomeo}, {\em Quasi-invariance of gaussian measures of negative regularity for fractional nonlinear schr\" odinger equations}, arXiv preprint arXiv:2205.11453,  (2022).

\bibitem{forlano_and_trenberth2019transport}
{\sc J.~Forlano and W.~J. Trenberth}, {\em On the transport of {G}aussian measures under the one-dimensional fractional nonlinear {S}chr\"{o}dinger equations}, Ann. Inst. H. Poincar\'{e} C Anal. Non Lin\'{e}aire, 36 (2019), pp.~1987--2025.

\bibitem{genovese2022quasi}
{\sc G.~Genovese, R.~Luc\`a, and N.~Tzvetkov}, {\em Quasi-invariance of low regularity {G}aussian measures under the gauge map of the periodic derivative {NLS}}, J. Funct. Anal., 282 (2022), pp.~Paper No. 109263, 45.

\bibitem{genovese2023quasi}
\leavevmode\vrule height 2pt depth -1.6pt width 23pt, {\em Quasi-invariance of {G}aussian measures for the periodic {B}enjamin-{O}no-{BBM} equation}, Stoch. Partial Differ. Equ. Anal. Comput., 11 (2023), pp.~651--684.

\bibitem{genovese2023transport}
\leavevmode\vrule height 2pt depth -1.6pt width 23pt, {\em Transport of {G}aussian measures with exponential cut-off for {H}amiltonian {PDE}s}, J. Anal. Math., 150 (2023), pp.~737--787.

\bibitem{gunaratnam2022quasi}
{\sc T.~Gunaratnam, T.~Oh, N.~Tzvetkov, and H.~Weber}, {\em Quasi-invariant {G}aussian measures for the nonlinear wave equation in three dimensions}, Probab. Math. Phys., 3 (2022), pp.~343--379.

\bibitem{kuo2006gaussian}
{\sc H.-H. Kuo}, {\em Gaussian measures in banach spaces}, Gaussian Measures in Banach Spaces,  (2006), pp.~1--109.

\bibitem{Li_Wu_Xu_global}
{\sc Y.~Li, Y.~Wu, and G.~Xu}, {\em Global well-posedness for the mass-critical nonlinear {S}chr\"{o}dinger equation on {$\mathbb{T}$}}, J. Differential Equations, 250 (2011), pp.~2715--2736.

\bibitem{mcconnell_nonlin_smoothing}
{\sc R.~McConnell}, {\em Nonlinear smoothing for the periodic generalized nonlinear {S}chr\"{o}dinger equation}, J. Differential Equations, 341 (2022), pp.~353--379.

\bibitem{oh_soeng2021quasi}
{\sc T.~Oh and K.~Seong}, {\em Quasi-invariant {G}aussian measures for the cubic fourth order nonlinear {S}chr\"{o}dinger equation in negative {S}obolev spaces}, J. Funct. Anal., 281 (2021), pp.~Paper No. 109150, 49.

\bibitem{oh2018optimal}
{\sc T.~Oh, P.~Sosoe, and N.~Tzvetkov}, {\em An optimal regularity result on the quasi-invariant {G}aussian measures for the cubic fourth order nonlinear {S}chr\"{o}dinger equation}, J. \'{E}c. polytech. Math., 5 (2018), pp.~793--841.

\bibitem{oh2019quasi}
{\sc T.~Oh, Y.~Tsutsumi, and N.~Tzvetkov}, {\em Quasi-invariant {G}aussian measures for the cubic nonlinear {S}chr\"{o}dinger equation with third-order dispersion}, C. R. Math. Acad. Sci. Paris, 357 (2019), pp.~366--381.

\bibitem{oh2017quasi}
{\sc T.~Oh and N.~Tzvetkov}, {\em Quasi-invariant {G}aussian measures for the cubic fourth order nonlinear {S}chr\"{o}dinger equation}, Probab. Theory Related Fields, 169 (2017), pp.~1121--1168.

\bibitem{oh2020quasi}
\leavevmode\vrule height 2pt depth -1.6pt width 23pt, {\em Quasi-invariant {G}aussian measures for the two-dimensional defocusing cubic nonlinear wave equation}, J. Eur. Math. Soc. (JEMS), 22 (2020), pp.~1785--1826.

\bibitem{planchon2020transport}
{\sc F.~Planchon, N.~Tzvetkov, and N.~Visciglia}, {\em Transport of {G}aussian measures by the flow of the nonlinear {S}chr\"{o}dinger equation}, Math. Ann., 378 (2020), pp.~389--423.

\bibitem{planchon2022modified}
\leavevmode\vrule height 2pt depth -1.6pt width 23pt, {\em Modified energies for the periodic generalized {K}d{V} equation and applications}, Ann. Inst. H. Poincar\'{e} C Anal. Non Lin\'{e}aire, 40 (2023), pp.~863--917.

\bibitem{Simon+1974}
{\sc B.~Simon}, {\em The {$P(\phi )\sb{2}$} {E}uclidean (quantum) field theory}, Princeton Series in Physics, Princeton University Press, Princeton, NJ, 1974.

\bibitem{sosoe2020quasi}
{\sc P.~Sosoe, W.~J. Trenberth, and T.~Xian}, {\em Quasi-invariance of fractional {G}aussian fields by the nonlinear wave equation with polynomial nonlinearity}, Differential Integral Equations, 33 (2020), pp.~393--430.

\bibitem{sun2023quasi}
{\sc C.~Sun and N.~Tzvetkov}, {\em Quasi-invariance of gaussian measures for the $3 d $ energy critical nonlinear schr\" odinger equation}, arXiv preprint arXiv:2308.12758,  (2023).

\bibitem{thomann2010gibbs}
{\sc L.~Thomann and N.~Tzvetkov}, {\em Gibbs measure for the periodic derivative nonlinear {S}chr\"{o}dinger equation}, Nonlinearity, 23 (2010), pp.~2771--2791.

\bibitem{tzvetkov2008invariant}
{\sc N.~Tzvetkov}, {\em Invariant measures for the defocusing nonlinear {S}chr\"{o}dinger equation}, Ann. Inst. Fourier (Grenoble), 58 (2008), pp.~2543--2604.

\bibitem{tzvetkov2010construction}
{\sc N.~Tzvetkov}, {\em Construction of a {G}ibbs measure associated to the periodic {B}enjamin-{O}no equation}, Probab. Theory Related Fields, 146 (2010), pp.~481--514.

\bibitem{tzvetkov2015quasiinvariant}
{\sc N.~Tzvetkov}, {\em Quasiinvariant {G}aussian measures for one-dimensional {H}amiltonian partial differential equations}, Forum Math. Sigma, 3 (2015), pp.~Paper No. e28, 35.

\bibitem{tzvetkov2013gaussian}
{\sc N.~Tzvetkov and N.~Visciglia}, {\em Gaussian measures associated to the higher order conservation laws of the {B}enjamin-{O}no equation}, Ann. Sci. \'{E}c. Norm. Sup\'{e}r. (4), 46 (2013), pp.~249--299.

\end{thebibliography}

\end{document}